\documentclass{amsart}
\usepackage{color}
\usepackage{ulem}

\newtheorem{theorem}{Theorem}[section]
\newtheorem{lemma}[theorem]{Lemma}

\theoremstyle{definition}
\newtheorem{definition}[theorem]{Definition}
\newtheorem{example}[theorem]{Example}

\theoremstyle{remark}
\newtheorem{remark}[theorem]{Remark}

\numberwithin{equation}{section}

\newcommand{\be}{\begin{equation}}
\newcommand{\ee}{\end{equation}}
\newcommand{\argone}{(q^{1/2}z+q^{-1/2}z^{-1})/2}

\newcommand{\awNf}[2]{\widetilde{N}_\mathrm{AW}\left({#1},\,{#2}\right)}
\newcommand{\awnf}[2]{\tilde{n}_\mathrm{AW}\left({#1},\,{#2}\right)}
\newcommand{\D}{{\mathcal{D}_q}}
\begin{document}

\title[Nevanlinna theory \& Askey-Wilson operator]{Nevanlinna theory of the Askey-Wilson divided difference operator}

\author{Yik-Man Chiang}
\address{Department of Mathematics, Hong Kong University of Science and
Technology, Clear Water Bay, Kowloon, Hong Kong,\ China.}
\email{machiang@ust.hk}
\thanks{This research was supported in part by the Research Grants Council of the Hong Kong Special Administrative Region, China
(600806, 16306315). The second author was also partially supported by National Natural Science
Foundation of China (Grant No. 11271352) and by the HKUST PDF
Matching Fund.}

\author{Shao-Ji Feng}
\address{Academy of Mathematics and Systems Science, Chinese Academy
of Sciences, Beijing, 100080, P. R. China.}
\email{fsj@amss.ac.cn}
%    General info

%\date{25 Dec. 2007--07 February 2015; version 6.1}

%\dedicatory{This paper is dedicated to our authors.}


\begin{abstract}
This paper establishes a version of Nevanlinna theory based on Askey-Wilson divided difference operator for meromorphic functions of finite logarithmic order in the complex plane $\mathbb{C}$. A second main theorem that we have derived allows us to define an Askey-Wilson type Nevanlinna deficiency which gives a new interpretation 
that one should regard many important infinite products arising from the study of basic hypergeometric series as zero/pole-scarce. That is, their zeros/poles are indeed deficient in the sense of difference Nevanlinna theory. A natural consequence is a version of Askey-Wilosn type Picard theorem. We also give an alternative and self-contained characterisation of the kernel functions of the Askey-Wilson operator.  In addition we have established a version of unicity theorem in the sense of Askey-Wilson. This paper concludes with an application to difference equations generalising the Askey-Wilson second-order divided difference equation.
\end{abstract}

\maketitle
\tableofcontents

\section{Introduction}
\label{}
Without loss of generality, we assume $q$ to be a complex number with $|q|<1$. Askey and Wilson evaluated a $q-$beta integral (\cite[Theorem 2.1]{Askey:Wilson1985}) that allows them to construct a family of orthogonal polynomials (\cite[Theorems 2.2--2.5]{Askey:Wilson1985}) which are eigen-solutions of a second order difference equation (\cite[\S 5]{Askey:Wilson1985}) now bears their names. The divided difference operator $\mathcal{D}_q$ that appears in the second-order difference equation is called \textit{Askey-Wilson operator}. These polynomials, their orthogonality weight, the difference operator and related topics have found numerous applications and connections with a wide range of research areas beyond the basic hypergeometric series. These research areas include, for examples, Fourier analysis (\cite{Brown:Ismail1995}), interpolations (\cite{Magnus2009}, \cite{Ismail:Stanton2003_a}), combinatorics (\cite{Corteel:Williams2010}), Markov process (\cite{Bryc:Wesolowski2010}, \cite{Szablowski2011}), quantum groups (\cite{Koelink:Stokman2001b}, \cite{Noumi:Stokman2004}), double affine Hecke (Cherednik) algebras (\cite{Cherednik2001}, \cite{Koelink:Stokman2001a}).

In this paper, we show, building on the strengths of the work of Halburd and Korhonen \cite{HK-1}, \cite{HK-2} and as well as our earlier work on logarithmic difference estimates (\cite{Chiang:Feng2008}, \cite{Chiang:Feng2009}),  that there is a very natural function theoretic interpretation of the Askey-Wilson operator (abbreviated as $\mathrm{AW}-$operator) $\mathcal{D}_q$ and related topics. It is not difficult to show that the $\mathrm{AW}-$operator is well-defined on meromorphic functions. In particular, we show that there is a Picard theorem associates with the Askey-Wilson operator just as the classical Picard theorem is associated with the conventional differential operator $f^\prime$. Moreover, we have obtained a 
full-fledged Nevanlinna theory for slow-growing meromorphic functions with respect to the $\mathrm{AW}-$operator on $\mathbb{C}$ for which the associated Picard theorem follows as a special case, just as the  classical Picard theorem is a simple consequence of the classical Nevanlinna theory (\cite{Nevanlinna1925}, see also \cite{Nev70}, \cite{Hayman1964} and \cite{Yang1993}). This approach allows us to gain new insights into the $\mathcal{D}_q$ and
that give a radically different viewpoint from the established views on the value distribution properties of certain meromorphic functions, such as the Jacobi theta -functions, generating functions of certain orthogonal polynomials that were used in L. J. Rogers' derivation of the two famous Rogers-Ramanujan identities \cite{Rogers1895}, etc. We also characterise  the functions that lie in the kernel of the Askey-Wilson operator, which we can regard as the \textit{constants} with respect to the $\mathrm{AW}-$operator.

A value $a$ which is not assumed by a meromorphic function $f$ is called a \textit{Picard (exceptional) value}. The Picard theorem states that if a meromorphic $f$ that has three Picard values, then $f$ necessarily reduces to a constant. For each complex number $a$, Nevanlinna defines a deficiency $0\le \delta(a)\le 1$. If  $\delta(a)\sim 1$, then that means $f$ rarely assumes $a$. In fact, if $a$ is a Picard value of $f$, then $\delta (a)=1$. If $f$ assumes $a$ frequently, then $\delta(a)\sim 0$. Nevanlinna's second fundamental theorem implies that $\sum_{a\in\mathbb{C}}\delta(a)\le 2$ for a non-constant meromorphic function. Thus, the Picard theorem follows easily. 
For each $a\in\mathbb{C}$, we formulate a $q-$deformation of the Nevanlinna deficiency $\Theta_{\textrm{AW}}(a)$ and Picard value which we call $\textrm{AW}-$\textit{deficiency} and $\textrm{AW}-$\textit{Picard value} respectively. Their definitions will be given in \S\ref{S:deficient}.
The $\textrm{AW-}$deficiency  also satisfies the inequalities $0\le  \Theta_{\textrm{AW}}(a)\le 1$.  

A very special but illustrative example for $a\in\mathbb{C}$ to be an $\textrm{AW}-$Picard value of a certain $f$ if the pre-image of $a\in\mathbb{C}$ assumes the form, with some $z_a\in\mathbb{C}$,
	\begin{equation}
		\label{E:AW-except}
		x_n:=\frac12\big(z_a \,q^{n}+q^{-n}/z_a \big),\ 
% \stackrel{f}{\longmapsto}\ a, 
\qquad n\in\mathbb{N}\cup\{0\}. 
	\end{equation}
This leads to $\Theta_{\textrm{AW}}(a)=1$.

 We illustrate some such  $\textrm{AW}-$Picard values in the following examples from the viewpoint with our new interpretation. Let us first introduce some notation.

We define the $q-$shifted factorials:
\begin{equation}
\label{E:q-factorial}
	(a;\, q)_0:=1,\qquad (a;\, q)_n:=\prod_{k=1}^n(1-aq^{k-1}),
	\quad n=1,\, 2,\, \cdots, %\ \textrm{or}\ \infty,
\end{equation}
and the multiple $q-$shifted factorials:
\begin{equation}
\label{E:multiple-factorial}
	(a_1,\, a_2,\,\cdots, a_k;\, q)_n:= \prod_{j=1}^k
	(a_j;\, q)_n.
\end{equation}

Thus, the infinite product
	\[
		(a_1,\, a_2,\,\cdots, a_k;\, q)_\infty=\lim_{n\to+\infty} (a_1,\, a_2,\,\cdots, a_k;\, q)_n
	\]
always converge since $|q|<1$.

The infinite products that appear in the \textit{Jacobi triple-product} formula (\cite[p. 15]{GR2004})
	\begin{equation}
		\label{E:triple}
		f(x)=(q;\; q)_\infty(q^{1/2}z,\, q^{1/2}/z;\,q)_\infty
		=\sum_{k=-\infty}^\infty(-1)^kq^{k^2/2} z^k,
	\end{equation}
can be considered as a function of $x$ where $x=\frac12(z+z^{-1})$. The corresponding (zero) sequence is given by
	\[
		 x_n:=\frac12\big(q^{1/2+n}+q^{-1/2-n}\big),\  
%\stackrel{}{\longmapsto}\  0
 \qquad n\in\mathbb{N}\cup\{0\}
	\]
where $z_{a}=(q^{1/2}+q^{-1/2})/2\ (a=0)$ . Thus $0$ is an $\mathrm{AW-}$Picard value of $f$ when viewed as a function of $x$, and hence $f$ has $\Theta_{\textrm{AW}}(0)=1$.

 Our next example is a generating function for a class of orthogonal polynomials  known as \textit{continuous $q-$Hermite polynomials} first derived by Rogers in 1895 \cite{Rogers1895}
	\[
		f(x)=\frac{1}{(t e^{i\theta},\, t e^{-i\theta};\, q)_\infty}
		=\sum_{k=0}^\infty \frac{H_k(x\,|\, q)}{(q;\, q)_k}\, t^k,\qquad 0<|t|<1,
	\]
where
	\[
		H_n(x\,|\, q)=\sum_{k=0}^n\frac{(q;\, q)_n}{(q;\, q)_{k}\,(q;\, q)_{n-k}}\, e^{i(n-2k)\theta},\quad x=\cos\theta.
	\]
 The orthogonality of these polynomials were worked out by Askey and Ismail \cite[1983]{Askey:Ismail1983}. 
We easily verify that  the $\infty$ is an $\mathrm{AW-}$Picard value of $f$ when viewed as a functions of $x$ with the pole-sequence given by
	\begin{equation}
		\label{E:Rogers-poles}
		x_n:=\frac12\big(t\,q^{n}+q^{-n}/t\big),
%\ \stackrel{}{\longmapsto}\infty, 
\qquad n\in\mathbb{N}\cup\{0\},
\end{equation}
where $z_a=(t+t^{-1})/2\ (a=\infty)$. This implies $\Theta_{\textrm{AW}}(\infty)=1$.

Our third example has both zeros and poles. It is again a generating function for a more general class of orthogonal polynomials also derived by Rogers in 1895 \cite{Rogers1895}. That is, 
	\begin{equation}
		\label{E:generating}
		H(x):=\frac{(\beta e^{i\theta}t,\, \beta e^{-i\theta}t;\, q)_\infty}
		{(e^{i\theta}t,\, e^{-i\theta}t;\, q)_\infty}
		=\sum_{n=0}^\infty {C_n(x;\, \beta\, |\, q)}\, t^n,\quad 
		x=\cos\theta,
	\end{equation}
where 
%	\begin{align}
	\[
		\begin{split}
		C_n(x;\, \beta\, |\, q) &= \sum_{k=0}^n
		\frac{(\beta;\, q)_k(\beta;\, q)_{n-k}}{(q;\, q)_k(q;\, q)_{n-k}}\cos(n-2k)\theta\\
		&= \sum_{k=0}^n
		\frac{(\beta;\, q)_k(\beta;\, q)_{n-k}}{(q;\, q)_k(q;\, q)_{n-k}}\, T_{n-2k}(x)
		\end{split}
	\]
%	\end{align}
 is called \textit{continuous $q-$ultraspherical polynomials} by Askey and Ismail \cite{Askey:Ismail1983}. Here the $T_n(x)$ denotes the $n-$th Chebychev polynomial of the first kind. Rogers \cite{Rogers1895} used these polynomials to derive the two celebrated Rogers-Ramanujan identities
	\[
		\sum_{n=0}^\infty \frac{q^{n^2}}{(q;\, q)_n}=\frac{1}{(q;\, q^5)_\infty(q^4;\, q^5)_\infty}
	\quad\textrm{and}\quad
		\sum_{n=0}^\infty \frac{q^{n^2+n}}{(q;\, q)_n}=\frac{1}{(q^2;\, q^5)_\infty(q^3;\, q^5)_\infty}.
	\]
One can find a thorough discussion about the derivation of these identities in Andrews \cite[\S 2.5]{Andrews1986}.

%\noindent In particular, we have
%		\[
%			C_n(\cos\theta;\, q\,| q)=\sum_{k=0}^n\cos(n-2k)\theta =\frac{\sin(n+1)\theta}{\sin\theta},
%		\]
%and when $\beta=q^\lambda$ and

%		\[
%			\lim_{q \to 1} C_n(x;\, q^\lambda\,|\, q)=C_n^\lambda(x)
%		\]
%\smallskip

%\noindent where $C_n^\lambda(x)$ is the  $n$th-Gegenbauer (or ultraspherical) polynomials. 
The zero- and pole-sequences of $H(x)$ in the $x-$plane are given, respectively, by 
\begin{equation}
		\label{E:Rogers-zeros}
		x_n:=\frac12\big(\beta t\,q^{n}+q^{-n}/(\beta t)\big),
%\ \stackrel{}{\longmapsto} 0, 
\qquad n\in\mathbb{N}\cup\{0\}. 
	\end{equation}
 and (\ref{E:Rogers-poles}). %Clearly $H_n(x\, |\,q)=(q;\, q)_nC_n(x;\, 0\,|q)$. 
The point is that we have both $0$ and $\infty$ to be the $\textrm{AW}-$Picard values according to our interpretation. Thus $\Theta_{\mathrm{AW}}(0)=1$ and $\Theta_{\mathrm{AW}}(\infty)=1$ for the generating function $H(x)$.

Our Askey-Wilson version of Nevanlinna's second fundamental theorem (Theorem \ref{T:2nd-Main-2}) for slow-growing meromorphic functions not belonging to the kernel of $\D$ also implies that 	
\begin{equation}
		\label{E:AW-Nevanlinna-sum}
		\sum_{a\in\mathbb{C}}\Theta_{\textrm{AW}}(a)\le 2.
	\end{equation}
This new relation allows us to deduce a $\mathrm{AW}-$Picard theorem (Theorem \ref{T:ker-mero}): 
	\textit{Suppose a slow-growing meromorphic function $f$ has three values $a,\, b,\, c \in\mathbb{C}$ such that $\Theta_{\mathrm{AW}}(a)=\Theta_{\mathrm{AW}}(b)=\Theta_{\mathrm{AW}}(c)=1$. Then  $f$ lies in the kernel of $\mathcal{D}_q$.}

Note that, what Nevanlinna proved can be viewed when a meromorphic function has three Picard values then the function lies in the kernel of differential operator. 
 
By the celebrated Jacobi triple-product formula \cite[p. 497]{Andrews:Askey:Roy1999}, we can write the Jacobi theta-function $\vartheta_4(z,\, q)= 1+2\sum_{k=1}^\infty (-1)^n q^{k^2} \, \cos 2k z$  in the infinite product form
	\[
		\vartheta_4(z,\, q)=(q^2;\, q^2)_\infty\, (q\,e^{2iz},\,q\, e^{-2iz},\, q^2)_\infty
	\]
implying that it too has $\Theta_{\mathrm{AW}}(0)=1$ when viewed as a function $f(x)$ of $x$. Since the $f(x)$ is entire, so that the relationship (\ref{E:AW-Nevanlinna-sum}) becomes
	\[
		1=\Theta_{\mathrm{AW}}(0)\le \sum_{a\in\mathbb{C}}\Theta_{\textrm{AW}}(a)\le 1.
	\]
We deduce from this inequality that there could not be a non-zero $a$ such that the theta function have $f(x_n)=a$ only on a  sequence $\{x_n\} $ of the form (\ref{E:AW-except}). Otherwise, it would follow from Theorem \ref{T:AW-Picard} that the theta function $\vartheta_4$ would belong to the kernel $\ker \D$, contradicting the kernel functions representation that we shall discuss in the next paragraph. The same applies to the remaining three Jacobi theta-functions. Intuitively speaking, the more zeros the function has out of the \textit{maximal allowable number} of zeros of the meromorphic function can have implies the \textit{larger} the AW-Nevanlinna deficiency $\Theta_{\mathrm{AW}}(0)$. That is, the function assumes $x=0$ less often in the $\mathrm{AW-}$sense, even though the function actually assumes $x=0$ more often in the conventional sense. Thus, since the theta function assumes zero \textit{maximally}, so it misses $x=0$ also maximally in the $\mathrm{AW-}$sense. The following examples that we shall study in in details in \S\ref{S:examples} show how zeros are missed/assumed in proportion to the maximally allowed number of zeros against their $\mathrm{AW-}$deficiencies. That is,  for a given integer $n$,

	\[
			f_{\frac{n-1}{n}}
(x)=\prod_{k=0}^{n-1} (q^ke^{i\theta},\, q^k e^{-i\theta};\,q^{n+1})_\infty,
\qquad \Theta_{\mathrm{AW}}(0)=\frac{n-1}{n},
	\]
and
		\[
			f_{\frac1n}(x)=
		\prod_{k=0}^{n-1} (q^{2k}e^{i\theta},\, q^{2k}e^{-i\theta};\, q^{2n-1})_\infty,\qquad \Theta_{\mathrm{AW}}(0)=\frac1n.
		\] 
In \cite[p. 365]{Ism2005} Ismail has given an example of meromorphic function that belongs to $\ker\mathcal{D}_q$:
\begin{equation}
	\label{E:kernel}
	f(x)=\frac{(\cos\theta-\cos\phi)
\big(qe^{i(\theta+\phi)},\, qe^{-i(\theta+\phi)};\,q)_\infty(
 qe^{i(\theta-\phi)},\, qe^{-i(\theta-\phi)};\, q\big)_\infty}
{\big(q^{1/2}e^{i(\theta+\phi)},\, q^{1/2}e^{-i(\theta+\phi)};\,q)_\infty(
q^{1/2}e^{i(\theta-\phi)},\, q^{1/2}e^{-i(\theta-\phi)};\, q\big)_\infty},
\end{equation}
for a fixed $\phi$. Let $f$ belong to the kernel of $\mathcal{D}_q$, that is $\mathcal{D}_qf\equiv 0$. Then one can readily deduce from \eqref{E:AW-operator-2} that 
the $f$, when viewed upon as a function of $\theta$, is doubly periodic, and hence must be an elliptic function in $\theta$\footnote{The authors are grateful for the referee who pointed out this fact.}. However, the authors could not find an explicit discussion of this observation in the literature. Here we offer an alternative and self-contained derivation to characterise these kernel functions when viewed as a function of $x=\cos \theta$. Our Theorem \ref{T:ker-mero} shows that all functions in the $\ker\mathcal{D}_q$ are essentially a product of functions of this form. Intuitively speaking, the functions that lie in the kernel must have zero- and pole-sequences described by (\ref{E:AW-except}). We utilise the linear structure of the $\ker \mathcal{D}_q$ to deduce any given number of 
 linear combination of $q-$infinite products can again be expressed in terms of a single $q-$infinite product of the same form whose zero- and pole-sequences can again be described by (\ref{E:AW-except}) (see Theorems \ref{T:Kernel-I} and \ref{T:Kernel-II}). Many important Jacobi theta-function identities such as the following well-known  (see \cite{Whittaker:Watson1927})
	\begin{equation}
		\label{E:theta-identity-1}		\vartheta_4^2(z)\,\vartheta_4^2+\vartheta_2^2(z)\,\vartheta_2^2=\vartheta_3^2(z)\,\vartheta_3^2
	\end{equation}
and
	\begin{equation}
		\label{E:theta-identity-2}
	\vartheta_3(z+y)\,\vartheta_3(z-y)\, \vartheta_2^2=
		\vartheta_3^2(y)\, \vartheta_3^2(z) +\vartheta_1^2(y)\,\vartheta_1^2(z)
	\end{equation}
 are of the forms described amongst our Theorems \ref{T:Kernel-I} and \ref{T:Kernel-II}.

The key to establishing a $q-$deformation of the classical Nevanlinna second fundamental theorem is based on our $\mathrm{AW-}$logarithmic difference estimate  of the proximity function for the meromorphic function $f$:
	\begin{equation}
\label{E:log-est}
	m\Big(r,\,\frac{(\mathcal{D}_qf)(x)}{f(x)}\Big)=O\big((\log r)^{\sigma_{\log}-1+\varepsilon}\big)
	\end{equation}
holds for \textit{all} $x=r$ sufficiently large, where $\sigma_{\log}$ is the logarithmic order of $f$ (Theorem \ref{T:log-lemma}).  This estimate is the key of our argument to establishing our AW-Nevanlinna theory (see also \cite{Nev70}, \cite{HK-1}, \cite{HK-2}).
We have also obtained a corresponding \textrm{pointwise} estimate (Theorem \ref{T:pointwise}) outside a set of finite logarithmic measure.

Similar estimates for the simple difference operator $\Delta f(x)=f(x+\eta)-f(x)$ for a fixed $\eta\not=0$, were obtained by Chiang and Feng in \cite{Chiang:Feng2008}, \cite{Chiang:Feng2009}, and by Halburd and Korhonen \cite{HK-1} independently in a slightly different form for finite-order meromorphic functions. Halburd et al showed the same for the case of $q-$difference operator $\Delta_qf(x)=f(qx)-f(x)$ in \cite{HK-4} for zero-order meromorphic functions, whilst Cheng and Chiang \cite{Cheng:Chiang2017} showed that a similar logarithmic difference estimate again holds for the \textit{Wilson operator}.

There has been a surge of activities in extending the classical Nevanlinna theory which is based on differential operator to various difference operators in recent years such as the ones mentioned above \cite{Chen2014}. The idea has been extended to tropical functions \cite{Halburd:Southall2009}, \cite{Laine:Tohge2011}. The original intention was to apply Nevanlinna theory to study integrability of non-linear difference equations (\cite{Ablowitz:Halburd:Herbst2000}, \cite{HK-2}). But as it turns out that difference type Nevanlinna theories have revealed previously unnoticed complex analytic structures of seemingly unrelated subjects far from the original intention, such as the topic discussed in this paper.

This paper is organised as follows. We will introduce basic notation of Nevanlinna theory and Askey-Wilson theory in \S2. The $\mathrm{AW-}$type Nevanlinna second main theorems will be stated in \S\ref{S:main-results-I} and \S\ref{S:main-results-II}. The definition of $\mathrm{AW-}$type Nevanlinna counting function will also be defined in \S\ref{S:main-results-II}. The proofs of the logarithmic difference estimate (\ref{E:log-est}) and the truncated form of the second main theorem are given in \S\ref{S:pointwise} and \S\ref{S:main-results-II} respectively. The $\mathrm{AW-}$type Nevanlinna defect relations as well as an $\mathrm{AW-}$type Picard theorem are given in \S\ref{S:deficient}. This is followed by examples constructed with \textit{arbitrary rational} $\mathrm{AW-}$Nevanlinna deficient values in \S\ref{S:examples}. We characterize  the transcendental functions that belongs to the kernel of the $\mathrm{AW-}$operator in \S\ref{S:kernel}. These are the so-called $\mathrm{AW-}$constants. We also illustrate how these functions are related to certain classical identities of Jacobi theta-functions there. It is known that the Askey-Wilson orthogonal polynomials are eigen-functions to a second-order linear self-adjoint difference equation given in \cite{Askey:Wilson1985}. In \S \ref{S:unicity} we demonstrate that if two finite logarithmic order meromorphic functions such that the pre-images at five distinct points in $\mathbb{C}$ are identical except for an infinite sequences of the form as given in (\ref{E:AW-except}), then the two functions must be identical, thus giving an $\mathrm{AW-}$Nevanlinna version of the well-known unicity theorem.  We study the Nevanlinna growth of entire solutions to a more general second-order difference equation in \S\ref{E:equation} than the Askey-Wilson self-adjoint Strum-Liouville type equation \eqref{E:AW-eqn} using the tools that we have developed in this paper.

\section{\text{Askey-Wilson operator and Nevanlinna characteristic}}

%\section*{This is an unnumbered first-level section head}
%This is an example of an unnumbered first-level heading.

%% The correct journal style for \specialsection is all uppercase; a known bug
%% in amsart.cls prevents this, so input must be uppercase until it is fixed.
%\specialsection*{This is a Special Section Head}
%\specialsection*{THIS IS A SPECIAL SECTION HEAD}
%This is an example of a special section head%
%%%%%%%%%%%%%%%%%%%%%%%%%%%%%%%%%%%%%%%%%%%%%%%%%%%%%%%%%%%%%%%%%%%%%%%%
%\footnote{Here is an example of a footnote. Notice that this footnote
%text is running on so that it can stand as an example of how a footnote
%with separate paragraphs should be written.
%\par
%And here is the beginning of the second paragraph.}%
%%%%%%%%%%%%%%%%%%%%%%%%%%%%%%%%%%%%%%%%%%%%%%%%%%%%%%%%%%%%%%%%%%%%%%%%

%\section{This is a numbered first-level section head}
%This is an example of a numbered first-level heading.

%\subsection{This is a numbered second-level section head}
%This is an example of a numbered second-level heading.

%\subsection*{This is an unnumbered second-level section head}
%This is an example of an unnumbered second-level heading.

%\subsubsection{This is a numbered third-level section head}
%This is an example of a numbered third-level heading.

%\subsubsection*{This is an unnumbered third-level section head}
%This is an example of an unnumbered third-level heading.

Let $f(x)$ be a meromorphic function on $\mathbb{C}$. Let $r=|x|$, then we denote $\log^+ r=\max\{\log r,\, 0\}$. We define the \textit{Nevanlinna characteristic of} $f$ to be the real-valued function
	\begin{equation}
		\label{E:nevanlinna}
			T(r,\, f):=m(r,\, f)+N(r,\, f),
	\end{equation}
where
	\begin{equation}
		\label{E:m-N}
			m(r,\, f)=\int_0^{2\pi}\log^+|f(re^{i\theta})|\,d\theta,\quad
			N(r,\, f)=\int_0^r\frac{n(t,\, f)-n(0,\, f)}{t}\, dt
	\end{equation}
and $n(r,\, t)$ denote the number of poles in $\{|x|<r\}$. The real-valued functions 
$m(r,\, f)$ and $N(r,\, f)$ are called the \textit{proximity} and \textit{integrated counting functions} respectively. The characteristic function $T(r,\, f)$ is an increasing convex function of $\log r$, which plays the role of $\log M(r,\, f)$ for an entire function. The \textit{first fundamental theorem} states that for any complex number $c\in {\mathbb{C}}$
\begin{equation}
	\label{E:T}
	T(r,\, c):=T\Big(r,\, \frac{1}{f-c}\Big)=T(r,\, f)+O(1)
\end{equation}
as $r\to+\infty$.   We refer the reader to Nevanlinna's \cite{Nev70} and Hayman's classics \cite{Hayman1964} for the details of the Nevanlinna theory.

We now consider the Askey-Wilson operator. We shall follow the original notation introduced by Askey and Wilson in \cite{Askey:Wilson1985} (see also alternative notation in \cite[p. 300]{Ism2005}) with slight modifications. Let $f(x)$ be a meromorphic function on $\mathbb{C}$. Let $x=\cos\theta$. We define 
%$E^\pm_q(e^{i\theta})$ to be 
%	\begin{equation}
%		E^+_q\,f(e^{i\theta})=f(q^{\frac12}\, e^{i\theta}),\qquad E^-_q\,f(e^{i\theta})=f(q^{-\frac12}\, e^{i\theta})
%	\end{equation}
% and
%	\begin{equation}
%		\delta_q\, f(e^{i\theta})=(E^+_q-E^-_q)\,f(e^{i\theta}).
%	\end{equation}
%The \textit
\begin{equation}
\label{E:breve}
	\breve{f}(z)=f\big((z+1/z)/2\big)=f(x)=f(\cos \theta),\qquad z=e^{i\theta}.
\end{equation}
That is, we regard the function $f(x)$ as a function  $\breve{f}(z)$ of $e^{i\theta}=z$. 
Then for $x\not=\pm1$ the $q-$divided difference operator
\begin{equation}
\label{E:AW-operator}
	\big(\mathcal{D}_qf\big)(x):=\frac{\delta_q \breve{f}}{\delta_q\breve{x}}:=
\frac{\breve{f}(q^{\frac12}e^{i\theta})-\breve{f}(q^{-\frac12}e^{i\theta})}
	{\breve{e}(q^{\frac12}e^{i\theta})-\breve{e}(q^{-\frac12}e^{i\theta})},\qquad e^{i\theta}=z,
\end{equation}
where $e(x)=x$ is the identity map, is called the \textit{Askey-Wilson divided difference operator}. In these exceptional cases, we have
		$\big(\mathcal{D}_qf\big)(\pm1)
		=\lim_{\substack{
			x\to\pm 1\\
			x\not=\pm1}}\big(\mathcal{D}_qf\big)(x)\\ =f^{\prime}(\pm (q^{\frac12}+q^{-\frac12})/2)$ instead. It can also be written in the equivalent form
\begin{equation}
\label{E:AW-operator-2}
	\big(\mathcal{D}_qf\big)(x):=\frac{\breve{f}(q^{\frac12}e^{i\theta})-\breve{f}(q^{-\frac12}e^{i\theta})}
	{(q^\frac12-q^{-\frac12})(z-1/z)/2},\qquad x=(z+1/z)/2=\cos\theta.
\end{equation}
Since there are two branches of $z$ that corresponds to each fixed $x$, we choose a branch of $z=x+\sqrt{x^2-1}$ such that $\sqrt{x^2-1}\approx x$ as $x\to\infty$ and $x\not\in[-1,\, 1]$. 
%By choosing this branch which we denote by branch $A$ for $\sqrt{x^2-1}$, $f(z)=f(x+\sqrt{x^2-1})$ is defined for all $z$ in $\mathbb{C}-[-1,\, 1]$. 
We define the values of $z$ on $[-1,\, 1]$ by the limiting process that $x$ approaches the interval $[-1,\, 1]$ from above the real axis. Thus, $z$ assumes the value $z=x+i\sqrt{1-x^2}$ where $x$ is now real and $|x|\le 1$. So we can guarantee that for each $x$ in $\mathbb{C}$ there corresponds a unique $z$ in $\mathbb{C}$ and $z\to\infty$ as $x\to\infty$. Finally we note that if we know that $f(x)$ is analytic at $x$, then
\[
	\lim_{q\to 1}\big(\mathcal{D}_qf\big)(x)=f^\prime(x).
\]
%\begin{lemma}[\cite{CF2008}]\label{C:discrete-quotient} Let $f(z)$ be a meromorphic function of finite order $\sigma$  and let $\eta$ be a non-zero complex number. Then for each $\varepsilon>0$, we have
%\begin{equation}
%\label{E:discrete-proximity}
%    m\Big(r,\,\frac{f(z+\eta)}{f(z)}\Big)+m\Big(r,\,
%    \frac{f(z)}{f(z+\eta)}\Big)=O(r^{\sigma-1+\varepsilon}).
%\end{equation}
%\end{lemma}  

%The above estimate is known as a \textit{logarithmic difference estimate} which is a discrete analogue of the original logarithmic derivative lemma proved by Nevanlinna that $m(r,\,f^\prime/f)=O\big(\log T(r,\, f)\big)$ holds outside a set of $r$ of finite linear measure for any meromorphic function (without order restriction) on $\mathbb{C}$. A similar estimate is also proved independently by Halburd and Korhonen in \cite{HK-1}.
We can now define a polynomials basis
\begin{equation}
\label{E:(1-x^n)}
	\phi_n(\cos\theta; a):=(ae^{i\theta},\, ae^{-i\theta};\, q)_n
	=\prod_{k=0}^{n-1}(1-2axq^k+a^2q^{2k});
\end{equation}
%and
%\[
%\begin{equation}
%\label{E:x^n}
%	\rho_n(\cos\theta):=(1+e^{2i\theta})(-q^{2-n}e^{2i\theta};\, q^2)_{n-1}e^{-in\theta}
%\end{equation}
%\]
 which plays the role of the $(1-x)^n$ in conventional differential operator. Askey and Wilson \cite{Askey:Wilson1985} computed that 
\begin{equation}
\label{E:diff-basis}
	\mathcal{D}_q\phi_n(x;\, a)=-\frac{2a(1-q^n)}{1-q}\,\phi_{n-1}(x;\, aq^\frac12),
\end{equation}
 for each integer $n\ge 1$. 
%In particular, the $\mathcal{D}_q$ acts naturally on the Chebyshev polynomials:
%\[
%	\mathcal{D}_qT_n(x)=\frac{q^{n/2}-q^{-n/2}}{q^{1/2}-q^{-1/2}}U_{n-1}(x),
%\]
%where the $T_n(x)$ and $U_n(x)$ are the \textit{Chebyshev polynomials of the first and the second kinds}. 
Ismail and Stanton \cite{Ismail:Stanton2003_a} established that if $f(x)$ is an entire function satisfying
\begin{equation}
\label{E:growth-rate}
	\limsup_{r\to\infty}\frac{\log M(r,\, f)}{(\log r)^2}=c<\frac{1}{2\log |q|},
\end{equation}
where $M(r,\, f):=\max_{|x|=r}|f(x)|$ denotes the maximum modulus of $f$, then one has
\begin{equation}
\label{E:taylor}
	f(x)=\sum_{k=0}^\infty f_{k, \phi}\,
(ae^{i\theta},\, ae^{-i\theta};\, q)_k ,\quad
	f_{k,\,\phi}= \frac{(q-1)^k}{(2a)^k(q;\, q)_k}\, q^{-k(k-1)/4}
	(\mathcal{D}_q^kf)(x_k)
\end{equation}
where the $f_{k,\,\phi}$ is the $k\mathrm{-th}$\textit{Taylor coefficients} and the $x_k$ is defined by
\begin{equation}
\label{E:interpolation-point}
	x_k:= \big(a q^{k/2}+q^{-k/2}/a\big)/2, \qquad k\ge 0.
\end{equation}
We note, however, that the \textit{interpolation points} $x_k$ (\ref{E:AW-except}) are those points with $k$ being even.

We record here some simple observations about the operator $\mathcal{D}_q$ acting on meromorphic functions, the justification of them will be given in the Appendix \ref{S:meromorphic}. We first need the  \textit{averaging operator} \cite[p. 301]{Ism2005}:  
\begin{equation*}
%\label{E:average-operator}
		(\mathcal{A}_qf)(x)=\frac12\big[\breve{f}(q^\frac12z)+
		\breve{f}(q^{-\frac12}z)\big].
\end{equation*}
\begin{theorem}\label{T:meromorphic} Let $f$ be an entire function. Then $\mathcal{A}_qf$ and $\mathcal{D}_q f$ are entire. Moreover, if $f(x)$ is meromorphic, then so are $\mathcal{A}_qf$ and $\mathcal{D}_q f$.
\end{theorem}
The useful \textit{product/quotient rule} \cite[p. 301]{Ism2005} is given by
\begin{equation}
\label{E:product}
	\mathcal{D}_q(f/g)=(\mathcal{A}_qf)(\mathcal{D}_q1/g)+(\mathcal{A}_q1/g)(\mathcal{D}_qf).
\end{equation}
Since we consider meromorphic functions in this paper so we extend the growth restriction on $f$ from (\ref{E:growth-rate}) to those of \textit{finite logarithmic order} \cite{Chern2005}, \cite{BP2007} defined by

\begin{equation}
\label{E:growth-rate-2}
		\limsup_{r\to\infty}\frac{\log T(r,\, f)}{\log\log r}=\sigma_{\log}(f)=\sigma_{\log}<+\infty.
\end{equation}
It follows from an elementary consideration that the logarithmic order  for transcendental function must have $\sigma_{\log}(f)\ge 1$, and  $\sigma_{\log}(f)=1$ for rational functions and $\sigma_{\log}=0$ for constant functions. We note that the growth assumption in (\ref{E:growth-rate}) is a special case of our (\ref{E:growth-rate-2}).

Now let us suppose that $f(x)$ is a meromorphic function that satisfies (\ref{E:growth-rate-2}). Then $f$ has order zero (in the proper Nevanlinna order sense (see \cite{Hayman1964})). It follows from a result of Miles \cite{Miles1972} that $f$ can be represented as a quotient $f=g/h$ where both $g$ and $h$ are entire functions that each of them again satisfies (\ref{E:growth-rate-2}). Thus the Askey-Wilson operator is well-defined on the class of slow-growing finite logarithmic order meromorphic functions. We refer the reader to \cite{Barry1962}, \cite{Chern2005}, \cite{BP2007}  and \cite{Hayman1989} for further properties of slow-growing meromorphic functions.

\section{\textrm{Askey-Wilson type Nevanlinna theory -- Part I: Preliminaries}}
\label{S:main-results-I}
 Nevanlinna's \textit{second main theorem} is a deep generalisation of the Picard theorem. Nevanlinna's second main theorem implies that for any meromorphic function $f$ satisfies the \textit{defect relation} $\sum_{c \in \hat{\mathbb{C}}}\delta(c)\le 2$. That is, if $f\not=a,\,b,\, c$ on $\hat{\mathbb{C}}$, then $\delta(a)=\delta(b)=\delta(c)=1$. This is a  contradiction to Nevanlinna's defect relation. 
 The proof of the \textit{Second Main Theorem} is based on the logarithmic derivative estimates $m(r,\,f^\prime/f)=o(T(r,\,f))$ which is valid for all $|x|=r$ if $f$ has finite order and outside an exceptional set of finite linear measure in general. We have obtained earlier that for a fixed $\eta\not= 0$ and any finite order meromorphic function $f$ of finite order $\sigma$, and arbitrary $\varepsilon>0$ the estimate $m\big(f(x+\eta)/f(x)\big)=O(r^{\sigma-1+\varepsilon})$  \cite{Chiang:Feng2008} 
 valid for all $|x|=r$. Halburd and Korhonen \cite{HK-1} proved a comparable estimate independently for their pioneering work on a difference version of Nevanlinna theory \cite{HK-2} and their work on the integrability of discrete Painlev\'e equations \cite{HK-3}. Here we also have a $\mathrm{AW-}$\textit{logarithmic difference lemma}: 

\begin{theorem}\label{T:log-lemma} Let $f(x)$ be a meromorphic function of finite logarithmic order $\sigma_{\log}$ (\ref{E:growth-rate-2}) such that $\mathcal{D}_qf\not\equiv 0$. Then we have, for each $\varepsilon>0$, that
\begin{equation}
\label{E:log-lemma}
	m\Big(r,\,\frac{(\mathcal{D}_qf)(x)}{f(x)}\Big)=O\big((\log r)^{\sigma_{\log}-1+\varepsilon}\big)
\end{equation}
holds for all $|x|=r>0$ sufficiently large.
\end{theorem}

This estimate is crucial to the establishment of the Nevanlinna theory in the sense of Askey-Wilson put forward in this paper. 
In addition to the average estimate above, we have obtained a corresponding pointwise estimate of the logarithmic difference which holds outside some exceptional set of $|x|$:

\begin{theorem}\label{T:pointwise} Let $f(x)$ be a meromorphic function of finite logarithmic order $\sigma_{\log}$ (\ref{E:growth-rate-2}) such that $\mathcal{D}_qf\not\equiv 0$. Then we have, for each $\varepsilon>0$, that
\smallskip
\begin{equation}
\label{E:pointwise}
	\log^+\Big|\frac{(\mathcal{D}_qf)(x)}{f(x)}\Big|=O\big((\log r)^{\sigma_{\log}-1+\varepsilon}\big)
\end{equation}
holds for all $|x|=r>0$ outside an exceptional set of finite logarithmic measure.
\end{theorem}
\medskip

This estimate when written in the form
	\[
		\Big|\frac{(\mathcal{D}_qf)(x)}{f(x)}\Big|=\exp\big[{(\log |x|)^{\sigma_{\log}-1+\varepsilon}}\big]
	\]
should be compared to our earlier estimate $|f(x+1)/f(x)|\le \exp\big[|x|^{\sigma-1+\varepsilon}\big]$ \cite[Theorem 8.2]{Chiang:Feng2008} and the classical estimate $|f^\prime(x)/f(x)|\le |x|^{\sigma-1+\varepsilon}$ of Gundersen \cite[Corollary 2]{Gund88} for meromorphic function $f$ of order $\sigma$, both hold outside exceptional sets of $|x|$ of finite logarithmic measures. An analogue has been obtained recently by Cheng and the first  author of this paper in \cite{Cheng:Chiang2017} for the Wilson divided difference operator.  However, unlike in all these estimates where Cartan's lemma was used in their derivations, our argument is direct and avoids the Cartan lemma. 
\medskip

We will prove the Theorem  \ref{T:log-lemma} and Theorem \ref{T:pointwise} 
 in section \S\ref{S:pointwise}.

\begin{theorem}
\label{T:integrated-counting} Let $f$ be a meromorphic function of finite logarithmic order $\sigma_{\log}$. Then, for each $\varepsilon>0$,

\begin{equation}
	N(r,\, \mathcal{D}_qf)\le 2 N(r,\, f)+O\big((\log r)^{\sigma_{\log}-1+\varepsilon}\big)+O(\log r).
\end{equation}
\end{theorem}

We shall prove this theorem in \S\ref{S:proof_counting}.

\begin{theorem}
\label{T:simple-order} Let $f$ be a meromorphic function of finite logarithmic order $\sigma_{\log}$. Then, for each $\varepsilon>0$,
\smallskip
\begin{equation}
\label{E:AW-T-est-1}
	T(r,\, \mathcal{D}_qf)\le 2T(r,\, f)+O\big((\log r)^{\sigma_{\log}-1+\varepsilon}\big)+O(\log r).
\end{equation}
In particular, this implies
	\begin{equation}
	\sigma_{\log} (\mathcal{D}_qf)\le \sigma_{\log} (f)=\sigma_{\log}.
	\end{equation}
\end{theorem}

\begin{proof} We deduce from Theorem \ref{T:log-lemma} and Theorem \ref{T:integrated-counting} that
\begin{equation}
	\begin{split}
	T(r,\, \mathcal{D}_qf)&\le m\Big(r,\,\frac{(\mathcal{D}_qf)(x)}{f(x)}\Big)+m(r,\, f)+N(r,\, \mathcal{D}_qf)\\
	&\le m(r,\, f)+2 N(r,\, f)+O\big((\log r)\big)^{\sigma_{\log}-1+\varepsilon}+O(\log r)\\
	&\le 2T(r,\, f)+O\big((\log r)^{\sigma_{\log}-1+\varepsilon}\big)+O(\log r),
	\end{split}
\end{equation}

\noindent as required.
\end{proof}

We are now ready to state our first version of {the Second Main Theorem} whose proof will be given in \S\ref{S:2nd-proof}.

\begin{theorem}\label{T:2nd-Main-1} Suppose that $f(z)$ is a meromorphic function of finite logarithmic order $\sigma_{\log}$ $\mathrm{(\ref{E:growth-rate-2})}$ such that $\mathcal{D}_qf\not\equiv 0$ and let $A_1,\, A_2,\cdots, A_p$ $\mathrm{(}p\ge 2\mathrm{)}$, be mutually distinct elements in $\mathbb{C}$. Then we have for every $\varepsilon>0$
\smallskip
\begin{equation}
\label{E:2nd-main-1}
	m(r,\, f)+\sum_{\nu=1}^pm(r,\, A_\nu)\le 2\,T(r,\,f)-\mathfrak{N}_{\mathrm{AW}}(r,\, f)+O\big((\log r)^{\sigma_{\log}-1+\varepsilon}\big)
\end{equation}
 holds for all $r=|x|>0$, where
\begin{equation}
\label{E:remainder}
	\mathfrak{N}_{\mathrm{AW}}(r,\, f):=2N(r,\, f)-N(r,\, \mathcal{D}_qf)+
	N\Big(r,\, \frac{1}{\mathcal{D}_qf}\Big).
\end{equation}
\end{theorem}

\section{\text{Logarithmic difference estimates and proofs of Theorem \ref{T:pointwise}  and \ref{T:log-lemma}
}}\label{S:pointwise}

%\subsubsection*{Proof of the Theorem \ref{T:pointwise}}

We recall the following elementary estimate.

\begin{lemma}[\cite{Chiang:Feng2008}]
\label{L:lemma-2}
    Let $\alpha,\, 0<\alpha\le 1$ be given. Then there exists a constant $C_\alpha>0$ depending only on $\alpha$, such that for any two complex numbers $x_1$ and $x_2$, we have the inequality
\begin{equation}
\label{E:lemma-2}
    \left|\log\left|\frac{x_1}{x_2}\right|\right|\le C_\alpha
    \left(\left|\frac{x_1-x_2}{x_2}\right|^\alpha+
    \left|\frac{x_2-x_1}{x_1}\right|^\alpha\right).
\end{equation}
In particular, $C_1=1$.
\end{lemma}
%\medskip

%\begin{lemma}[\cite{HX} and \cite{JV}]
%\label{L:lemma-3}
% Let $\alpha,\, 0<\alpha<1$ be given, then for every given complex number $w$, we have
%\begin{equation}
%\label{E:lemma-3}
 %   \frac{1}{2\pi}\int_0^{2\pi}\frac{1}{|re^{i\theta}-w|^{\alpha}}\,
 %   d\theta\leq \frac{1}{(1-\alpha)r^{\alpha}}.
%\end{equation}
%\end{lemma}

%\bigskip

\begin{lemma}\label{T:pointwise-lemma} Let $f(x)$ be a meromorphic function of finite logarithmic order $\sigma_{\log}$ (\ref{E:growth-rate-2}) such that $\mathcal{D}_qf\not\equiv 0$ and $\alpha$ is an arbitrary real number such that $0<\alpha<1$. Then there exist a positive constant $C_\alpha$ such that for $2(|q^{1/2}|+|q^{-1/2}|)|x|<R$, we have
		\begin{equation}
		\begin{split}
		\label{E:first-part-1}
			&\log^+\Big|\frac{(\mathcal{D}_q f)(x)}{f(x)}\Big|
		\le  \frac{4\,R\,(|q^{1/2}-1|+|q^{-1/2}-1|)\,|x|}{(R-|x|)[R-2(|q^{1/2}|+|q^{-1/2}|)|x|]}\,\Big(m\big(R,\, f\big)+m\big(R,\, \frac1f\big)\Big)\\
				& +  2 (|q^{1/2}-1|+|q^{-1/2}-1|)|x|\Big(\frac{1}{R-|x|}+\frac{1}{R-2(|q^{1/2}|+|q^{-1/2}|)|x|}\Big)\\
				&\qquad\times\Big(n(R,\, f)+n(R,\,\frac{1}{f})\Big)\\
				&+ 2C_\alpha (|q^{1/2}-1|^\alpha+|q^{-1/2}-1|^\alpha)|x|^\alpha)
		\sum_{|c_n|<R} \frac{1}{|x-c_n|^\alpha}\\
		&+ 2C_\alpha (|q^{-1/2}-1|^\alpha|x|^\alpha)\sum_{|c_n|<R} \frac{1}{|x+c(q)q^{-1/2}z^{-1}-q^{-1/2}c_{n}|^\alpha}\\
		&+ 2C_\alpha (|q^{1/2}-1|^\alpha|x|^\alpha)\sum_{|c_n|<R} \frac{1}{|x-c(q)q^{1/2}z^{-1}-q^{1/2}c_{n}|^\alpha}+\log 2.
		\end{split}
	\end{equation}
where the  $\{c_n\}$ denotes the combined zeros and poles sequences of $f$.
\end{lemma}

\begin{proof}   
 We start by expressing all the logarithmic difference in terms of complex variables $x$ as well as in $z$ in the Askey-Wilson divided difference operator. So it follows from (\ref{E:AW-operator-2}) that
\begin{align}
\label{E:preliminary}
	&\frac{(\mathcal{D}_q f)(x)}{f(x)}
		= \frac{\breve{f}(q^{\frac12}e^{i\theta})-\breve{f}(q^{-\frac12}e^{i\theta})}
	{f(x)(q^\frac12-q^{-\frac12})(z-1/z)/2},\qquad x=(z+1/z)/2=\cos\theta\\
		&=\frac{f\big[(q^{1/2}z+q^{-1/2}z^{-1})/2\big]-
		f\big[(q^{-1/2}z+q^{1/2}z^{-1})/2\big]}
		{f(x)(q^\frac12-q^{-\frac12})(z-1/z)/2}\notag\\
%		&=\frac{f\big[q^{1/2}/2(x+\sqrt{x^2-1})+q^{-1/2}/2(x+\sqrt{x^2-1})^{-1}\big]}{f(x)(q^\frac12-q^{-\frac12})(z-1/z)/2}\notag\\
%		&\qquad-\frac{f\big[q^{-1/2}/2(x+\sqrt{x^2-1})+q^{1/2}/2(x+\sqrt{x^2-1})^{-1}\big]}{f(x)(q^\frac12-q^{-\frac12})(z-1/z)/2}\notag
		&=\frac{1}{(q^\frac12-q^{-\frac12})(z-1/z)/2}\Bigg(
\frac{f\big[(q^{1/2}z+q^{-1/2}z^{-1})/2\big]}{f(x)}-\frac{f\big[(q^{-1/2}z+q^{1/2}z^{-1})/2\big]}{f(x)}\Bigg)\notag
\end{align}
where we recall that we have fixed our branch of $z$ for the corresponding $x$ in the above expressions.  Let
		\begin{equation}
			\label{E:substitution}
			c(q)={(q^{-1/2}-q^{1/2})}/{2}.
		\end{equation}
We deduce from (\ref{E:preliminary}) that, by letting $|x|$ and hence $|z|$ to be sufficiently large
	\begin{equation}
		\label{E:first-split}
			\begin{split}
				\log^+ & \Big|\frac{(\mathcal{D}_q f)(x)}{f(x)}\Big|
		\le \log^+\Big|\frac{1}{(q^\frac12-q^{-\frac12})(z-1/z)/2}
\Big|+\log^+\Big|\frac{f\big[(q^{1/2}z+q^{-1/2}z^{-1})/2\big]}{f(x)}\Big|\\
		&\qquad +\log^+\Big|\frac{f\big[(q^{-1/2}z+q^{1/2}z^{-1})/2\big]}{f(x)}\Big|+\log 2\\
		&\le \log^+{2}/{|c(q)z|}+\log^+\Big|\frac{f\big[(q^{1/2}z+q^{-1/2}z^{-1})/2\big]}{f(x)}\Big|\\
		&\qquad +\log^+\Big|\frac{f\big[(q^{-1/2}z+q^{1/2}z^{-1})/2\big]}{f(x)}\Big|+\log 2\\
		&=\bigg|\log\Big|\frac{f\big[(q^{1/2}z+q^{-1/2}z^{-1})/2\big]}{f(x)}\Big|\,\bigg|+\bigg|\log\Big|\frac{f\big[(q^{-1/2}z+q^{1/2}z^{-1})/2\big]}{f(x)}\Big|\,\bigg|+\log 2.
			\end{split}
	\end{equation}
For $|x|$ and hence $|z|$ to be sufficiently large,
\begin{equation}
			\label{}
			|{(q^{\pm 1/2}z+q^{\mp 1/2}z^{-1})}/{2}|=|q^{\pm 1/2}x+{(q^{\mp 1/2}-q^{\pm 1/2})z^{-1}}/{2}|\le 2|q^{\pm 1/2}x|<R.
		\end{equation}
%		Similarly 
%\begin{equation}
%			\label{}
%			|{q^{-1/2}z+q^{1/2}z^{-1}}/{2}|=|q^{-1/2}x+{(q^{1/2}-q^{-1/2})z^{-1}}/{2}|\le 2|q^{-1/2}x|<R.
%		\end{equation}
It is obvious that $|x|<R$. We apply the Poisson-Jensen formula (see e.g., \cite[p. 1]{Hayman1964}) to estimate the individual terms on the right-hand side of the above expression (\ref{E:first-split}). Thus,
	\begin{equation*}
		\begin{split}
		\log&  \Bigg|\frac{f\big[(q^{1/2}z +q^{-1/2}z^{-1})/2\big]}{f(x)}\Bigg|=\log \Big|f\big[(q^{1/2}z+q^{-1/2}z^{-1})/2\big]\Big|-\log|f(x)|\\
	&=\frac{1}{2\pi}\int_0^{2\pi} \log|f(Re^{i\phi})\,|\Re\Big(
	\frac{Re^{i\phi}+(q^{1/2}z+q^{-1/2}z^{-1})/2}{Re^{i\phi}-(q^{1/2}z+q^{-1/2}z^{-1})/2}\Big)\,d\phi\\
	&\qquad -\frac{1}{2\pi}\int_0^{2\pi} \log\,|f(Re^{i\phi})|\,\Re
	\Big(\frac{Re^{i\phi}+x}{Re^{i\phi}-x}\Big)\,d\phi\\
	&\qquad+\sum_{|b_\mu|<R}\log\Big|\frac{R^2-\bar{b}_\mu(q^{1/2}z+q^{-1/2}z^{-1})/2}
{R[(q^{1/2}z+q^{-1/2}z^{-1})/2-b_\mu]}\Big|\\
	&\qquad-\sum_{|a_\nu|<R}\log\Big|\frac{R^2-\bar{a}_\nu\argone}
{R[\argone-a_\nu]}\Big|\\
	&\qquad-\sum_{|b_\mu|<R}\log\Big|\frac{R^2-\bar{b}_\mu x}{R(x-{b}_\mu)}\Big|+\sum_{|a_\nu|<R}\log\Big|\frac{R^2-\bar{a}_\nu x}{R(x-{a}_\nu)}\Big|.
	\end{split}
	\end{equation*}
% \eject
That is,
\begin{equation}
	\label{E:first-ratio}
	\begin{split}
	&\log \Bigg|\frac{f\big[(q^{1/2}z+q^{-1/2}z^{-1})/2\big]}{f(x)}\Bigg| \\
%=\log \Big|f\big[(q^{1/2}z+q^{-1/2}z^{-1})/2\big]\Big|-\log|f(x)|\\
%	&=\frac{1}{2\pi}\int_0^{2\pi} \log|f(Re^{i\phi})\,|\Re\Big(
%	\frac{Re^{i\phi}+(q^{1/2}z+q^{-1/2}z^{-1})/2}{Re^{i\phi}-(q^{1/2}z+q^{-1/2}z^{-1})/2}\Big)\,d\phi\\
%	&\qquad -\frac{1}{2\pi}\int_0^{2\pi} \log\,|f(Re^{i\phi})|\,\Re
%	\Big(\frac{Re^{i\phi}+x}{Re^{i\phi}-x}\Big)\,d\phi\\
%	&\qquad+\sum_{|b_\mu|<R}\log\Big|\frac{R^2-\bar{b}_\mu(q^{1/2}z+q^{-1/2}z^{-1})/2}
%{R[(q^{1/2}z+q^{-1/2}z^{-1})/2-b_\mu]}\Big|-\sum_{|a_\nu|<R}\log\Big|\frac{R^2-\bar{a}_\nu\argone} {R[\argone-a_\nu]}\Big|\\
%	&\quad-\sum_{|b_\mu|<R}\log\Big|\frac{R^2-\bar{b}_\mu x}{R(x-{b}_\mu)}\Big|+\sum_{|a_\nu|<R}\log\Big|\frac{R^2-\bar{a}_\nu x}{R(x-{a}_\nu)}\Big| \\
	&=\frac{1}{2\pi}\int_0^{2\pi} \log\,|f(Re^{i\phi})\,|\Re\Big(
	\frac{Re^{i\phi}\, (q^{1/2}z+q^{-1/2}z^{-1}-2x)}
	{(Re^{i\phi}-x)[Re^{i\phi}-\argone]}\Big)\,d\phi\\
	\medskip
	&\qquad+\sum_{|b_\mu|<R}\log\Big|\frac{R^2-\bar{b}_\mu\argone}{R^2-\bar{b}_\mu x}\Big|\\
	&\qquad-\sum_{|b_\mu|<R}\log\Big|\frac{R(\argone-b_\mu)}{R(x-{b}_\mu)}\Big|\\
	&\qquad-\sum_{|a_\nu|<R}\log\Big|\frac{R^2-\bar{a}_\nu\argone}{R^2-\bar{a}_\nu x}\Big|\\
	&\qquad+\sum_{|a_\nu|<R}\log\Big|\frac{R(\argone-a_\nu)}{R(x-{a}_\nu)}\Big|\\
	&=\frac{1}{2\pi}\int_0^{2\pi} \log\,|f(Re^{i\phi})\,|\Re\Big(
	\frac{Re^{i\phi}\, [2(q^{1/2}-1)\,x+(q^{-1/2}-q^{1/2})z^{-1}]}
	{(Re^{i\phi}-x)[Re^{i\phi}-(q^{1/2}x+(q^{-1/2}-q^{1/2})z^{-1}/2)]}\Big)\,d\phi\\
	\medskip
	&\qquad+\sum_{|b_\mu|<R}\log\Big|\frac{R^2-\bar{b}_\mu[q^{1/2}x+(q^{-1/2}-q^{1/2})/2\, z^{-1}]}{R^2-\bar{b}_\mu x}\Big|\\
	&\qquad-\sum_{|b_\mu|<R}\log\Big|\frac{[q^{1/2}x+(q^{-1/2}-q^{1/2})/2\, z^{-1}]-b_\mu}{(x-{b}_\mu)}\Big|\\
	&\qquad-\sum_{|a_\nu|<R}\log\Big|\frac{R^2-\bar{a}_\nu[q^{1/2}x+(q^{-1/2}-q^{1/2})/2\, z^{-1}]}{R^2-\bar{a}_\nu x}\Big|\\
	&\qquad+\sum_{|a_\nu|<R}\log\Big|\frac{[q^{1/2}x+(q^{-1/2}-q^{1/2})/2\, z^{-1}]-a_\nu}{(x-{a}_\nu)}\Big|\\
	&=\frac{1}{2\pi}\int_0^{2\pi} \log\,|f(Re^{i\phi})\,|\Re\Big(
	\frac{2\, Re^{i\phi}\, [(q^{1/2}-1)\,x+c(q)\,z^{-1}]}
	{(Re^{i\phi}-x)[Re^{i\phi}-(q^{1/2}x+c(q)\, z^{-1})]}\Big)\,d\phi\\
	\medskip
	&\qquad+\sum_{|b_\mu|<R}\log\Big|\frac{R^2-\bar{b}_\mu[q^{1/2}x+c(q)\, z^{-1}]}{R^2-\bar{b}_\mu x}\Big|-\sum_{|b_\mu|<R}\log\Big|\frac{[q^{1/2}x+c(q)\, z^{-1}]-b_\mu}{x-{b}_\mu}\Big|\\
	&\qquad-\sum_{|a_\nu|<R}\log\Big|\frac{R^2-\bar{a}_\nu[q^{1/2}x+c(q)\, z^{-1}]}{R^2-\bar{a}_\nu x}\Big|+\sum_{|a_\nu|<R}\log\Big|\frac{[q^{1/2}x+c(q)\, z^{-1}]-a_\nu}{x-{a}_\nu}\Big|\\
\end{split}
\end{equation} 
where we have made the substitution (\ref{E:substitution}).
 We let $|x|$ and hence $|z|$ be sufficiently large, so we may assume that 
	\[
		|c(q)z^{-1}|<\min(|q^{-1/2}x|,\,|q^{1/2}x|, |(q^{1/2}-1)x|,\, |(q^{-1/2}-1)x|)
	\]
in the following calculations.

We notice that the integrated logarithmic average term from (\ref{E:first-ratio}) has the following upper bound
	\be
		\label{E:proximity-estimate}
	\begin{split}
	&\Big|\frac{1}{2\pi}\int_0^{2\pi} \log\,|f(Re^{i\phi})\,|\Re\Big(
	\frac{2\, Re^{i\phi}\, [(q^{1/2}-1)\,x+c(q)\,z^{-1}]}
	{(Re^{i\phi}-x)[Re^{i\phi}-(q^{1/2}x+c(q)\, z^{-1})]}\Big)\,d\phi\Big|\\
	&\le \frac{1}{2\pi}\int_0^{2\pi} \big|\log\,|f(Re^{i\phi})\,|\big|
	\frac{4\,R\,|q^{1/2}-1|\,|x|}{(R-|x|)(R-2|q^{1/2}x|)}\,d\phi\\
	&\le \frac{4\,R\,|q^{1/2}-1|\,|x|}{(R-|x|)(R-2|q^{1/2}x|)}\,\Big(m\big(R,\, f\big)+m\big(R,\, \frac1f\big)\Big).
	\end{split}
	\ee
Hence \eqref{E:first-ratio} becomes
	\begin{equation}
		\label{E:first-part}
			\begin{split}
				&\Bigg|\log \Big|\frac{f\big[(q^{1/2}z+q^{-1/2}z^{-1})/2\big]}{f(x)}\Big|\,\Bigg|\\
				&\le  \frac{4\,R\,|q^{1/2}-1|\,|x|}{(R-|x|)(R-2|q^{1/2}x|)}\,\Big(m\big(R,\, f\big)+m\big(R,\, \frac1f\big)\Big)\\
				& +\sum_{|b_\mu|<R}\Bigg|\log\Big|\frac{R^2-\bar{b}_\mu[q^{1/2}x+c(q)\, z^{-1}]}{R^2-\bar{b}_\mu x}\Big|\,\Bigg|+\Bigg|\sum_{|b_\mu|<R}\log\Big|\frac{[q^{1/2}x+c(q)\, z^{-1}]-b_\mu}{x-{b}_\mu}\Big|\,\Bigg|\\
	&+\sum_{|a_\nu|<R}\Bigg|\log\Big|\frac{R^2-\bar{a}_\nu[q^{1/2}x+c(q)\, z^{-1}]}{R^2-\bar{a}_\nu x}\Big|\,\Bigg|+\sum_{|a_\nu|<R}\Bigg|\log\Big|\frac{[q^{1/2}x+c(q)\, z^{-1}]-a_\nu}{x-{a}_\nu}\Big|\,\Bigg|\\
			\end{split}
	\end{equation}
 Applying the Lemma \ref{L:lemma-2} with $\alpha =1$, to each individual term in the first summand of (\ref{E:first-part}) with $|b_\nu|<R$ yields
\begin{align}	
\label{E:first-summand}
	&\Bigg|\log\Big|\frac{R^2-\bar{b}_\mu[q^{1/2}x+c(q)z^{-1}]}{R^2-\bar{b}_\mu x}\Big|\Bigg|\\
	&\le\Big|\frac{\bar{b}_\mu[(1-q^{1/2})x-c(q)z^{-1}] )}{R^2-\bar{b}_\mu[q^{1/2}x+c(q)z^{-1}]}\Big|
	+\Big|\frac{\bar{b}_\mu[(1-q^{1/2})x-c(q)z^{-1}] )}{R^2-\bar{b}_\mu x}\Big|\notag\\
	&\le \frac{2 R|1-q^{1/2}||x|}{R^2-2R|q^{1/2}x|}+\frac{2 R|1-q^{1/2}||x|}{R^2-R |x|}\notag\\	
	&= 2 |q^{1/2}-1||x|\Big(\frac{1}{R-|x|}+\frac{1}{R-2|q^{1/2}x|}\Big)\notag
%	& \le 2 |q^{1/2}-1|\Big(\frac{R}{R-|x|}+\frac{R}{R-(|q^{1/2}x|+d(q))}\Big).\notag
\end{align}
Similarly, we have, for the third summand that for $|a_\mu|<R$,

\begin{align}
\label{E:third-summand}
	&\Bigg|\log\Big|\frac{R^2-\bar{a}_\nu[q^{1/2}x+c(q)z^{-1}]}{R^2-\bar{a}_\nu x}\Big|\Bigg|\\
	&=2 |q^{1/2}-1||x|\Big(\frac{1}{R-|x|}+\frac{1}{R-2|q^{1/2}x|}\Big)\notag
%	& \le 2 |q^{1/2}-1|\Big(\frac{R}{R-|x|}+\frac{R}{R-(|q^{1/2}x|+d(q))}\Big).\notag
\end{align}
Again applying the Lemma \ref{L:lemma-2} with $0\le\alpha<1$ to each individual term in the second summand of (\ref{E:first-ratio}) yields
	\begin{align}
		\label{E:second-summand}
			&\Bigg|\log\Big|\frac{q^{1/2}x+c(q)z^{-1}-b_\mu}{x-{b}_\mu}\Big|\Bigg|\\
			&\le C_\alpha\Big(\Big|\frac{(q^{1/2}-1)x+c(q)z^{-1}}{q^{1/2}x+c(q)z^{-1}-b_{\mu}}\Big|^\alpha+  \Big|\frac{(q^{1/2}-1)x+c(q)z^{-1}}{x-b_\mu}\Big|^\alpha\Big)\notag\\
			&\le 2C_\alpha (|q^{1/2}-1|^\alpha|x|^\alpha)
		\Big(\frac{1}{|x-b_\mu|^\alpha}+\frac{1}{|(q^{1/2}x+c(q)z^{-1}-b_{\mu}|^\alpha}\Big)\notag
	\end{align}	

Similarly, we have, for the fourth summand,
\begin{align}
\label{E:fourth-summand}
	&\Bigg|\log\Big|\frac{q^{1/2}x+c(q)z^{-1}-a_\nu}{x-{a}_\nu}\Big|\Bigg|\\
%			&\le C_\alpha\Big(\Big|\frac{(q^{1/2}-1)x+c(q)z^{-1}}{q^{1/2}x+c(q)z^{-1}-a_{\mu}}\Big|^\alpha+  \Big|\frac{(q^{1/2}-1)x+c(q)z^{-1}}{x-a_\nu}\Big|^\alpha\Big)\notag\\
			&\le 2C_\alpha (|q^{1/2}-1|^\alpha|x|^\alpha)
		\Big(\frac{1}{|x-a_\nu|^\alpha}+\frac{1}{|(q^{1/2}x+c(q)z^{-1}-a_{\nu}|^\alpha}\Big).\notag
	\end{align}	

Combining the inequalities (\ref{E:first-part}), (\ref{E:first-summand}--\ref{E:fourth-summand}) yields
	\begin{equation}
		\label{E:first-part-estimate}
			\begin{split}
				&\Bigg|\log \Big|\frac{f\big[(q^{1/2}z+q^{-1/2}z^{-1})/2\big]}{f(x)}\Big|\,\Bigg|\\
				&\le  \frac{4\,R\,|q^{1/2}-1|\,|x|}{(R-|x|)[R-2|q^{1/2}x|]}\,\Big(m\big(R,\, f\big)+m\big(R,\, \frac1f\big)\Big)\\
				& \qquad+  2 |q^{1/2}-1||x|\Big(\frac{1}{R-|x|}+\frac{1}{R-2|q^{1/2}x|}\Big)\Big(n(R,\, f)+n(R,\,\frac{1}{f})\Big)\\
				& \qquad + 2C_\alpha (|q^{1/2}-1|^\alpha|x|^\alpha)
		\sum_{|a_\nu|<R} \Big(\frac{1}{|x-a_\nu|^\alpha}+\frac{1}{|(q^{1/2}x+c(q)z^{-1}-a_{\nu}|^\alpha}\Big)\\
				& \qquad + 2C_\alpha (|q^{1/2}-1|^\alpha|x|^\alpha)
		\sum_{|b_\mu|<R} \Big(\frac{1}{|x-b_\mu|^\alpha}+\frac{1}{|(q^{1/2}x+c(q)z^{-1}-b_{\mu}|^\alpha}\Big)\\
		&=  \frac{4\,R\,|q^{1/2}-1|\,|x|}{(R-|x|)[R-2|q^{1/2}x|]}\,\Big(m\big(R,\, f\big)+m\big(R,\, \frac1f\big)\Big)\\
				& \qquad+  2 |q^{1/2}-1||x|\Big(\frac{1}{R-|x|}+\frac{1}{R-2|q^{1/2}x|}\Big)\Big(n(R,\, f)+n(R,\,\frac{1}{f})\Big)\\
				& \qquad + 2C_\alpha (|q^{1/2}-1|^\alpha|x|^\alpha)
		\sum_{|c_n|<R} \Big(\frac{1}{|x-c_n|^\alpha}+\frac{1}{|(q^{1/2}x+c(q)z^{-1}-c_{n}|^\alpha}\Big)\\
			\end{split}
	\end{equation}
\medskip
where we have re-labelled all the zeros $\{a_\nu\}$ and poles $\{b_\mu\}$ by the single sequence $\{c_n\}$. 

Replacing $q$ by $q^{-1}$ in the (\ref{E:first-part-estimate}), we obtain for $|x|$ sufficiently large 
\begin{equation}
		\label{E:second-part-estimate}
			\begin{split}
				&\Bigg|\log \Big|\frac{f\big[(q^{-1/2}z+q^{1/2}z^{-1})/2\big]}{f(x)}\Big|\,\Bigg|\\
				&\le  \frac{4\,R\,|q^{-1/2}-1|\,|x|}{(R-|x|)[R-2|q^{-1/2}x|]}\,\Big(m\big(R,\, f\big)+m\big(R,\, \frac1f\big)\Big)\\
				& \qquad+  2 |q^{-1/2}-1||x|\Big(\frac{1}{R-|x|}+\frac{1}{R-2|q^{-1/2}x|}\Big)\Big(n(R,\, f)+n(R,\,\frac{1}{f})\Big)\\
				& \qquad + 2C_\alpha (|q^{-1/2}-1|^\alpha|x|^\alpha)
		\sum_{|c_n|<R} \Big(\frac{1}{|x-c_n|^\alpha}+\frac{1}{|(q^{-1/2}x-c(q)z^{-1}-c_{n}|^\alpha}\Big)
			\end{split}
	\end{equation}
Substituting the (\ref{E:first-part-estimate}) and (\ref{E:second-part-estimate}) into (\ref{E:first-split}) yields \eqref{E:first-part-1}.
\end{proof}

\subsubsection*{Proof of Theorem \ref{T:pointwise}} 

In order to give an upper bound estimate for the last three summands of the above sum, we need to avoid exceptional sets arising from the sequence given by 
	\[
		\{d_n\}=\{c_n\}\cup\{c_n\,q^{1/2}\}\cup\{c_n\,q^{-1/2}\}.
	\]
 Given $\varepsilon>0$, let 
 	\begin{equation}
		\label{E:exceptional-a}
			E_n=\Bigg\{r:\ r\in \Big[|d_n|-\frac{|d_n|}{\log^{\sigma_{\log}+\varepsilon} (|d_n|+3)},\ 
			|d_n|+\frac{|d_n|}{\log^{\sigma_{\log}+\varepsilon} (|d_n|+3)}\Big]
			\Bigg\}
	\end{equation}
and
	\[
		E=\cup_{n} E_n.
	\]
 Henceforth we consider the $|x|\not\in E$. It is not difficult to see the inequality 
	\begin{equation}
		|x-d_n|\ge \big||x|-|d_n|\big| \ge \frac{|x|}{2\log^{\sigma_{\log}+\varepsilon} (|x|+3)}.
	\end{equation}
holds for all $|x|$ sufficiently large. Thus 
	\begin{equation}
		\label{E:final-summand-a}
	 \sum_{|c_n|<R}\frac{1}{|x-c_n|^\alpha}\le
		\frac{2^\alpha\log^{\alpha\,(\sigma_{\log}+\varepsilon)}(|x|+3)}{|x|^\alpha}\Big(n(R,\, f)+n(R,\, \frac1f)\Big).
	\end{equation}
 Similarly, we have 
	\begin{equation}
		\begin{split}
		|x+c(q)q^{-1/2}z^{-1}-q^{-1/2}c_{n}| &\ge |x-q^{-1/2}c_{n}|-|c(q)q^{-1/2}z^{-1}|\\
		&\ge \big||x|-|q^{-1/2}c_{n}|\,\big|-|c(q)q^{-1/2}z^{-1}|\\
		&\ge \frac{|x|}{2\log^{\sigma_{\log}+\varepsilon} (|x|+3)}-|c(q)q^{-1/2}z^{-1}|\\
		&\ge \frac{|x|}{3\log^{\sigma_{\log}+\varepsilon} (|x|+3)}.
		\end{split}
	\end{equation}
and
	\begin{equation}
		|x-c(q)q^{1/2}z^{-1}-q^{1/2}c_{n}| \ge 
		\frac{|x|}{3\log^{\sigma_{\log}+\varepsilon} (|x|+3)}
	\end{equation}
holds for $|x|$ sufficiently large. Hence
	\begin{equation}
		\label{E:final-summand-b}
	 \sum_{|c_n|<R}\frac{1}{|x+c(q)q^{-1/2}z^{-1}-q^{-1/2}c_{n}|^\alpha}\le
		\frac{3^\alpha\log^{\alpha(\sigma_{\log}+\varepsilon)}(|x|+3)}{|x|^\alpha}\Big(n(R,\, f)+n(R,\, \frac1f)\Big)
	\end{equation}
and
\begin{equation}
		\label{E:final-summand-c}
	 \sum_{|c_n|<R}\frac{1}{|x-c(q)q^{1/2}z^{-1}-q^{1/2}c_{n}|^\alpha}\le
		\frac{3^\alpha\log^{\alpha(\sigma_{\log}+\varepsilon)}(|x|+3)}{|x|^\alpha}\Big(n(R,\, f)+n(R,\, \frac1f)\Big).
	\end{equation}

\noindent 
We obtain from (\ref{E:first-part-1}), after substituting (\ref{E:final-summand-a}), (\ref{E:final-summand-b}--\ref{E:final-summand-c}), the  inequality
\begin{equation}
		\begin{split}
		\label{E:pointwise-lemma}
			&\log^+\Big|\frac{(\mathcal{D}_q f)(x)}{f(x)}\Big|
		\le  \frac{4\,R\,(|q^{1/2}-1|+|q^{-1/2}-1|)\,|x|}{(R-|x|)[R-2(|q^{1/2}|+|q^{-1/2}|)|x|]}\,\Big(m\big(R,\, f\big)+m\big(R,\, \frac1f\big)\Big)\\
				& +  2 (|q^{1/2}-1|+|q^{-1/2}-1|)|x|
				\Big(\frac{1}{R-|x|}+\frac{1}{R-2(|q^{1/2}|+|q^{-1/2}|)|x|}\Big)\\
				&\qquad \times \Big(n(R,\, f)+n(R,\,\frac{1}{f})\Big)\\
				&+ D_\alpha \,(|q^{1/2}-1|^\alpha+|q^{-1/2}-1|^\alpha)\log^{\alpha(\sigma_{\log}+\varepsilon)}(|x|+3)\Big(n(R,\, f)+n(R,\,\frac{1}{f})\Big)
		+\log 2
		\end{split}
	\end{equation}
where $D_\alpha:=4\,C_\alpha 3^\alpha$, $|x|\not\in E$.

On the other hand,
	\begin{equation}
			\begin{split}
			&N(R^2,\, f) \ge \int_R^{R^2} \frac{n(t,\, f)-n(0,\, f)}{t}\,dt
			+n(0,\, f)\log R^2\\
			&\ge n(R,\, f)\int_R^{R^2} \frac{1}{t}\, dt
			-n(0,\, f) \int_R^{R^2} \frac{dt}{t}+n(0,\, f)\log R^2\\
			&\ge n(R,\, f)\, \log R.
		\end{split}
	\end{equation}
 Hence
	\begin{align}
	\label{E:n-upper}
		&n(R,\, f)\le \frac{N(R^2,\, f)}{\log R}=\frac{O[(\log R^2)^{\sigma_{\log}+\frac{\varepsilon}{2}}]}{\log R}\\
		&=O\big(\log^{\sigma_{\log}-1+\frac{\varepsilon}{2}} R\big).\notag
	\end{align}
Similarly, we have 
	\begin{equation}
	\label{E:n-lower}
		n\Big(R,\, \frac1f\Big)=O\big(\log^{\sigma_{\log}-1+\frac{\varepsilon}{2}} R\big).
	\end{equation}
 We now choose $\alpha=\frac{\varepsilon}{2(\sigma_{\log}+\varepsilon)}$ and substitute $|x|=r$, $R=r\,\log r$ into Lemma \ref{T:pointwise-lemma} to obtain the (\ref{E:pointwise}).

 We now compute the logarithmic measure of $E$. To do so, we first note the elementary inequality that given $\delta>0$ sufficiently small, there is a positive constant $C_\delta$ so that 
	\begin{equation}
		\log\frac{1+t}{1-t}\le C_\delta t
	\end{equation}
for $0\le t<\delta$. We assume, in the case when there are infinitely many $\{d_n\}$ (otherwise, the logarithmic measure of $E$ is obviously finite), they are ordered in the increasing moduli. Then we choose an $N$ sufficiently large such that
	\begin{equation}
			\frac{1}{{\log}^{\sigma_{\log}+\varepsilon}|d_N|}<\delta.
	\end{equation}
 Hence
	\begin{equation}
		\begin{split}
		&\text{log-meas\, } E =
		\int_{E\cap [1,\, +\infty)}\frac{dt}{t}=\int_{E\cap [1,\,|d_N|]}\frac{dt}{t}+\int_{E \cap[|d_N|,\, +\infty)}\frac{dt}{t}\\
		&\le \log |d_N|+\sum_{n=N}^\infty\int_{E_n}\frac{dt}{t}=\log|d_N|+\sum_{n=N}^\infty \log
			\Big(
			\frac{1+1/{\log^{\sigma_{\log}+\varepsilon} }|d_n|}{1-1/{\log^{\sigma_{\log}+\varepsilon} }|d_n|}
			\Big)\\
			\notag
			&\le \log |d_N|+C_\delta\sum_{n=N}^\infty\frac{1}
		{\log^{\sigma_{\log}+\varepsilon} |d_n|}< \infty
			\end{split}
		\end{equation}
where the conclusion of the last sum is convergent follows from \cite[Lemma 4.2]{Chern2005} for meromorphic functions of finite logarithm order $\sigma_{\log}$.
\qed

\subsubsection*{Proof of Theorem \ref{T:log-lemma}}
We first prove a crucial estimate. 
\begin{lemma}\label{L:trouble}
Let $0<\alpha<1$ be given. Then for each fixed $A\in\mathbb{C}$ and an arbitrary $w$ and we have
	\begin{equation}
		\label{E:trouble}
			\int_0^{2\pi} \frac{1}{| re^{i\phi}-A(re^{i\phi}-\sqrt{r^2e^{2i\phi}-1})-w|^\alpha}
			d\phi
			\le \frac{E_\alpha}{r^\alpha}
	\end{equation}
for all $r$ sufficiently large, where $E_\alpha$ is independent of $w$. Here the square root is given by $\sqrt{r^2e^{2i\phi}-1}\approx re^{i\phi}$ as $re^{i\phi}\to\infty$ (as the agreed convention in \S2).
\end{lemma}
\medskip

\begin{remark}  We note that when $A=0$ in the above lemma, with the same $0<\alpha<1$ and $w\in\mathbb{C}$ arbitrary,  recovers the known estimate
	\begin{equation}
		\label{E:lemma-3}
		    \frac{1}{2\pi}\int_0^{2\pi}\frac{1}{|re^{i\phi}-w|^{\alpha}}\,
		    d\phi\leq \frac{1}{(1-\alpha)r^{\alpha}},
	\end{equation}
holds for all $r>0$ (See e.g., \cite[p. 62]{He:Xiao1988} and \cite[p. 66]{Jank:Volkman1985}).
\end{remark}
\medskip

\begin{proof} Since for each $r$ large enough,
	\begin{equation}
		\label{E:curve}
			\Big\{ re^{i\phi}-A(re^{i\phi}-\sqrt{r^2e^{2i\phi}-1})
			\ : \
			0\le \phi\le 2 \pi\Big\}
	\end{equation}
 is a closed curve enclosing the origin, therefore, for each non-zero $w$, the curve must intersect with the array
	\begin{equation}
		\label{E:array}
			\big\{tw\, :\ 0<t<\infty\big\}
	\end{equation}
at least once. So we can find a \textit{real} $C>0$ and $0\le \beta<2\pi$ such that we can represent $w$ as
	\begin{equation}
		\label{E:w}
			w=C\big( re^{i\beta}-A(re^{i\beta}-\sqrt{r^2e^{2i\beta}-1})\big).
	\end{equation}
When $w=0$, we can still represent $w$  by \eqref{E:w} with $C=0$.

Substituting \eqref{E:w} into the integral on the left side of \eqref{E:trouble} yileds

\begin{equation}
		\label{E:trouble-2}
			\begin{split}
				&\int_0^{2\pi} \frac{1}{| re^{i\phi}-A(re^{i\phi}-\sqrt{r^2e^{2i\phi}-1})-w|^\alpha}
			d\phi \\
			&=\int_0^{2\pi} \frac{d\phi}{\big|re^{i\phi}-Cre^{i\beta}-
			A\big[(re^{i\phi}-\sqrt{r^2e^{2i\phi}-1})-C(re^{i\beta}-\sqrt{r^2e^{2i\beta}-1})\big]\big|^\alpha}\\
			&=\int_0^{2\pi} \frac{d\phi}{\big|re^{i(\phi-\beta)}-Cr-
			A\big[(re^{i(\phi-\beta)}-\sqrt{r^2 e^{2i(\phi-\beta)}-e^{-2i\beta}})-C(r-\sqrt{r^2-e^{-2i\beta}})\big]\big|^\alpha}\\
			&=\int_{-\beta}^{2\pi-\beta} \frac{d\phi}{\big|re^{i\phi}-Cr-
			A\big[(re^{i\phi}-\sqrt{r^2 e^{2i\phi}-e^{-2i\beta}})-C(r-\sqrt{r^2-e^{-2i\beta}})\big]\big|^\alpha}\\
			&=\Big(\int_{-\frac{\pi}{2}}^{\frac\pi2} + \int_{\frac{\pi}{2}}^{\frac{3\pi}{2}}\Big)\\
			&\qquad \frac{d\phi}{\big|re^{i\phi}-Cr-
			A\big[(re^{i\phi}-\sqrt{r^2 e^{2i\phi}-e^{-2i\beta}})-C(r-\sqrt{r^2-e^{-2i\beta}})\big]\big|^\alpha}\\
			&:=I_1+ I_2.
			\end{split}
	\end{equation}

We now estimate $I_1$.  We first consider
	\begin{equation}\label{E:trouble-2.5}
		\begin{split}
			&\big|re^{i\phi}-Cr-
			A\big[(re^{i\phi}-\sqrt{r^2 e^{2i\phi}-e^{-2i\beta}})-C(r-\sqrt{r^2-e^{-2i\beta}})\big]\big|\\
			&=\big| (1-C)(r-A(r-\sqrt{r^2-e^{-2i\beta}}))\\
			&\quad +r(e^{i\phi}-1)- A(re^{i\phi}-r+\sqrt{r^2-e^{-2i\beta}}
			-\sqrt{r^2 e^{2i\phi}-e^{-2i\beta}})\big|\\
			&=|r-A(r-\sqrt{r^2-e^{-2i\beta}})|\\
			&\times \Big| (1-C)+
			\frac{r(e^{i\phi}-1)- A(re^{i\phi}-r+\sqrt{r^2-e^{-2i\beta}}
			-\sqrt{r^2 e^{2i\phi}-e^{-2i\beta}})}
			{r-A(r-\sqrt{r^2-e^{-2i\beta}})}\Big|.
		\end{split}
	\end{equation}
We see (from the agreed convention) that when and $|\phi|<\pi/2$ and $r\to \infty$, 
	\begin{equation}\label{E:trouble-3}
		{re^{i\phi}}-\sqrt{r^2 e^{2i\phi}-e^{-2i\beta}}=\frac{e^{-2i\beta}}{re^{i\phi}+\sqrt{r^2 e^{2i\phi}+e^{-2i\beta}}}
		=O(1/r)
	\end{equation}
and in particular, 
	\begin{equation}
	\label{E:trouble-3.25}
		r-\sqrt{r^2-e^{-2i\beta}}=O(1/r),
	\end{equation}
where and henceforth in the rest of this proof the big-$O$ notation  denotes constants, not necessary the same each time,  are \textit{independent of} $w$ and $\phi$, though they may depend on the (corresponding) $A$. Then for all $r$ sufficiently large,
	\begin{equation}
		\label{E:trouble-3.5}
		|r-A(r-\sqrt{r^2-e^{-2i\beta}})|\ge \frac{r}{2}.
	\end{equation}
	
Since $C$ is real, we deduce
	\begin{equation}	
	 \label{E:trouble-3.75}
		\begin{split}
			&
			\Big| (1-C)+
			\frac{r(e^{i\phi}-1)- A(re^{i\phi}-r+\sqrt{r^2-e^{-2i\beta}}
			-\sqrt{r^2 e^{2i\phi}-e^{-2i\beta}})}
			{r-A(r-\sqrt{r^2-e^{-2i\beta}})}\Big|\\
			&\ge \Big|\Im\Big(
			\frac{r(e^{i\phi}-1)- A(re^{i\phi}-r+\sqrt{r^2-e^{-2i\beta}}
			-\sqrt{r^2 e^{2i\phi}-e^{-2i\beta}})}
			{r-A(r-\sqrt{r^2-e^{-2i\beta}})}
			\Big)\Big|\\
			&=\Big|
			\Re\Big(r(e^{i\phi}-1)- A(re^{i\phi}-r+\sqrt{r^2-e^{-2i\beta}}
			-\sqrt{r^2 e^{2i\phi}-e^{-2i\beta}})\Big)\\
			&\qquad\times
				\Im\Big(\frac{1}{r-A(r-\sqrt{r^2-e^{-2i\beta}})}\Big)\\
			&\quad	+
			\Im\Big(r(e^{i\phi}-1)- A(re^{i\phi}-r+\sqrt{r^2-e^{-2i\beta}}
			-\sqrt{r^2 e^{2i\phi}-e^{-2i\beta}})\Big)\\
			&\qquad\times
				\Re\Big(\frac{1}{r-A(r-\sqrt{r^2-e^{-2i\beta}})}\Big)
			\Big|.
		\end{split}
	\end{equation}
We deduce from \eqref{E:trouble-3.25}, 
		\begin{equation}
			\label{E:trouble-5}
			\Re\Big(\frac{1}{r-A(r-\sqrt{r^2-e^{-2i\beta}})}\Big)=\frac1r+O(\frac{1}{r^3})
		\end{equation}
and
		\begin{equation}
			\label{E:trouble-7}
			\Im\Big(\frac{1}{r-A(r-\sqrt{r^2-e^{-2i\beta}})}\Big)=O(\frac{1}{r^3})
		\end{equation}
as $r\to\infty$.

Let us now estimate $|re^{i\phi}-r+\sqrt{r^2-e^{-2i\beta}}
			-\sqrt{r^2 e^{2i\phi}-e^{-2i\beta}}|$, which we rewrite into the form
	\begin{equation}
		\begin{split}
			& \big| re^{i\phi}-r+\sqrt{r^2-e^{-2i\beta}}-\sqrt{r^2 e^{2i\phi}-e^{-2i\beta}}\big|\\
			&=\Big|\frac{(re^{i\phi}-r) \big(\sqrt{r^2-e^{-2i\beta}}+\sqrt{r^2 e^{2i\phi}-e^{-2i\beta}}\big)
			+(r^2-r^2e^{2i\phi})}
			{\sqrt{r^2-e^{-2i\beta}}+\sqrt{r^2 e^{2i\phi}-e^{-2i\beta}}}\Big|\\
%			&=
%			|re^{i\phi}-r|\cdot \Big|
%			\frac{\sqrt{r^2-e^{-2i\beta}}+\sqrt{r^2 e^{2i\phi}-e^{-2i\beta}}-r-re^{i\phi}}
%			{\sqrt{r^2-e^{-2i\beta}}+\sqrt{r^2 e^{2i\phi}-e^{-2i\beta}}}\Big|\\
			&=|re^{i\phi}-r|\cdot \Big|
			\frac{(r-\sqrt{r^2-e^{-2i\beta}})+(re^{i\phi}-\sqrt{r^2 e^{2i\phi}-e^{-2i\beta}})}
			{\sqrt{r^2-e^{-2i\beta}}+\sqrt{r^2 e^{2i\phi}-e^{-2i\beta}}}\Big|\\
			& = r|1-e^{i\phi}|  \cdot \Big|
			 \frac{O(1/r )}
			{r+re^{i\phi}+O(1/r)}\Big|,  
		\end{split}
	\end{equation}				
where we have applied the estimates \eqref{E:trouble-3} and \eqref{E:trouble-3.25} to both the numerator and denominator in the last step above. When $|\phi|<\frac{\pi}{2}$ and $r$ sufficiently large
	\[
		|r+re^{i\phi}+O(1/r)|\ge \sqrt{2} r-O\big(1/r\big)\ge r,
	\]
 hence
	 \begin{equation}
	 	\label{E:trouble-3.11}
		\begin{split}
			& \big| re^{i\phi}-r+\sqrt{r^2-e^{-2i\beta}}-\sqrt{r^2 e^{2i\phi}-e^{-2i\beta}}\big|\\
			&= |1-e^{i\phi}|  \cdot O(1/r) = O\Big(\frac{|\phi|}{r}\Big).
		\end{split}
	\end{equation}
\medskip
Therefore
	\begin{equation}
		\label{E:trouble-3.13}
			\begin{split}
			|\Re\Big(r(e^{i\phi}-1) &- A(re^{i\phi}-r+\sqrt{r^2-e^{-2i\beta}}
			-\sqrt{r^2 e^{2i\phi}-e^{-2i\beta}})\Big)|\\
			&= r|\cos\phi-1|+O\Big(\frac{|\phi|}{r}\Big)=O(r|\phi |),
		\end{split}
	\end{equation}
and
	\begin{equation}
		\label{E:trouble-3.15}
			\begin{split}
				\big| \Im\Big(r(e^{i\phi}-1)- & A(re^{i\phi}-r+\sqrt{r^2-e^{-2i\beta}}
			-\sqrt{r^2 e^{2i\phi}-e^{-2i\beta}})\Big) \big|\\
			&\ge r|\sin\phi|+O\Big(\frac{|\phi|}{r}\Big)\\
			&\ge (2/\pi) r|\phi |+O\Big(\frac{|\phi|}{r}\Big)\ge  r|\phi |/2.
			\end{split}
	\end{equation}
Substituting \eqref{E:trouble-3.13}, \eqref{E:trouble-3.15}, \eqref{E:trouble-5} and \eqref{E:trouble-7} into \eqref{E:trouble-3.75} yields the inequality
 \begin{equation}	
	 \label{E:trouble-3.17}
		\begin{split}
			\Big| (1-&C) +
			\frac{r(e^{i\phi}-1)- A(re^{i\phi}-r+\sqrt{r^2-e^{-2i\beta}}
			-\sqrt{r^2 e^{2i\phi}-e^{-2i\beta}})}
			{r-A(r-\sqrt{r^2-e^{-2i\beta}})}\Big|\\
			&		
			\ge \Big|
			\Im\Big(r(e^{i\phi}-1)- A(re^{i\phi}-r+\sqrt{r^2-e^{-2i\beta}}
			-\sqrt{r^2 e^{2i\phi}-e^{-2i\beta}})\Big)\\
			&\qquad\times
				\Re\Big(\frac{1}{r-A(r-\sqrt{r^2-e^{-2i\beta}})}\Big)\Big|\\
			&\quad
			 - \Big|
			\Re\Big(r(e^{i\phi}-1)- A(re^{i\phi}-r+\sqrt{r^2-e^{-2i\beta}}
			-\sqrt{r^2 e^{2i\phi}-e^{-2i\beta}})\Big)\\
			&\qquad\times
				\Im\Big(\frac{1}{r-A(r-\sqrt{r^2-e^{-2i\beta}})}\Big)\Big|\\
			&\ge  \frac{r|\phi |}{2}\cdot \Big(\frac1r+O(\frac{1}{r^3})\Big)
			  -O(r|\phi|)\cdot O(\frac{1}{r^3})\\
			&=\frac{|\phi |}{2} \Big(1+O\Big(\frac{1}{r^2}\Big)\Big)\ge \frac{|\phi |}{3},
		\end{split}
\end{equation}
as $r\to \infty$.

We substitute \eqref{E:trouble-3.5} and \eqref{E:trouble-3.17} into \eqref{E:trouble-2.5} give
		\begin{equation}
			\label{E:trouble-3.18}
%			\begin{split}
		\big|re^{i\phi}-  Cr-
			A\big[(re^{i\phi}-\sqrt{r^2 e^{2i\phi}-e^{-2i\beta}})-C(r-\sqrt{r^2-e^{-2i\beta}})\big]\big|
			\ge r|\phi |/6
%			\end{split}
		\end{equation}
for $r$ large enough. It follows from \eqref{E:trouble-3.18} that
	\begin{equation}
		\label{E:trouble:3.19}
		\begin{split}
			I_1 &=\int_{-\frac{\pi}{2}}^{\frac\pi2} 
			\frac{d\phi}{\big|re^{i\phi}-Cr-
			A\big[(re^{i\phi}-\sqrt{r^2 e^{2i\phi}-e^{-2i\beta}})-C(r-\sqrt{r^2-e^{-2i\beta}})\big]\big|^\alpha}\\
			       &\le \frac{6^\alpha}{r^\alpha} \int_{-\frac{\pi}{2}}^{\frac\pi2} 
			       \frac{1}{|\phi |^\alpha} \, d\phi=\frac{12^\alpha \cdot \pi^{1-\alpha}}{(1-\alpha)r^\alpha}
		\end{split}
	\end{equation}
for all $r$ sufficiently large. We can convert the integration range  
$(\frac{\pi}{2},\, \frac{3\pi}{2})$ of $I_2$ into that of $(-\frac{\pi}{2},\, \frac{\pi}{2})$  by
	\begin{equation}
		\label{E:I_2}
			\begin{split}			
			&I_2=\\
			&=\int_{-\frac{\pi}{2}}^{\frac\pi2} 
			\frac{d\phi}{\big|re^{i(\phi+\pi)}-Cr-
			A\big[(re^{i(\phi+\pi)}-\sqrt{r^2 e^{2i(\phi+\pi)}-e^{-2i\beta}})-C(r-\sqrt{r^2-e^{-2i\beta}})\big]\big|^\alpha}\\
%			&=\int_{-\frac{\pi}{2}}^{\frac\pi2} 
%			\frac{d\phi}{\big|-re^{i\phi}-Cr-
%			A\big[(-re^{i\phi}-\sqrt{(-r e^{i\phi})^2-e^{-2i\beta}})-C(r-\sqrt{r^2-e^{-2i\beta}})\big]\big|^\alpha}\\
			&=\int_{-\frac{\pi}{2}}^{\frac\pi2} 
			\frac{d\phi}{\big|-re^{i\phi}-Cr-
			A\big[(-re^{i\phi}+\sqrt{(r e^{i\phi})^2-e^{-2i\beta}})-C(r-\sqrt{r^2-e^{-2i\beta}})\big]\big|^\alpha}\\
			&=\int_{-\frac{\pi}{2}}^{\frac\pi2} 
			\frac{d\phi}{\big|re^{i\phi}+Cr-
			A\big[(re^{i\phi}-\sqrt{(r e^{i\phi})^2-e^{-2i\beta}})+C(r-\sqrt{r^2-e^{-2i\beta}})\big]\big|^\alpha}\\
			\end{split}
	\end{equation}
Notice that the integrand of this integral is the same as that of $I_1$ with $C$ replaced by $-C$. We can easily see the above estimate for $I_1$ is still valid when $C$ is replaced by $-C$. Hence 
		\[
			I_2\le \frac{12^\alpha \cdot \pi^{1-\alpha}}{(1-\alpha)r^\alpha}
		\]
for all $r$ sufficiently large. This completes the proof with $E_\alpha:=2\frac{12^\alpha \cdot \pi^{1-\alpha}}{(1-\alpha)}$.
 \end{proof}
\medskip

\subsubsection*{Completion of the proof of Theorem \ref{T:log-lemma}}

Let us choose $R=r\log r$. Then we integrate the inequality \eqref{E:first-part-1} from $0$ to $2\pi$ and apply the inequalities \eqref{E:trouble} and \eqref{E:lemma-3} to yield
\begin{equation}
		\begin{split}
%		\label{E:first-part-1}
			&m\Big( r,\, \frac{(\mathcal{D}_q f)(x)}{f(x)}\Big)
			\le O\Big( \frac{1}{\log r}\Big) \cdot 
			\Big(m\big(r\log r,\, f)+ m\big( r\log r,\, 1/f)\Big)\\
			&\qquad+O\Big(n\big(r\log r, f\big)+ n\big(r\log r, 1/f\big)\Big)+O(1)\\
			&\quad +O(r^\alpha) \cdot \sum_{|c_n|\le r\log r}\int_0^{2\pi}\frac{1}{|re^{i\phi}-c_n|^{\alpha}}\,d\phi\\
			&\quad +O(r^\alpha) \cdot \sum_{|c_n|\le r\log r}\int_0^{2\pi}
			\frac{1}{|x+c(q)q^{-1/2}z^{-1}-q^{-1/2}c_{n}|^\alpha}\,
			d\phi\\
			&\quad+O(r^\alpha) \cdot \sum_{|c_n|\le r\log r}\int_0^{2\pi}
			\frac{1}{|x-c(q)q^{1/2}z^{-1}-q^{1/2}c_{n}|^\alpha}\, d\phi+O(1)\\
			&=O\Big( \frac{1}{\log r}\Big) \cdot 
			\Big(m\big(r\log r,\, f)+ m\big( r\log r,\, 1/f)\Big)\\
			&\qquad+O\Big(n\big(r\log r, f\big)+ n\big(r\log r, 1/f\big)\Big)+O(1).
		 		\end{split}
	\end{equation}
As a result, we obtained the desired estimate \eqref{E:log-lemma} after applying \eqref{E:n-upper} and \eqref{E:n-lower}.

\section{\text{Askey-Wilson type counting functions and proof of Theorem \ref{T:integrated-counting}}}\label{S:proof_counting}
We need to set up some preliminary estimates first.

Let $g:\mathbb{C}\longrightarrow\mathbb{C}$ be a map, \textit{not necessary} entire. Let $f$ be a meromorphic function on $\mathbb{C}$ and $a\in\hat{\mathbb{C}}$, we define the counting function $n(r,\, f(g(x))=a)$ to be the number of $a-$points of $f$, counted according to multiplicity of $f=a$ at the point $g(x)$, in $\{g(x):\, |x|<r\}$.  The \textit{integrated counting function} is defined by
\be
\label{E:counting-fn}
	\begin{split}
	N\big(r,\, f(g(x))=a\big) = & \int_0^r\frac{n\big(t,\, f(g(x))=a\big)-n\big(0,\, f(g(x))=a\big)}{t}\, dt\\
		&\quad +n\big(0,\, f(g(x))=a\big)\log r.
	\end{split}
\ee
 For $z=e^{i\theta}$, we shall write (\ref{E:AW-operator}) in the following notation
\be
\label{E:AW-operator-hat}
	\big(\mathcal{D}_qf\big)(x):=\frac{\breve{f}(q^{\frac12}e^{i\theta})-\breve{f}(q^{-\frac12}e^{i\theta})}
	{\breve{e}(q^{\frac12}e^{i\theta})-\breve{e}(q^{-\frac12}e^{i\theta})}
	=\frac{f(\hat{x}_q)-f(\check{x}_q)}{\hat{x}_q-\check{x}_q}
	=\frac{f(\hat{x})-f(\check{x})}{\hat{x}-\check{x}}	
\end{equation}
where 
\be
\label{E:AW-hat}
	\hat{x}=\hat{x}_q:=\frac{q^{1/2}z+q^{-1/2}z^{-1}}{2},\qquad \check{x}=\check{x}_q:=\frac{q^{-1/2}z+q^{1/2}z^{-1}}{2}.
\ee
Note that the maps $\hat{x}$ and $\check{x}$ are analytic and invertible ($\label{}
	\hat{\check{x}}=\check{\hat{x}}=x$) when $|x|$ is sufficiently large.

The Theorem \ref{T:integrated-counting} is a direct consequence of the following Theorem. 

\begin{theorem}
\label{T:counting-est}
Let $f$ be a meromorphic function of finite logarithmic order $\sigma_{\log}(f)~\ge~1$. Then for each extended complex number $a\in\widehat{\mathbb{C}}$, and each $\varepsilon>0$, we have 
\be
\label{E:counting-est-1}
   N\big(r,\, f(\hat{x})=a\big) =N\big(r,\, f(x)=a\big) +
    O\big((\log r)^{\sigma_{\log}-1+\varepsilon}\big)+O(\log r).
\ee
and similarly,
\be
\label{E:counting-est-2}
    N\big(r,\, f(\check{x})=a\big)= N\big(r,\, f(x)=a\big)+
    O\big((\log r)^{\sigma_{\log}-1+\varepsilon}\big)+O(\log r),
\ee
\be
\label{E:counting-est-3}
    N\big(r,\, f(\hat{\hat{x}})=a\big)= N\big(r,\, f(x)=a\big)+
    O\big((\log r)^{\sigma_{\log}-1+\varepsilon}\big)+O(\log r).
\ee
Here the meaning of $N(r,\,f(\hat{x})=a)$ is interpreted as taking $g(x)=\hat{x}$ mentioned above. The expressions $N\big(r,\, f(\check{x})=a)$ and $N\big(r,\, f(\hat{\hat{x}})=a)$ have similar interpretations.
\end{theorem}
%\smallskip
\begin{proof} We shall only prove the (\ref{E:counting-est-1}) since the (\ref{E:counting-est-2}) and  (\ref{E:counting-est-3}) can be proved similarly. Let $(a_\mu)_{\mu\in\mathbb{N}}$ be a sequence of $a-$points of $f$, counting multiplicities.

Recall that for $|x|$ and hence $|z|$ to be sufficiently large we have $
	\hat{\check{x}}=\check{\hat{x}}=x$. Therefore, there exists a sufficiently large $M>1$ and for $r\ge M$,
	\[
%\label{E:diff-N-growth-1}
\begin{split}
    &N\big(r,\, f(\hat{x})=a\big)
    =\int_0^r\frac{n\big(t,\, f(\hat{x})=a\big)-n\big(0,\, f(\hat{x})=a\big)}{t}\, dt+n\big(0,\, f(\hat{x})=a\big)\log r\\  
     =&\int_0^M\frac{n\big(t,\, f(\hat{x})=a\big)-n\big(0,\, f(\hat{x})=a\big)}{t}\, dt+\int_M
^r\frac{n\big(t,\, f(\hat{x})=a\big)-n\big(0,\, f(\hat{x})=a\big)}{t}\, dt\\
&+n\big(0,\, f(\hat{x})=a\big)\log r\\  
     =&\int_M
^r\frac{n\big(t,\, f(\hat{x})=a\big)}{t}\, dt+O(\log r)\\  
      =&\int_M
^r\frac{n\big(t,\, f(\hat{x})=a\big)-n\big(M,\, f(\hat{x})=a\big)}{t}\, dt+O(\log r)\\  
    =&\sum_{M\le|\check{a}_{\mu}|<r}\log\frac{r}{|\check{a}_{\mu}|}+O(\log r)
\end{split}
	\]
by the definition (\ref{E:counting-fn}). Then
\be
\label{E:diff-N-growth-1}
\begin{split}
    &\big|N\big(r,\, f(x)=a\big)-N\big(r,\, f(\hat{x})=a\big)\big|\\
    &=\bigg|\sum_{0<|a_{\mu}|<r}\log\frac{r}{|a_{\mu}|}+n(0,
    f(x)=a)\cdot\log r-\sum_{M\le|\check{a}_{\mu}|<r}\log\frac{r}{|\check{a}_{\mu}|}+O(\log r)\bigg|\\
      &=\bigg|\sum_{M\le|a_{\mu}|<r}\log\frac{r}{|a_{\mu}|}-\sum_{M\le|\check{a}_{\mu}|<r}\log\frac{r}{|\check{a}_{\mu}|}+O(\log r)\bigg|\\
    &\leq\Bigg|\sum_{
    \substack { M\le|a_{\mu}|<r,\\
   M\le|\check{a}_{\mu}|<r }}
    \left(\log\frac{r}{|\check{a}_{\mu}|}-\log\frac{r}{|a_{\mu}|}\right)\Bigg|
    +\sum_{\substack{
    M\le|\check{a}_{\mu}|<r,\\ |a_{\mu}|\geq r\ \text{or}\ |a_{\mu}|<M
    }} \log\frac{r}{|\check{a}_{\mu}|}\\
    &\quad\quad +\sum_{
    \substack{
   M\le|a_{\mu}|<r,\\
     |\check{a}_{\mu}|\geq r\ \text{or}\ |\check{a}_{\mu}|<M
    }} \log\frac{r}{|a_{\mu}|}+O(\log r)\\
    &\leq \sum_{
    \substack{
    M\le|\check{a}_{\mu}|<r,\\ M\le|a_{\mu}|<r}}
    \bigg|\log\bigg|\frac{a_{\mu}}{\check{a}_{\mu}}\bigg|\bigg|
    +\sum_{
    \substack{
    M\le|\check{a}_{\mu}|<r,\\ |a_{\mu}|\geq r}}
    \log\frac{r}{|\check{a}_{\mu}|}
    +\sum_{
    \substack{
    M\le|a_{\mu}|<r,\\ |\check{a}_{\mu}|\geq r}}
    \log\frac{r}{|a_{\mu}|}+
    O(\log r).
\end{split}
\ee
 Let us write
	\begin{equation}
		\label{E:reminder-1}
			\check{x}=q^{-1/2}x+\eta(x),
	\end{equation}
where
	\begin{equation}
		\eta(x)=\frac{q^{1/2}-q^{-1/2}}{2(x+\sqrt{x^2-1})}
	\end{equation}
which clearly tends to zero as $x\to\infty$. Thus, there exists a constant $h>0$ such that 
\be
\label{E:reminder-1-estimate}
	|\eta(x)|\le h
\ee
for all sufficiently large $|x|$. Thus, it follows from Lemma \ref{L:lemma-2} with $\alpha=1$, (\ref{E:reminder-1}) and  (\ref{E:reminder-1-estimate}) that for $M\le|\check{a}_{\mu}|,\ M\le|a_{\mu}|$,
\be
\begin{split}
	\bigg|\log\bigg|\frac{a_{\mu}}{\check{a}_{\mu}}\bigg|\bigg|
	&=\bigg|\log\bigg|\frac{q^{-1/2}a_\mu+\eta(a_{\mu})}{{a}_{\mu}}\bigg|\bigg|=
	\bigg|\log\bigg|\frac{q^{-1/2}a_\mu}{{a}_{\mu}}\bigg|+\log
	\bigg|\frac{q^{-1/2}a_\mu+\eta(a_{\mu})}{q^{-1/2}{a}_{\mu}}\bigg|\bigg|\\
	&\le|\log|q^{-1/2}||+ \bigg|\frac{\eta(a_{\mu})}{q^{-1/2}{a}_{\mu}}\bigg|+\bigg|\frac{\eta(a_{\mu})}{q^{-1/2}{a}_{\mu}+\eta(a_{\mu})}\bigg|\\
	&\le |\log |q^{-1/2}||+ \frac{h}{|q^{-1/2}||a_\mu|}+\frac{h}{|\check{a}_\mu|}.
\end{split}
\ee
Let
\be
	c= |\log |q^{-1/2}||+|\log |q^{1/2}||+\frac{h|q^{1/2}|}{M}+\frac{h|q^{-1/2}|}{M}+\frac{h}{M}.
\ee
Then,
\be
\label{E:part-I}
	\sum_{
    \substack{
    M\le|\check{a}_{\mu}|<r,\\ M\le|a_{\mu}|<r}}
    \bigg|\log\bigg|\frac{a_{\mu}}{\check{a}_{\mu}}\bigg|\bigg|
\le c\cdot\bigg(\sum_{
    \substack{
    M\le|\check{a}_{\mu}|<r,\\ M\le|a_{\mu}|< r}}1\bigg)
\end{equation}

 Similarly we have 
	\begin{equation}
		\label{E:part-II}
			\sum_{
				\substack{
				M\le|\check{a}_{\mu}|<r,\\ |a_{\mu}|\geq r}}
				\log\frac{r}{|\check{a}_{\mu}|}
				\le \sum_{
				\substack{
				M\le|\check{a}_{\mu}|<r,\\ |a_{\mu}|\geq r}}
				\log\frac{|a_{\mu}|}{|\check{a}_{\mu}|}
				\le c\,\bigg(\sum_{
				\substack{
				M\le|\check{a}_{\mu}|<r,\\ |a_{\mu}|\geq r}}1\bigg).
	\end{equation}
and
	\begin{equation}
		\label{E:part-III}
			\sum_{
			\substack{
			M\le|a_{\mu}|<r,\\ |\check{a}_{\mu}|\geq r}}
			\log\frac{r}{|a_{\mu}|}
			\le
			c\,\bigg(\sum_{
			\substack{
			M\le|a_{\mu}|<r,\\ |\check{a}_{\mu}|\geq r}} 1\bigg).
\end{equation}
Combining the (\ref{E:diff-N-growth-1}), (\ref{E:part-I}), (\ref{E:part-II}) and (\ref{E:part-III}) yields
\be
\label{E:diff-N-growth-2}
    \big|N\big(r,\, f(x)=a\big)-N\big(r,\, f(\hat{x})=a\big)\big|\le
    c\,\bigg(\sum_{M\le|{a}_{\mu}|<r} 1
    +\sum_{M\le|\check{a}_{\mu}|<r} 1 \bigg)+O(\log r).
\ee
For $M\le|\check{a}_{\mu}|<r $ and $r$ large enough and taking into account of the (\ref{E:reminder-1-estimate}),
\be
	|a_\mu|\le
	|q^{1/2}\check{a}_\mu|+|\eta(\check{a}_\mu)|\leq 2|q^{1/2}|r.
\ee
This together with (\ref{E:diff-N-growth-2}) and an inequality similar to  (\ref{E:n-upper}) imply that, for every $\varepsilon>0$,  we have
\be
\begin{split}
	\big|N\big(r,\, f(x)&=a\big)-N\big(r,\, f(\hat{x})=a\big)\big|\\
	&\le
	 c\,\bigg(\sum_{|{a}_{\mu}|<r} 1
    +\sum_{|{a}_{\mu}|<2|q^{1/2}|r} 1 \bigg)+O(\log r)\\
    &=c\Big[\, n\big(r,\, f(x)=a\big)+ n\big(2|q^{1/2}|r,\, f(x)=a\big)\Big]+O(\log r)\\
    &=O\big((\log r)\big)^{\sigma_{\log}-1+\varepsilon}+O(\log r).
\end{split}
\ee
This completes the proof.
\end{proof}

The above estimate should be compared with the estimate of $N(r,\, f(x+\eta))=N(r,\, f(x))+O(r^{\sigma-1+\varepsilon})$ obtained by the authors in \cite[Theorem 2.2]{Chiang:Feng2008} for a meromorphic function of finite order $\sigma$, where $\eta$ is a fixed, though arbitrary non-zero, complex number.

\section{\textrm{Proof of the Second Main theorem \ref{T:2nd-Main-1}}}\label{S:2nd-proof}
We shall follow Nevanlinna's argument\footnote{According to 
 \cite[pp. 238--240]{Nev70} Nevanlinna proved the original version of (\ref{E:2nd-main-1}) for $p=3$ in 1923 and the general case for $p>3$ for entire functions was due to Collingwood in 1924 \cite{Col1ingwood1924}.}
 by replacing the  $f^\prime(x)$ by the $\mathrm{AW-}$operator $\mathcal{D}_qf$ \cite[pp. 238--240]{Nev70}. The methods used in \cite{HK-4} and \cite{HK-3} were based on 
Mohon'ko's theorem (see \cite[p. 29]{Laine1993}).
 We let 
\begin{equation}
\label{E:universal-fn}
	F(x):=\sum_{\nu=1}^p\frac{1}{f(x)-A_\nu}.
\end{equation}
 We deduce from \cite[p. 5]{Hayman1964} that
\begin{equation}
\label{E:m on f}
	m(r,\,F)=m\Big(r,\, F\mathcal{D}_qf\cdot\frac{1}{\mathcal{D}_qf}\Big)
	\le m\Big(r,\, \frac{1}{\mathcal{D}_qf}\Big)+
	m\Big(r,\,\sum_{\nu=1}^p\frac{\mathcal{D}_qf}{f-A_\nu}\Big).
\end{equation}
 On the other hand, for a given $\mu$ amongst $\{1,\,\cdots, p\}$ we write $F$ in the for
\begin{equation}
\label{E:universal-fn-2}
	F(x):=\frac{1}{f(x)-A_\mu}\Big(1+\sum_{
	\substack{\nu=1\\
			\nu\not=\mu}}^p\frac{f(x)-A_\mu}{f(x)-A_\nu}\Big).
\end{equation}
Let $\delta=\min\{|A_h-A_k|,\, 1\}$ whenever $h\not=k$. We follow the argument used by Nevanlinna \cite[pp. 238--240]{Nev70} to arrive at the inequality
\begin{equation}
	m(r,\, F) >\sum_{\mu=1}^p m\big(r,\, A_\mu\big)
	-p\log\frac{2p}{\delta}-\log 3.\notag
\end{equation}
Combining this inequality with (\ref{E:m on f}) yields
\begin{equation}
		m\Big(r,\, \frac{1}{\mathcal{D}_qf}\Big)>
		\sum_{\mu=1}^p m\big(r,\, A_\mu\big)-
		m\Big(r,\,\sum_{\mu=1}^p\frac{\mathcal{D}_qf}{f-A_\mu}\Big)
		-p\log\frac{2p}{\delta}-\log 3.
\end{equation}
 Let us now add $N\big(r,\,1/\mathcal{D}_qf\big)$ on both sides of this inequality and utilizing the first main theorem \cite{Nevanlinna1925} (see also \cite{Hayman1964} and \cite{Nev70}), we deduce
\begin{align}
\label{E:upper-T}
	T(r,\,\mathcal{D}_qf) &=T\Big(r,\,\frac{1}{\mathcal{D}_qf}\Big)+O(1)=m\Big(r,\,\frac{1}{\mathcal{D}_qf}\Big)+N\Big(r,\,\frac{1}{\mathcal{D}_qf}\Big)+O(1)\\
	&>N\Big(r,\,\frac{1}{\mathcal{D}_qf}\Big)+\sum_{\nu=1}^p m\big(r,\, A_\nu\big)-
		m\Big(r,\,\sum_{\nu=1}^p\frac{\mathcal{D}_qf}{f-A_\nu}\Big)
		+O(1).\notag
\end{align}
But is it elementary that
\begin{align}
\label{E:lower-T}
	T(r,\,\mathcal{D}_qf)&=m(r,\,\mathcal{D}_qf)+N(r,\,\mathcal{D}_qf)\\
	&\le N(r,\,\mathcal{D}_qf)+m\big(r,\,{\mathcal{D}_qf}/{f}\big)+m(r,\, f).\notag
\end{align}
Eliminating the $T(r,\,\mathcal{D}_qf)$ from the inequalities (\ref{E:lower-T}) and (\ref{E:upper-T}), adding $m(r,\, f)$ on both sides of the combined inequalities and rearranging the terms yield
\begin{align}
\label{E:final-steps}
		m(r,\, f)+\sum_{\nu=1}^pm(r,\, A_\nu) &\le 2T(r,\,f)
		-\Big(2N(r,\, f)-N(r,\,\mathcal{D}_qf)+N\Big(r,\,\frac{1}{\mathcal{D}_qf}\Big)\Big)\\
		&\quad +m\Big(r,\, \frac{\mathcal{D}_qf}{f}\Big)+
	m\Big(r,\,\sum_{\nu=1}^p\frac{\mathcal{D}_qf}{f-A_\nu}\Big)
	+O(1).\notag
\end{align}
The inequality (\ref{E:2nd-main-1}) now follows by noting the 
\[
	m\Big(r,\,\sum_{\nu=1}^p\frac{\mathcal{D}_qf}{f-A_\nu}\Big)=
	m\Big(r,\,\sum_{\nu=1}^p\frac{(\mathcal{D}_q)(f-A_\nu)}{f-A_\nu}\Big)\le \sum_{\nu=1}^pm\Big(r,\,\frac{(\mathcal{D}_q)(f-A_\nu)}{f-A_\nu}\Big),
\]
%\smallskip

\noindent Theorem \ref{T:log-lemma} and the (\ref{E:remainder}).
\qed
\bigskip
\section{\textrm{Askey-Wilson type Second Main theorem -- Part II: Truncations}}
\label{S:main-results-II}

We recall that in classical Nevanlinna theory, for each element $a$, the counting function $\bar{n}\big(r,\, \frac{1}{f-a}\big)$  counts distinct $a-$points for a meromorphic function $f$ in $\mathbb{C}$ can be written as a sum of integers ``$h-k$" summing over all the points $x$ in $\{|x|<r\}$ at which $f(x)=a$ with multiplicity ``$h$", and where ``$k\, (=h-1)$" is the multiplicity of $f^\prime(x)=0$ where $f(x)=a$. We  define an Askey-Wilson analogue of the $\bar{n}(r,\, f)$. We define the \textit{Askey-Wilson-type counting function} of $f$

\begin{equation}
\label{E:counting}
	\awnf{r}{f=a}=\awnf{r}{\frac{1}{f-a}}
\end{equation}
to be the sum of integers of the form ``$h-k$" summing over all the points $x$ in $\{|x|<r\}$ at which $f(x)=a$ with multiplicity ``$h$", while the $k$ is defined by $k:=\min\{h,\, k^{\prime}\}$ and where ``$k^{\prime}\,$" is the multiplicity of $\mathcal{D}_qf(\hat{x})=0$ at $\hat{x}$. 
%That is, we count ``$h-k$" where $h$ is the multiplicity of the $a-$point of $f(x)$, and $k$ is the multiplicity of the $0-$point of $\mathcal{D}_qf(\hat{x})$. 
%Thus, if $a=0$, then the sum ``$h-k$" is over all $x$ with $|x|<r$ at which $f(x)=0$ with multiplicity ``$h$", and then ``$k$" gives the the multiplicities of the zero of $\mathcal{D}_qf(\hat{x})$ at $x$, respectively. 
Similarly, we define 
\begin{equation}
\label{E:counting-poles}
	\awnf{r}{f}=\awnf{r}{f=\infty}=\awnf{r}{\frac{1}{f}=0}
\end{equation}
to be the sum of integers ``$h-k$", summing over all $x$ in $|x|<r$ at which $(1/f)(x)=0$ with multiplicity ``$h$",  $k:=\min\{h,k^{\prime}\}$ and where ``$k^{\prime}\,$" is the multiplicity of $\mathcal{D}_q(1/f)(\hat{x})=0$ at $\hat{x}$.

% is the multiplicity of the zeros of $\big(1/f\big)(x)=0$ , and $k$ is the multiplicity of zeros of $\mathcal{D}_q\big(1/f\big)(\hat{x})=0$. 

We define the \textit{Askey-Wilson-type integrated counting function} of $f(x)$ by
\begin{equation}
\label{E:int-counting}
		\begin{split}
	\awNf{r}{f=a}=\awNf{r}{\frac{1}{f-a}} &=\int_0^r\frac{\awnf{t}{f=a}-\awnf{0}{f=a}}{t}\, dt\\
			&\qquad +\awnf{0}{f=a}\log r,
		\end{split}
\end{equation}
and
\begin{equation}
\label{E:int-counting-poles}	
	\awNf{r}{f}=\int_0^r\frac{\awnf{t}{f}-\awnf{0}{f}}{t}\, dt
	+\awnf{0}{f}\log r.
\end{equation}
The (\ref{E:int-counting}) and (\ref{E:int-counting-poles}) are respectively the analogues for the $\bar{N}(r,f=a)$ and $\bar{N}(r,f)$ from the classical Nevanlinna theory.

%\section{Interpretation of the term $N_{\mathrm{AW}}(r)$}

We are now ready to state an alternative Second Main Theorem in terms of the $\mathrm{AW-}$type integrated counting function defined above. The theorem could be regarded as a \textit{truncated form} of the original Second Main Theorem, the Theorem \ref{T:2nd-Main-1}.

\begin{theorem}\label{T:2nd-Main-2} Suppose that $f(z)$ is a non-constant meromorphic function of finite logarithmic order $\sigma_{\log}(f)$ as defined in (\ref{E:growth-rate-2}) such that $\mathcal{D}_qf\not\equiv 0$, and let $a_1,\, a_2,\cdots, a_p$ where $p\ge 2$, be mutually distinct elements in $\mathbb{C}$. Then we have, for $r< R$ and for every $\varepsilon>0$,
\begin{equation}
\label{E:2nd-main-2}
	\big(p-1\big) T(r,\, f)\le  \awNf{r}{f}+\sum_{\nu=1}^p \awNf{r}{f=a_\nu} +S_{\log}(r,\, \varepsilon;\, f)
\end{equation}
where
$S_{\log}(r,\, \varepsilon;\, f)=O\big((\log r)^{\sigma_{\log}-1+\varepsilon}\big)+O\big(\log r\big)$ holds for all $|x|=r$ sufficiently large, where $\awNf{r}{f=a_\nu}$ and $\awNf{r}{f}$ are defined by (\ref{E:int-counting}) and (\ref{E:int-counting-poles}), respectively.
\end{theorem}
%\bigskip

This truncated form of the Second Main Theorem leads to new interpretation of Nevanlinna's original defect relation, deficiency, etc, and perhaps the most important of all, is a new type of Picard theorem gears toward the Askey-Wilson operator. We will discuss the functions that lie in the kernel of the $\mathrm{AW-}$operator in \S\ref{S:kernel}. We note that Halburd and Korhonen \cite{HK-2} was the first to give such a truncated form of a Second Main theorem for the difference operator $\Delta f(x)=f(x+\eta)-f(x)$. However, both the formulation of our counting functions $\tilde{N}_{\mathrm{AW}}(r)$ and the method of proof differs greatly from their original argument.
%\bigskip

%\section{\textrm{Proof of Theorem \ref{T:2nd-Main-2}}}
%\label{S:Proof-trunction}

\subsubsection*{Proof of the Theorem \ref{T:2nd-Main-2}}
We are ready to prove the Theorem \ref{T:2nd-Main-2}. 
%\smallskip
%\noindent 
Adding the sum
\be
\label{E:Nev-sum}
	N(r,\, f)+\sum_{\nu=1}^p N(r,\, f=a_\nu)
\ee
on both sides of (\ref{E:2nd-main-1}) and rearranging the terms yields
%\smallskip
\begin{equation}
\label{E:Nev-intermediate-sum}
	\begin{split}
	\big(p- & 1\big)\, T(r,\, f)\\
	& \le N(r,\, f)+  \sum_{\nu=1}^p N(r,\, f=a_\nu)-\mathfrak{N}_{\mathrm{AW}}(r,f)
	+O\big((\log r)^{\sigma_{\log}-1+\varepsilon}\big)
	\end{split}
\end{equation}
%\smallskip
holds for all $|x|=r>0$.
It remains to compare the sizes of (\ref{E:Nev-sum}) and
%\smallskip
\begin{equation}
\label{E:Nev-sum-2}
	\awNf{r}{f}+\sum_{\nu=1}^p \awNf{r}{f=a_\nu}.
\end{equation}
%\smallskip
%\noindent 
Subtracting (\ref{E:Nev-sum-2}) from (\ref{E:Nev-sum}) yields
%\smallskip
\begin{equation}
\label{E:first-differece}
	\big(N(r,\, f)-\awNf{r}{f}\big)+\sum_{\nu=1}^p \big(N(r,\, f=a_\nu)-\awNf{r}{f=a_\nu}\big).
\end{equation}
%\smallskip
%\noindent 
It follows from the definitions of $\awnf{r}{f}$ and $\awnf{r}{f=a_\nu}$ that the difference $n(r,\,f)-\awnf{r}{f}$ enumerates the number of zeros of $\mathcal{D}_q(1/f)(\hat{x})$ at which $(1/f)(x)$ has a zero in the disk $|x|<r$, with due count of multiplicities, while the difference $n(r,\,f=a_\nu)-\awnf{r}{f=a_\nu}$ enumerates the number of zeros of $\mathcal{D}_qf(\hat{x})$ in the disk $|x|<r$ at which $f(x)=a_\nu$, with due count of multiplicities, and those points $x$ in $|x|<r$ that arise from the common $a_\nu-$points of $f(\hat{x})$ and $f(\check{x})$ , respectively. 
%\smallskip
We have 
%\smallskip
\begin{equation}
\label{E:der-1/f}
	\Big(\mathcal{D}_q\frac{1}{f}\Big)(\hat{x})
	=\frac{f(x)-f\big(\hat{\hat{x}})}{(\hat{\hat{x}}-{x})\,
	f(x)\,f\big(\hat{\hat{x}}\big)}=\frac{-(\mathcal{D}_qf)(\hat{x})}{f(x)\,f\big(\hat{\hat{x}}\big)}.
\end{equation}
%\smallskip
Recall that the maps $\hat{x}$ and $\check{x}$ are analytic and invertible ($\label{}
	\hat{\check{x}}=\check{\hat{x}}=x$) when $|x|$ is sufficiently large.
 It follows from (\ref{E:der-1/f}) that the zeros of $(\mathcal{D}_q1/f)(\hat{x})$ originate
 from the poles of $f(x)$, $f\big(\hat{\hat{x}}\big)$ or from the zeros of $(\mathcal{D}_qf)(\hat{x})$. On the other hand, the poles of $(\mathcal{D}_qf)(\hat{x})$ must be amongst the poles of $f(x)$ and/or poles of $f\big(\hat{\hat{x}}\big)$, and  in this case, the multiplicity of zeros of $(\mathcal{D}_q1/f)(\hat{x})$, which is non-negative,  equals to subtracting the multiplicity of poles of $(\mathcal{D}_qf)(\hat{x})$ from the sum of multiplicities of the poles of $f(x)$ and the poles of $f\big(\hat{\hat{x}}\big)$. It follows from this consideration and the definitions of (\ref{E:counting}) and (\ref{E:counting-poles}) that
%\smallskip
%However,  each pole of $\mathcal{D}_qf(\hat{x})$ must either come from a pole of $f(x)$ or from $f\big(\hat{\hat{x}}\big)$ or from both. 
%These poles contributions needs to be removed from the above counting since they are not from  the zeros of $\mathcal{D}_q1/f$. 
\begin{equation}
\label{E:second-difference}
\begin{split}
	&\big(N(r,\, f)-\awNf{r}{f}\big)+\sum_{\nu=1}^p \big(N(r,\, f=a_\nu)-\awNf{r}{f=a_\nu}\big)\\
	&\le N(r,\,f(x))+N\big(r,\,f\big(\hat{\hat{x}}\big)\big)+N\Big(r,\, \frac{1}{\mathcal{D}_qf(\hat{x})}\Big)-N(r,\, \mathcal{D}_qf(\hat{x}))
\end{split}
\end{equation}
%\smallskip
%\noindent
holds. We deduce from Theorem \ref{T:counting-est} and Theorem \ref{T:simple-order} that the followings 
%\smallskip
\begin{equation}
\label{E:shift-est-1}
	N\big(r,\, f(\hat{\hat{x}})\big)= N\big(r,\, f(x)\big)+
    O\big((\log r)^{\sigma_{\log}-1+\varepsilon}\big)+O(\log r);
\end{equation}
%\smallskip
\begin{equation}
\label{E:shift-est-2}
	N\big(r,\, \mathcal{D}_qf(\hat{x})\big)= N\big(r,\, \mathcal{D}_qf(x)\big)+
    O\big((\log r)^{\sigma_{\log}-1+\varepsilon}\big)+O(\log r);
\end{equation}
and
\begin{equation}
\label{E:shift-est-3}
	N\Big(r,\, \frac{1}{\mathcal{D}_qf(\hat{x})}\Big)= N\Big(r,\, \frac{1}{\mathcal{D}_qf(x)}\Big)+
    O\big((\log r)^{\sigma_{\log}-1+\varepsilon}\big)+O(\log r)
\end{equation}
%\smallskip
%\noindent 
hold. Substituting the (\ref{E:shift-est-1}), (\ref{E:shift-est-2}) and (\ref{E:shift-est-3}) into (\ref{E:second-difference}) yields
%\smallskip
\begin{equation}
\label{E:third-difference}
\begin{split}
	&\big(N(r,\, f)-\awNf{r}{f}\big)+\sum_{\nu=1}^p \big(N(r,\, f=a_\nu)-\awNf{r}{f=a_\nu}\big)\\
	&\le 2N(r,\,f(x))+N\Big(r,\, \frac{1}{\mathcal{D}_qf(x)}\Big)-N(r,\, \mathcal{D}_qf(x))+O\big((\log r)^{\sigma_{\log}-1+\varepsilon}\big)+O(\log r)\\
	&=\mathfrak{N}_{\mathrm{AW}}(r,f)+O\big((\log r)^{\sigma_{\log}-1+\varepsilon}\big)+O(\log r),
\end{split}
\end{equation}
%\smallskip
%\noindent 
where the $\mathfrak{N}_{\mathrm{AW}}(r,\, f)$ is given by (\ref{E:remainder}). 	
Combining the (\ref{E:Nev-intermediate-sum}) and (\ref{E:third-difference}) gives the desired inequality (\ref{E:2nd-main-2}).\qed

\section{Askey-Wilson-Type Nevanlinna Defect Relation}\label{S:deficient}
We recall Nevanlinna's original \textit{deficiency}, \textit{multiplicity index} and \textit{ramification index} are defined, respectively, by $\delta(a)=1-\varlimsup_{r\to\infty}{N(r,\,f=a)}/{T(r,\, f)}$, $\vartheta(a)=\vartheta(a,\, f)=\varliminf_{r\to\infty}({N(r,\,f=a)-\bar{N}(r,\,f=a)})/{T(r,\, f)}$
and $\Theta(a)=\Theta(a,\, f)=1-\varlimsup_{r\to\infty}{\bar{N}(r,\,f=a)}/{T(r,\, f)}$. Nevanlinna's second main theorem implies
%\smallskip
\be
	\sum_{a\in\widehat{\mathbb{C}}}\big(\delta(a)+\vartheta(a)\big)\le
	\sum_{a\in\widehat{\mathbb{C}}}\Theta(a)\le 2.
\ee
%\smallskip
%\noindent 
We define the
$\mathrm{AW-}$\textit{multiplicity index} and $\mathrm{AW-}$\textit{deficiency} by
\be
	\vartheta_{\mathrm{AW}}(a)=\vartheta_{\mathrm{AW}}(a,\, f)=\varliminf_{r\to\infty}\frac{N(r,\, f=a)-\awNf{r}{f=a}}{T(r,\, f)},
\ee
and
\be
	\Theta_{\mathrm{AW}}(a)=\Theta_{\mathrm{AW}}(a,\, f)=1-\varlimsup_{r\to\infty}\frac{\awNf{r}{f=a}}{T(r,\, f)}
\ee
respectively.
%\smallskip
It follows from the definition of $\awNf{r}{f=a}$ that we have the relationship
%\smallskip	
	\[
		0\le \vartheta_{\mathrm{AW}}(a,\, f)\le \Theta_{\mathrm{AW}}(a,\, f)\le 1.
	\]
%\smallskip
%\noindent 
Dividing the inequality (\ref{E:2nd-main-2}) in Theorem \ref{T:2nd-Main-2} by $T(r,\, f)$ and rearranging the terms yield
\be
	1-\frac{\awNf{r}{f=\infty}}{T(r,\, f)}+\sum_{\nu=1}^p\bigg(1-\frac{\awNf{r}{f=a_\nu}}{T(r,\, f)}\bigg)\le 2+\frac{S_{\log}(r,\,\varepsilon;\,f)}{T(r,\,f)}.
\ee
Taking limit infimum on both sides of the above inequality as $r\to +\infty$ yields the following theorem.\medskip

\begin{theorem}\label{T:2nd-Main-3} Suppose that $f(z)$ is a transcendental meromorphic function of finite logarithmic order, such that $\mathcal{D}_qf\not\equiv 0$. Then
\be
\label{E:2nd-main-3}
	\sum_{a\in\widehat{\mathbb{C}}}\big(\delta(a)+\vartheta_{\mathrm{AW}}(a)\big)\le
	\sum_{a\in\widehat{\mathbb{C}}}\Theta_{\mathrm{AW}}(a)\le 2.
\ee
\end{theorem}

\begin{remark} We note that Chern showed in \cite[Theorem 8.1]{Chern2005} that for entire function $f$ of finite logarithmic order growth with its log-order $\sigma_{\log}$ and lower order $\nu=\liminf_{r\to\infty}\displaystyle {\log T(r,\, f)}/{\log\log r}$ satisfying $\sigma_{\log}-\mu<1$, must have
	\[
		N(r,\, f=a)\sim T(r,\, f).
	\]
 This implies that the two quantities $\theta_{\mathrm{AW}}(a)$ and $\Theta_{\mathrm{AW}}(a)$ are identical for any finite $a$.
\end{remark}
%\smallskip

\begin{definition} \label{D:deficient}
We call a complex number $a\in\mathbb{C}$ an 
	\smallskip
	\begin{enumerate}
		\item $\mathrm{AW-}$\textit{Picard value} if $\awnf{r}{f=a}=O(1)$ (Note that this is equivalent to $\awNf{r}{f=a}=O(\log r)$),
		\bigskip
		\item $\mathrm{AW-}$\textit{Nevanlinna deficient value} if $\Theta_{\mathrm{AW}}(a)>0$.
	\end{enumerate}
\end{definition}
%\bigskip

We remark that $a$ is an $\mathrm{AW-}${Picard value} of $f$ means that except for at most a finite number of points, the multiplicity ``$h$" of $f(x)=a$ at $x$ is not larger than ``$k^{\prime}$", the multiplicity of $\mathcal{D}_qf(\hat{x})=0$ at $\hat{x}$. We also note that for a transcendental function $f$ to have $\mathrm{AW-}${Picard value} $a$ implies that $\Theta_{\mathrm{AW}}(a)=1$.
%\medskip

We immediately deduce from Theorem \ref{T:2nd-Main-3} the following $\mathrm{AW-}$type Picard theorem for finite logarithmic order meromorphic functions.
%\bigskip

\begin{theorem} \label{T:AW-Picard} Let $f$ be a meromorphic function with finite logarithmic order, and that $f$ has three distinct $\mathrm{AW-}$Picard values. Then $f$ is either a rational function or $f\in \ker \mathcal{D}_q$.
\end{theorem}
\smallskip
\qed

%\medskip
We also deduce from the Theorem \ref{T:2nd-Main-3} the following 
%\smallskip

\begin{theorem} Let $f$ be a transcendental meromorphic function with finite logarithmic order. Then $f$ has at most a countable number of  $\mathrm{AW-}$Nevanlinna deficient values.
\end{theorem}
%\smallskip

\begin{remark}  Suppose $f$ is  a meromorphic function of finite logarithmic order, and an extended complex number $A$. If there exists $0<q<1$, $f(x)$ has most finitely many $A-$points, or there exists positive integer $J$, complex numbers $a_j\  (1\le j\le J)$, and each $j$ associates an integer $d_j\ (1\le j\le J)$, such that, except for at most finitely many points, the $A-$points of $f(x)$ situate at the sequences

\be
\label{E:exceptional-sequences}
	\frac12\big(a_jq^n+a_j^{-1}q^{-n}\big),\quad j=1,\, 2,\,\cdots,\, J,\quad n=1,\, 2,\, 3,\cdots,
\ee
%\medskip

\noindent with multiplicity $d_j,\,j=1,\, 2,\,\cdots,\, J$. Then it is easy to check that $A$ is an $\mathrm{AW}-$Picard exceptional value.
\end{remark}
%\medskip

We note that the definition of $\mathrm{AW}-$type exceptional values includes the classical definition of Picard exceptional value, namely that the meromorphic $f(x)$ equals to $A$ at most finitely many times. It is known that for meromorphic functions of finite logarithmic order of growth, one needs only two Picard-exceptional values in order for $f$ to reduce to a (genuine) constant \cite{Chern2005}. Thus, when interpreted in the classical (most restricted ) setting, one needs only two classical Picard exceptional values in order for $f$ to reduce to a constant. However, as exhibited in earlier example of the generating function discovered by Rogers (\ref{E:generating}), when interpreted in the Askey-Wilson (most general) setting, one needs three  $\mathrm{AW-}$Picard exceptional values in order to conclude that $f\in \ker \mathcal{D}_q$.
%\medskip

\section{\text{Askey-Wilson type Nevanlinna deficient values}}\label{S:examples}
%\section{Rational $\mathrm{AW-}$type Nevanlinna deficient values}

We construct two kinds of examples below that both give  $\mathrm{AW-}$ Nevanlinna deficiencies at $x=0$ as arbitrary rational number, that is, $\Theta_{\mathrm{AW}}(0)=\frac{m}{n}>0$. The first category of examples is based on the definition \ref{D:deficient} that if the pre-image of zero for certain function $f$ lies on an infinite sequence of the form (\ref{E:AW-except}), then 
$\Theta_{\mathrm{AW}}(0)=1$. Our second category example is based on constructing multiple zeros interpreted in the conventional sense (that is, in the sense of differentiation). 
%\medskip

In our first example below all the zeros are simple when interpreted in the conventional sense, but when some of them are grouped into an infinite union of certain finite sequences are in fact multiple zeros when interpreted in the sense of Askey-Wilson.
%\smallskip

\begin{example} Let $n$ be a positive integer. Then the function
	\begin{equation}
		\label{E:1-1/n}
			f_{\frac{n-1}{n}}
(x)=\prod_{k=0}^{n-1} (q^ke^{i\theta},\, q^k e^{-i\theta};\,q^{n+1})_\infty
	\end{equation}
has 
		\[
			\Theta_{\mathrm{AW}}(0)=\frac{n-1}{n},
		\]
according to the definition of $\awNf{r}{f=0}$ in (\ref{E:int-counting}).
%\smallskip

\begin{proof} We first note that for arbitrary $j$, 
	\[
		(q^ke^{i\theta},\, q^ke^{-i\theta};\,q^j)_\infty
		=\prod_{\nu=1}^\infty(q^{2k+2j(\nu-1)}+1)\Big(1-\frac{x}{\frac12(q^{k+j(\nu-1)}+q^{-k-j(\nu-1)})}\Big).
	\]
Let $n(r)$ denote the number of zeros of $(q^ke^{i\theta},\, q^ke^{-i\theta};\,q^j)_\infty$ in $|x|<r$ for $k\le j$. Then we clearly can find constants $c_1$ and $c_2$ such that 
	\begin{equation}
			\label{E:n}
			c_1\log r\le n\big(r,\, (q^ke^{i\theta},\, q^ke^{-i\theta};\,q^j)_\infty=0\big)\le c_2 \log r.
	\end{equation}
Hence there are constants $C_1$ and $C_2$ such that
\smallskip

	\begin{equation}
	\label{E:counting_1}
			C_1(\log r)^2\le N\big(r,\, (q^ke^{i\theta},\, q^ke^{-i\theta};\,q^j)_\infty=0\big)\le C_2 (\log r)^2.
	\end{equation}
Similarly

\begin{equation}
	\label{E:counting_2}
			D_1(\log r)^2\le \awNf{r}{(q^ke^{i\theta},\, q^ke^{-i\theta};\,q^j)_\infty=0}\le D_2 (\log r)^2,
	\end{equation}
for some positive constants $D_1$ and $D_2$. On the other hand, it also follows from the definition \ref{D:deficient} and a simple observation from (\ref{E:1-1/n}) that 
	\begin{equation}
		\label{E:upper-lower}
		\frac1n\, \big(1-o(1)\big)\le \frac{\awNf{r}{\prod_0^{n-1}(q^ke^{i\theta},\, q^ke^{-i\theta};\,q^j)_\infty=0}}{ N\big(r,\, \prod_0^{n-1} (q^ke^{i\theta},\, q^ke^{-i\theta};\,q^j)_\infty=0\big)}
		\le \frac1n\, (1+o\big(1\big))
	\end{equation}
as $r\to +\infty$.
	\[
		\begin{split}
		\log \big|(q^ke^{i\theta},\, q^ke^{-i\theta};\,q^j)_N\big|
		&\le \sum_{\nu=1}^N \log \Big(1+\left|\frac{x}{\frac12(q^{k+j(\nu-1)}+q^{-k-j(\nu-1)})}\right|\Big)\\
		&\qquad +	\sum_{\nu=1}^N \log |1+q^{2k+2j(\nu-1)}|\\
		&\le \int_0^r\log\Big(1+\frac{|x|}{t}\Big)\, dn(t)+\sum_{\nu=1}^N |q|^{2k+2j(\nu-1)}\\
		&=n(r)\log\Big(1+\frac{|x|}{r}\Big)+
			\int_0^r \frac{|x|\, n(t)}{t(t+|x|)}\, dt+|q|^{2k}\sum_{\nu=1}^N |q|^{2j(\nu-1)}.
\end{split}
	\]
Taking limits of $N\to +\infty$ on both sides of the above inequality and with  reference to (\ref{E:n}) yield 
			\[
			\begin{split}
		\log \big|(q^ke^{i\theta},\, q^ke^{-i\theta};\,q^j)_\infty\big|
		& \le |x|\int_0^\infty \frac{\, n(t)}{t(t+|x|)}\, dt+C\\
		&\le \int_0^{|x|} \frac{n(t)}{t}\, dt+|x| \int_{|x|}^\infty\frac{n(t)}{t^2}\, dt+\frac{|q|^{2k}}{1-|q|^{2j}}.
			\end{split}
		\]
It is a standard technique to apply integration-by-parts repeatedly on the second integral above (see e.g. \cite[Lemma 5.1]{BP2007}, \cite[Theorem 7.1]{Chern2005}) to yield, for some positive constant $B$ that
	\begin{equation}
		\label{E:maximum}
		\begin{split}
		T\big[r,\, (q^ke^{i\theta},\, q^ke^{-i\theta};\,q^j)_\infty\big]
	&\le\log M\big(r,\, (q^ke^{i\theta},\, q^ke^{-i\theta};\,q^j)_\infty\big) \\
		&\le
		N\big(r, (q^ke^{i\theta},\, q^ke^{-i\theta};\,q^j)_\infty=0\big)+\\
			&\quad+
		[O(\log r)^{\sigma-1+\varepsilon}+O(\log r)^{\sigma-2+\varepsilon}\cdots+O(\log r)]+B\\
		&\le
		N\big(r, (q^ke^{i\theta},\, q^ke^{-i\theta};\,q^j)_\infty=0\big)\\
		&\quad + [O(\log r)^{\sigma-1+\varepsilon}+O(\log r)]+B,
		\end{split}
	\end{equation}
where $\sigma=\sigma_{\log}=2$. It follows from (\ref{E:maximum}) and (\ref{E:1-1/n}) that
	\[
		N(r,\, f)\le T(r,\, f)\le N(r,\, f)+O(\log^{1+\varepsilon} r).
	\]
Hence it follows from (\ref{E:upper-lower})
	\[
		\Theta_{\mathrm{AW}}(0)=1-\frac1n
	\]
as asserted.
\end{proof}
\end{example}

\begin{example} Applying similar idea used in the last example, we can show that the function
	\[
		f_{\frac1n}(x)=
		\prod_{k=0}^{n-1} (q^{2k}e^{i\theta},\, q^{2k}e^{-i\theta};\, q^{2n-1})_\infty
	\]
has
	\[
		\Theta_{\mathrm{AW}}(0)=\frac1n.
	\]
\end{example}
Again, one can generalise the above idea to construct an entire function with arbitrary rational $\mathrm{AW}-$Nevanlinna deficient value.

\begin{example} Let $m,\, n$ be positive integers such that $1\le m< n$. 
	\[
		\begin{split}
		f_{\frac{m}{n}}(x) &=
		\prod_{k=0}^{m-1} (q^{k}e^{i\theta},\, q^{k}e^{-i\theta};\, q^{2n-m})_\infty\\
		&\qquad \qquad \times \prod_{k^\prime=1}^{n-m} (q^{m+2k^\prime-1}e^{i\theta},\, q^{m+2k^\prime-1}e^{-i\theta};\, q^{2n-m})_\infty.
		\end{split}
	\]
Then 
	\[
		\Theta_{\mathrm{AW}}(0)=\frac{m}{n}.
	\]
\end{example}
\begin{remark} We note that the construction of entire functions with rational AW-deficient value above are far from being unique. For example, each of the following two functions has  AW-deficient value equal to $\frac12$:
	\[
		f_{\frac12}(x)=(e^{i\theta},\, e^{-i\theta};\, qe^{i\theta}, \, qe^{-i\theta};\, q^3)_\infty
	\]
	\[
		g_{\frac12}(x)=(e^{i\theta},\, e^{-i\theta};\, qe^{i\theta}, \, qe^{-i\theta};\,
		q^2e^{i\theta},\, q^2e^{-i\theta};\, q^4e^{i\theta}, \, q^4e^{-i\theta};\, q^5)_\infty.
	\]
\end{remark}
\medskip

We next consider an example of different type. 

\begin{example} Let $M,\, N$ be non-negative integers such that $M>N$. 

		\[
			f_{\frac{M}{N}}(x)=[(e^{i\theta},\, e^{-i\theta};\,q)_\infty]^M [(qe^{i\theta},\, qe^{-i\theta};\,q)_\infty]^N.
		\]
It follows from the above construction of $f$ and the definition of $\awNf{r}{f=a}$ in (\ref{E:int-counting}) that 
	\[
		\Theta_{\mathrm{AW}}(0)=1-\frac{M-N}{M+N}=\frac{2N}{M+N}.
	\]
	
\begin{proof}
	We skip the derivation.
\end{proof}
\end{example}

\section{\textrm{The Askey-Wilson kernel and theta functions}}\label{S:kernel}
Here we give an alternative and self-contained characterisation of the functions that lie in the kernel of the $\mathrm{AW-}$operator without appealing to elliptic functions.
A  way to look at the classical small Picard theorem is that when a meromorphic function omits three values in $\mathbb{C}$, then the function belongs to the kernel of conventional differential operator, that is, it is a constant. We now show that the ``constants" for the $\mathrm{AW-}$operator $\mathcal{D}_q$ are very different.

\begin{theorem}\label{T:ker-entire} Let $f(x)$ be an entire function in $\mathbb{C}$ that satisfies $(\mathcal{D}_qf)(x)\equiv 0$. Then $f(x)=c$  throughout $\mathbb{C}$ for some complex number $c$.
\end{theorem}

\begin{proof}
We recall our initial assumption that $|q|<1$. Let $f$ be an entire function that lies in the kernel of $\mathcal{D}_q$, that is, ${\mathcal{D}}_qf\equiv 0$. Hence for every complex number $z\not=0$, we have
	\[
		f\Big(\frac{z+1/z}{2}\Big)=f\Big(\frac{qz+q^{-1}/z}{2}\Big).
	\]
We deduce easily by induction that, for every integer $n$, the equality
	\begin{equation} 
		\label{E:propgation}
			f\Big(\frac{z+1/z}{2}\Big)=f\Big(\frac{q^nz+q^{-n}/z}{2}\Big).
	\end{equation}
holds. 

For each $x\in\mathbb{C}$, we can find a non-zero $z\in\mathbb{C}$ such that
	\[
		x=\frac{z+1/z}{2}.
	\]
Let
		\begin{equation}
			m=\Big[\frac{\log |z|}{\log\Big|\frac{1}{q}\Big|}\Big]
		\end{equation}
be in $\mathbb{Z}$, where $[\alpha]$ denote the integral part of real number $\alpha$.  Then we have 
		\begin{equation} 
			1\leq |q^mz |\leq \Big|\frac{1}{q}\Big|.
		\end{equation} 
Noting that the real-valued function $t+\frac1t$ is increasing for $t\ge 1$, we deduce that
		\begin{equation} 
			\Big|\frac{q^mz+q^{-m}/z}{2}\Big|\leq \frac{|q|+{|q|}^{-1}}{2}
		\end{equation} 
holds. Therefore 
	\begin{equation} \label{E:representative}
		|f(x)|=\Big|f\Big(\frac{q^mz+q^{-m}/z}{2}\Big)\Big|\leq M:=
		\max_{|y|\leq\frac{|q|+{|q|}^{-1}}{2}}|f(y)|.
	\end{equation}
Since $x$ is arbitrary, we
have shown that $f(x)$ is a bounded entire function and so it must reduce to a
constant function. 
\end{proof}

The example (\ref{E:kernel}) of meromorphic function that satisifes $\mathcal{D}_q f\equiv 0$ 
given by Ismail \cite[p. 365]{Ism2005} has finite logarithmic order. We call these functions $\mathrm{AW-}$\textrm{constants}. For the sake of similicity, we adopt Ismail's notation that (\ref{E:kernel}) can be rewritten in the form
\begin{equation}
	\label{E:kernel-2}
	f(x)=(\cos\theta-\cos\phi)\, \frac{\phi_\infty (\cos\theta;\, qe^{i\phi})\, \phi_\infty (\cos\theta;\, qe^{-i\phi})}{\phi_\infty (\cos\theta;\, q^{1/2}e^{i\phi})\,\phi_\infty (\cos\theta;\, q^{1/2}e^{-i\phi})}.
\end{equation}
 We show below that all functions in the $\ker\mathcal{D}_q$ are essentially functions made up of this form.

\begin{theorem}\label{T:ker-mero} Let $f(x)$ be a meromorphic function in $\mathbb{C}$ that satisfies\\ $(\mathcal{D}_qf)(x)\equiv 0$. Then there
exist a nonnegative integer $k$ and complex numbers $a_1, a_2,
\cdots, a_k$; $b_1, b_2, \cdots, b_k$; $C$ such that 
\begin{align}\label{E:q-constant}
	f(x) &=C\prod_{j=1}^k\frac{\phi_\infty(\cos\theta; a_j)\, \phi_\infty(\cos\theta;
{q}/{a_j})}{\phi_\infty(\cos\theta; b_j)\, \phi_\infty(\cos\theta; {q}/{
	b_j})}\notag\\
		&=C\prod_{j=1}^k\frac
	{(a_je^{i\theta},\, a_je^{-i\theta};\, q)_\infty\,(q/a_je^{i\theta},\, q/a_je^{-i\theta};\, q)_\infty}  
	{(b_je^{i\theta},\, b_je^{-i\theta};\, q)_\infty\,(q/b_je^{i\theta},\, q/b_je^{-i\theta};\, q)_\infty} 
,
\end{align}
where $x=\frac12(e^{i\theta}+e^{-i\theta})$, and $\phi_\infty(x;\, a):=(ae^{i\theta},\, ae^{-i\theta};\, q)_\infty$. That is, each $\mathrm{AW-}$ constant assumes the form $\mathrm{(\ref{E:q-constant})}$.
\end{theorem}
 
\begin{proof} 
For any complex numbers $a$ and $b$, let 

\be\label{E:sample_kernel_fn}
	f_{a,b}(x):=\frac{\phi_\infty(x;\, a)\,\phi_\infty(x;\,{q}/{a})}
		{\phi_\infty(x;\, b)\,\phi_\infty(x;\, {q}/{ b})}.
\ee 
 It is routine to check that
 
\be \label{E:kernel_representative}
		({\mathcal{D}}_qf_{a,b})(x)\equiv 0
\ee
holds. Without loss of generality, we assume that $f(x)\not\equiv0$. Let us suppose that $x_0=\displaystyle ({z_0+1/z_0})/{2}$ be a zero (resp. pole) of $f$, then so is each point that belongs to the sequence $\{\displaystyle({q^nz_0+1/q^{n}/z_0})/2\}_{n\in\mathbb{Z}}$ in view of (\ref{E:propgation}) with the same multiplicity. We introduce an
\textit{equivalence relation} on all the zeros (resp. poles) of $f$. For
$x_1=({z_1+1/z_1})/2$ and $x_2=({z_2+1/z_2})/2$, if
there exists an integer $n$ such that $z_1=z_2q^n$, then we say $x_1$
and $x_2$ is \textit{equivalent} to each other. We denote the class of zeros
(resp. poles) which is equivalent to $x_0$ by $\{x_0\}$. Clearly every zero
(resp. pole) in an equivalent class has the same multiplicity.  It follows from 
(\ref{E:representative}), that for every equivalent class of zeros (resp. poles), there exists an element $x^\prime$, say, such that $|x^\prime|\leq \displaystyle({q+1/q})/{2}$. Since $f$ is
meromorphic, it has at most finite number of zeros and poles in the disc
$\{|x|\leq\displaystyle ({q+1/q})/{2}\}$, and thus $f$ has at most finitely many equivalent classes of zeros (resp. poles) in the complex plane. Denote
by $\{a_1\}, \{a_2\},\cdots, \{a_l\}$ the equivalent classes of zeros
of $f$ and by $\{b_1\}, \{b_2\},\cdots, \{b_k\}$ the
equivalent classes of poles of $f$, list according to their multiplicities.

We now distinguish two cases:
\begin{description}
\item[Case A] $l\geq k$. Set 
\be
		g(x)=\prod_{j=1}^k\frac{\phi_\infty(x;\, a_j)\,\phi_\infty(x;\,
{q}/{a_j})}{\phi_\infty(x;\, b_j)\,\phi_\infty(x;\, {q}/{
b_j})}.
\ee 
Then it follows from the same principle as in (\ref{E:kernel_representative})
that it again satisfies 
\smallskip
\be 
		({\mathcal{D}}_qg)(x)\equiv0.
\ee 
 Notice that ${f(x)}/{g}(x)$ is now an entire function, and it also satisfies 
\be
	\Big({\mathcal{D}}_q\frac{f}{g}\Big)(x)\equiv 0.
\ee 
Theorem \ref{T:ker-entire} implies that we have 
	\be
		\frac{f(x)}{g(x)}\equiv C.
	\ee 
This establishes (\ref{E:q-constant}) as required.

\item[Case B] $k\geq l$. We consider the meromorphic function
${1}/{f(x)}$ instead. So it satisfies 
\be
	\Big({\mathcal{D}}_q\frac{1}{f}\Big)(x)\equiv 0.
\ee 
and has equivalent classes of zeros $\{b_1\}, \{b_2\},\cdots, \{b_k\}$ and the equivalent classes of poles $\{a_1\}, \{a_2\},\cdots, \{a_l\}$, listed according to their multiplicities. Notice that $1/f(x)$ falls into the category considered in case A above, so that
 \be
	\frac{1}{f(x)}=C\prod_{j=1}^l\frac{\phi_\infty(x;\, b_j)\,\phi_\infty(x;\,
	{q}/{b_j})}{\phi_\infty(x;\, a_j)\, \phi_\infty(x;\, {q}/{
a_j})}.
\ee 
Hence
 \be
	f(x)=\frac{1}{C}\prod_{j=1}^l\frac{\phi_\infty(x;\, a_j)\,\phi_\infty(x;\,
{q}/{a_j})}{\phi_\infty(x;\, b_j)\,\phi_\infty(x; {q}/{b_j})}
\ee 
as required.
\end{description}
\end{proof}

We now explore the fact that the space $\ker \mathcal{D}_q$ is a linear space. This allows us to derive a number of interesting \textit{relationships} amongst some arbitrary combinations of products of $\phi_\infty(x;\, a)\, \phi_\infty(x;\, q/a)$ can be represented by a single such product. We shall show that many well-known identities about Jacobi theta functions can be expressed in the forms that fit those relationships.

\begin{theorem}\label{T:Kernel-I} Given positive integer $k$ and complex
numbers $a_j,\, C_j$, \\ $j=1, 2,\cdots\, k$, there exist complex numbers $b$
and $C$ such that 
	\be 
		\label{E:kernel_identity_I}
			\sum_{j=1}^kC_j\,\phi_\infty(x;\, a_j)\,\phi_\infty(x;\, {q}/{a_j})=C\,\phi_\infty(x;\, b)\,\phi_\infty(x;\,{q}/{b}).
	\ee
Alternatively, we express this equation in $q-$rising factorial notation as
	\begin{equation}
		\label{E:new_identity_I}
			\begin{split}
			\sum_{j=1}^kC_j \,(a_je^{iz},\, a_je^{-iz};\, q)_\infty
			(q/a_je^{iz},\, & q/a_je^{-iz};\, q)_\infty\\
			&=C\, (be^{iz},\, be^{-iz};\, q)_\infty
			(q/be^{iz},\, q/be^{-iz};\, q)_\infty.
			\end{split}
	\end{equation}
\end{theorem}

\begin{proof} Let $d$ be a complex number such that $d\not=a_j$ for $1\leq
j\leq k$. 
Set 
\smallskip
	\begin{equation} \label{E:kernel_I_LHS}
		f(x)=\sum_{j=1}^k C_j\,\frac{\phi_\infty(x;\, a_j)\,\phi_\infty(x;\,{q}/{a_j})}{\phi_\infty(x;\, d)\,\phi_\infty(x;\, {q}/{d})}.
	\end{equation}

Then we know from (\ref{E:sample_kernel_fn}) that 
	\begin{equation}
			({\mathcal{D}}_qf)(x)\equiv 0.
	\end{equation}
Hence $f(x)$ lies in the kernel of $\mathcal{D}_q$. We deduce from Theorem \ref{T:ker-mero} that there exist a nonnegative integer $m$ and complex numbers $C$,  $c_1,\, c_2, \cdots,\, c_k$; $d_1,\, d_2, \cdots,\, d_k$, listed according to their multiplicities, such that 
\begin{equation}
		f(x)=C\, \prod_{j=1}^m\frac{\phi_\infty(x;\, c_j)\,\phi_\infty(x;\,
{q}/{c_j})}{\phi_\infty(x;\, d_j)\,\phi_\infty(x;\, {q}/{
d_j})}.
\end{equation}
However, $f(x)$ can only have a single equivalent class $\{d\}$ of poles of multiplicity one. We deduce $m=1$ and $d_1=d$, let $b=c_1$, we have
 \begin{equation}\label{E:kernel_I_RHS}
	f(x)=C\frac{\phi_\infty(x;\, b)\,\phi_\infty(x;\,
	{q}/{b})}{\phi_\infty(x;\, d)\,\phi_\infty(x;\, {q}/{ d})}.
	\end{equation}
Combining (\ref{E:kernel_I_LHS}) and (\ref{E:kernel_I_RHS}) yields (\ref{E:kernel_identity_I}).
\end{proof}

Similarly we obtain the following extension but we omit its proof.

\begin{theorem}\label{T:Kernel-II} Given nonnegative integers $k,\,  m$ and 
complex numbers $a_{ij},\,  C_j$, $i=1,\, 2,\, \cdots,\, m;\, j=1,\, 2,\,\cdots\, k$, there exist complex numbers $c_1,\, c_2,\, \cdots,\, c_m$ and  $C$ such that 
	\be \label{E:kernel_identity_II}
		\sum_{j=1}^kC_j\,\prod_{i=1}^m\phi_\infty(x;\,
		a_{ij})\, \phi_\infty(x;\,{q}/{a_{ij}})=C\,\prod_{i=1}^m\phi_\infty(x;\,
c_i)\,\phi_\infty(x;\, {q}/{c_i}).
	\ee
	Alternatively, we express this equation in $q-$rising factorial notation as
	\begin{equation}
		\label{E:new_identity_II}
			\begin{split}
			  \sum_{j=1}^kC_j\, \prod_{i=1}^m & \big(a_{ij}e^{iz},\,  a_{ij}e^{-iz};\, q\big)_\infty \big(q/a_{ij}e^{iz},\,  q/a_{ij}e^{-iz};\, q\big)_\infty\\
			&= C \prod_{i=1}^m \big(c_ie^{iz},\, c_ie^{-iz};\, q\big)_\infty
\big((q/c_i)e^{iz},\, (q/c_i)e^{-iz};\, q\big)_\infty.
			\end{split}
	\end{equation}			
\end{theorem}

Let  us write $q=e^{i\pi \tau}$ where $\Im (\tau)>0$. Hence  $|q|<1$. Identifying the theta functions in the notation of infinite $q-$product with \cite[pp. 469--473]{Whittaker:Watson1927}:
	\begin{itemize}
		\item
			\begin{equation}\label{E:theta_4}
				\vartheta_4(z,\, q)=(q^2;\, q^2)_\infty\, (q\,e^{2iz},\, q^2)_\infty\,  (q\, e^{-2iz},\, q^2)_\infty,
			\end{equation}
		\item	
			\begin{equation}\label{E:theta_3}
				\vartheta_3(z) 
			=\vartheta_4(z+\frac{\pi}{2})
			=(q^2;\, q^2)_\infty \, (q\,e^{2iz+i\pi},\, q\,e^{-2iz-i\pi};\,q^2)_\infty,
			\end{equation}
		\item	
			\begin{equation}\label{E:theta_1}
				\vartheta_1(z)/(-iq^{1/4}e^{iz})= \vartheta_4(z+\frac{\pi\tau}{2})
				= (q^2;\, q^2)_\infty \, (q^2\,e^{2iz};\,q^2)_\infty\,  (e^{-2iz};\,q^2)_\infty,
			\end{equation}
		and finally
		\item 
			\begin{equation}\label{E:theta_2}
				\vartheta_2(z)=\vartheta_1(z+\frac12\pi)
			=q^{\frac14}e^{iz}\,(q^2;\, q^2)_\infty\,(-q^2\,e^{2iz},\, -e^{-2iz};\, q^2)_\infty.
			\end{equation}
	\end{itemize}
Then it is straightforward to verify the theta identities \eqref{E:theta-identity-1}  corresponds to given  $k=2$, $a_1=q$, $a_2=-q^2$, $C_1=(q^2;\, q^2)_\infty^2\vartheta_4$, $C_2=q^{\frac12}(q^2;\, q^2)_\infty^2\vartheta_2$, then $b=-q$  and $C=(q^2;\, q^2)_\infty^2\vartheta_3$ from Theorem \ref{T:Kernel-I}.  Similarly, the identity \eqref{E:theta-identity-2} corresponds to given $k=2$, $a_1=q^2$, $a_2=-q^2$ and $q$ replaced by $q^2$,
\[
		C_1=[q^{1/2}(q^2;\, q^2)_\infty^2]^2(q^2\,e^{iy},\,  q^2\,  e^{-iy};\, q^2)_\infty\,
			 (q^2/q^2\,e^{iy},\, q^2/q^2\, e^{-iy}\, q^2)_\infty,
	\]
	\[
		C_2=[(q^2;\, q^2)_\infty^2]^2 (-q^2\,e^{iy},\; -q^2\,e^{-iy};\, q^2)_\infty\, \big(q^2/(-q^2)\,e^{iy},\; q^2/(-q^2)\,e^{-iy};\, q^2\big)_\infty,
	\]
then $C=[(q^2;\, q^2)_\infty^2]^2$ and $b=q\, e^{iy}$. We omit the detailed verification.
%%%%%%%%%%%%%%%%%%%%%%%%%%%%%%%%%%%%%%%%%%%%%%%%%%%%%%%%%%%%%%%%%%%%%%%%%5
%\end{example}

\section{Askey-Wilson type Five-value theorem}\label{S:unicity}

The above consideration allows us to obtain a variation of Nevanlinna's five values theorem for finite logarithmic order meromorphic functions. Nevanlinna showed in 1929 \cite[\S2.7]{Hayman1964} that if two arbitrary meromorphic functions share five values, that is, the pre-images  of the five points (ignoring their multiplicities) in $\mathbb{C}$ are equal, then the two functions must be identical. There has been numerous generalisations of this result, including those taking multiplicities into account. Halburd and Korhonen showed that there is a natural analogue of the five-value theorem for two \textit{finite order} meromorphic functions for the simple difference operator $\Delta f$ in \cite{HK-2}. We show below that there is also a natural extension for the five-value theorem for two \textit{finite logarithmic order} meromorphic functions with respect to the $\mathrm{AW-}$operator. Our definition for two functions sharing a value in Askey-Wilson appears to be different in spirit from that given in \cite{HK-2}.
\begin{definition} Let $f$ and $g$ be two meromorphic functions with finite logarithmic orders. Let $a\in \hat{\mathbb{C}}$. We write $E_f(a)$ to be the inverse image of $a$ under $f$, that is, it is the subset of $\mathbb{C}$ where $f(x)=a$. Then we say that $f$ and $g$ \textit{share the $\mathrm{AW-}$value} $a$ if  $E_f(a)=E_g(a)$ except perhaps on the subset of $\mathbb{C}$ such that
	\begin{equation}
		\label{E:share_count-1}
		\tilde{n}_{\mathrm{AW}}(r,\, f=a)-\tilde{n}_{\mathrm{AW}}(r,\, g=a)=O(1).
	\end{equation}
We can write the above statement in the equivalent form
	\begin{equation}
		\label{E:sharing}
		\awNf{r}{f=a}-\awNf{r}{g=a}=O(\log r).
	\end{equation}
\end{definition}
We recall from the \S\ref{S:main-results-II} that the definition (\ref{E:share_count-1}) means that 
	\[
		\sum_{|x|<r}\big(h_f(x)-k_f(x)\big)-\sum_{|x|<r}\big(h_g(x)-k_g(x)\big)=O(1),
	\]
where $k=\min\{h,\, k^\prime\}$. We note that the definition entails that two finite logarithmic order meromorphic functions share a $\mathrm{AW-}a$ value could be very different from two meromorphic functions share the value $a$ in the classical sense. If the pre-images of $a\in\mathbb{C}$  under $f$ and $g$ lie on a sequence defined by (\ref{E:AW-except}), then  $f$ and $g$ share $\mathrm{AW-}a$. On the other hand, there are many ways for which the $h_f(x)-k_f(x)$ and $h_g(x)-k_g(x)$ can behave that would lead to the upper bound stipulated in (\ref{E:sharing}). 
\begin{theorem}
\label{T:unicity} Let $f_i(z),\, i=1,\, 2$ be non-constant transcendental meromorphic functions of finite logarithmic orders (\ref{E:growth-rate-2}) such that $\mathcal{D}_qf_i\not\equiv 0\ (i=1,\, 2)$. Suppose that $f_i(z),\, i=1,\, 2$  share five distinct $\mathrm{AW-}$points $a_\nu,\, \nu=1,\, \cdots,\, 5$. Then $f_1\equiv f_2$.
\end{theorem}

\begin{proof} We denote $\sigma_{\log}$ to be the maximum of the logarithmic orders of $f_1,\, f_2$. We suppose on the contrary that the functions $f_1,\,f_2$ are not identically the same. According to the assumption, we shall assume that $E_{f_1}(a_\nu)\equiv E_{f_2}(a_\nu)$ except perhaps on those $x$ for which the (\ref{E:sharing}) holds with $\nu=1,\, \cdots,\, 5$.  Hence
	\[
	N_{{12},\,\nu}(r):=\awNf{r}{\frac{1}{f_1-a_\nu}}=\awNf{r}{\frac{1}{f_2-a_\nu}}+O(\log r),\quad \nu=1,\, \cdots,\, 5.
	\]

Choosing $p=5$ in (\ref{E:2nd-main-2}) yields
	\[
	4\, T(r,\, f_i)\le \awNf{r}{f_i}+\sum_{\nu=1}^5 N_{12,\, \nu}(r,\, f_i)+O(\log^{\sigma_{\log}-1+\varepsilon} r)+O(\log r),\quad i=1,\, 2,
	\]
and hence,
\be
	\label{E:share-2}
	3\, T(r,\, f_i)\le \sum_{\nu=1}^5 N_{12,\,\nu}(r)+O(\log^{\sigma_{\log}-1+\varepsilon} r)+O(\log r),\quad i=1,\, 2.
\ee
Since $f_1,\,f_2$ are not identical, so applications of \eqref{E:share-2} give
	\[
\begin{split}
	T\Big(r,\,(f_1-f_2)^{-1}\Big)&=T(r,\, f_1-f_2)+O(1)\\
	&\le T(r,\, f_1)+T(r,\, f_2)+O(1) \\
	&\le \frac23 \sum_{\nu=1}^5 N_{12,\,\nu}(r)+O(\log^{\sigma_{\log}-1+\varepsilon} r)+O(\log r).
\end{split}
	\]
Thus except for those $x$ for which the (\ref{E:sharing}) may hold with $a_\nu-$points ($\nu=1,\, \cdots, 5$), the zeros of $f_1-f_2$ satisfy
	\[
\begin{split}
	 \sum_{\nu=1}^5 N_{12,\,\nu}(r)&\le\awNf{r}{\frac{1}{f_1-f_2}}\le T\Big(r,\,(f_1-f_2)^{-1}\Big)\\
%	&=T(r,\, f_1-f_2)+O(1)\\
	&\le   \frac23\sum_{\nu=1}^5 N_{12,\,\nu}(r)+O(\log^{\sigma_{\log}-1+\varepsilon} r)+O(\log r).
\end{split}
	\]
Thus,
\begin{equation*}
	 \sum_{\nu=1}^5 N_{12,\,\nu}(r)=O(\log^{\sigma_{\log}-1+\varepsilon} r)+O(\log r).
\end{equation*}
Substitute the above equation into \eqref{E:share-2} yields
	\[
		T(r,\, f_i)= O(\log^{\sigma_{\log}-1+\varepsilon} r)+O(\log r)
	\]
which is impossible unless $f_1,\, f_2$ are rational functions. This is a contradiction.
\end{proof}

%\textcolor{blue}{\textbf{Does the number $4$ in the unicity theorem above the best possible?}}

\section{\text{Applications to difference equations}}
\label{E:equation}

Let 
	\begin{equation}
		\label{E:AW-poly}
			\begin{split}
			P_n(x)&=p_n(x;\, a,\, b,\, c,\, d\,|q):\\
				& =\frac{(ab,\, ac,\, ad,\,;q)_n}{a^n}\, {}_4\phi_3
			\Big(
				\begin{array}{cccc}
				q^{-n},& a\,b\,c\,d\,q^{n-1}, & a\,e^{i\theta}, & a\,e^{-i\theta}\\
				  \ \ \ a\,b,& a\,c, & a\,d\ \ \ 
				 \end{array} \Big|\, q;\, q\Big)
			\end{split}
	\end{equation}
be the $n-$th \textit{Askey-Wilson} polynomial, where we recall that $x=\cos\theta$. It is known from \cite[Theorem 2.2]{Askey:Wilson1985} that when $-1<q<1$, $a,\,\, b,\, c,\, d$ are real or appear in conjugate pairs, and that $\max\{|a|,\, |b|,\, |c|,\, |d|\}<1$, then the $\mathrm{AW}-$polynomials are orthogonal on $[-1,\, 1]$ with respect to the \textit{weight} function
\begin{equation}
	\label{E:weight}
		\begin{split}
			&\omega(x)=\omega(x;\, a,\, b,\, c,\, d\, |q): =\frac{{w}(x;\, a,\, b,\, c,\, d\, |q)}{\sqrt{1-x^2}}\\
			&=\frac{(e^{2i\theta},\, e^{-2i\theta};\, q)_\infty}{(a\,e^{i\theta},\, a\,e^{-i\theta};\, q)_\infty(b\,e^{i\theta},\, b\,e^{-i\theta};\, q)_\infty(c\,e^{i\theta},\, c\,e^{-i\theta};\, q)_\infty (d\,e^{i\theta},\, d\,e^{-i\theta};\, q)_\infty\sin\theta}.
		\end{split}
\end{equation}
Let 
\smallskip
	\begin{equation}
			\label{E:shifted-weight}
			\tilde{\omega}(x):=\omega(x;\, a\,q^{\frac12},\, b\,q^{\frac12},\, c\,q^{\frac12},\,  d\,q^{\frac12}\,|q),
		\end{equation}
be a shifted-weight function. Askey and Wilson showed  (see \cite[(5.16)]{Askey:Wilson1985}) that the $\mathrm{AW}-$polynomials are also eigen-solutions to the (self-adjoint) second-order difference equation
	\begin{equation}
		\label{E:AW-eqn}
			(1-q)^2\,\mathcal{D}_q\big[\tilde{\omega}(x)\,\mathcal{D}_q\,  y(x)\big]
			+\lambda_n \omega(x)\,y(x)=0
		\end{equation}
where $y(x)=p_n(x;\, a,\, b,\, c,\, d\,|q)$
		
		\[
			\lambda_n=4q^{-n+1}(1-q^n)(1-a\,b\,c\,d\,q^{n-1}),
		\]
are corresponding \textit{eigenvalues}. 
 We consider a self-adjoint type equation with a more general entire coefficient. Entire functions of zero-order have particularly simple Hadamard factorization. Littlewood \cite[\S14]{Littlewood1907} \footnote{The authors are grateful for the referee who pointed out this information.} gave a detailed but lengthy analysis of asymptotic behaviour of  $q-$infinite products. We instead derive a less accurate estimate but with shorter argument based on Bergweiler and Hayman \cite[Lemma 3]{Bergweiler:Hayman2003} for the Jacobi theta function $\vartheta_4(z;\, q)$ (see \cite[p. 469]{Whittaker:Watson1927})
 in the punctured plane $\mathbb{C}^\ast=\mathbb{C}\backslash \{0\}$ away from the zeros when considered as the function of $z$.   If $x=\cos\theta=(z+1/z)/2$, then in our notation, their theta function \cite[(4.5)]{Bergweiler:Hayman2003} is represented as
	\[
		(q^2,\, qe^{i\theta},\, qe^{-i\theta}; q^2)_\infty.
	\]
We modify their argument to suit the notation we use for our infinite products which allow for an extra non-zero parameter $a$. Unlike the restriction that $q$ is required to be real and $-1<q<1$ in \cite[Theorem 2.2]{Askey:Wilson1985}, we allow our $q$ to be complex.

\begin{lemma} \label{L:asymptotic} Suppose $a\in\mathbb{C}\backslash\{0\}$, $x=\cos\theta=\frac12(z+z^{-1})$, and
	
%\smallskip
	\begin{equation*}
%		\label{E:product}
			f(x)=(ae^{i\theta},\, ae^{-i\theta}; q)_\infty.
	\end{equation*}
Let  $|z|>\max\{|aq^{-\frac12}|,\, |a^{-1}q^\frac12|\}$ , $\nu\in\mathbb{N}$ and $\tau\in [0,\, 1)$ be two numbers (both depend on $z$) that satisfy

	\begin{equation}
		|az|=|q|^{\frac32-\tau-\nu}.
	\end{equation}
\noindent Then we have,
	\begin{equation}
		\label{E:theta-asymptotic}
%			\begin{split}
		\log |f(x)|=\frac{(\log |az|)^2}{-2\log |q|}+\frac12\log (|az|) +\log |1-a q^{\nu-1}z|
+O(1)
%			\end{split}
	\end{equation}
as $x\to\infty$ and hence $z\to\infty$. 
\end{lemma}

\begin{proof} We  write
	\begin{equation}
		\begin{split}
		\label{E:log-sum}
			\log \big|(ae^{i\theta},\, ae^{-i\theta}; q)_\infty\big|&=
			\sum_{k=1}^\infty \log\big|(1-aq^{k-1}z)(1-aq^{k-1}/z)\big|\\
			&=S_1+S_2+S_3 +\log|1-a q^{v-1}z|\\
		\end{split}
	\end{equation}
where
	\[
		S_1=\sum_{k=1}^{\nu-1}  \log\big|(1-aq^{k-1}z)\big|,\qquad
		S_2=\sum_{k=\nu+1}^\infty  \log\big|(1-aq^{k-1}z)\big|,
	\]
and
	\[
		S_3=\sum_{k=1}^{\infty}  \log\big|(1-aq^{k-1}/z)\big|.
	\]
We first consider
	\begin{equation}
		\begin{split}
			\label{E:split-1}
		S_1 &=\sum_{k=1}^{\nu-1} \log|aq^{k-1}z| + \sum_{k=1}^{\nu-1} \log\Big|1-\frac{1}{aq^{k-1}z}\Big|\\
			&= (\nu-1)\log |az|+(\nu-1)(\nu/2-1) \log|q| +\sum_{k=1}^{\nu-1} \log\Big|1-\frac{1}{aq^{k-1}z}\Big|.
		\end{split}
	\end{equation} 
	
Since $k\le \nu-1$, so that $\nu-k\ge 1$, we have
		\[
			\big|1/({aq^{k-1}z})\big|=|q|^{-\frac32+\nu+\tau+(1-k)}\le |q|^{\nu-k-1/2}\le |q|^{\frac12}<1.
		\]
So\footnote{Since $|\log (1+z)|<\log\frac{1}{1-|z|}<\frac{|z|}{1-|z|},\ (0<|z|<1)$, see e.g., \cite[p. 68, (4.1.34) and (4.1.38)]{Abramowitz:Stegun1964}.}
		\begin{equation}
			\label{E:log-1}
			\Bigg|\log\Big|1-\frac{1}{aq^{k-1}z}\Big|\Bigg|<\log \frac{1}{1-|q|^{\nu-k-\frac12}}
			<\frac{|q|^{\nu-k-\frac12}}{1-|q|^{\nu-k-\frac12}}
			<\frac{|q|^{\nu-k}}{|q|^\frac12(1-|q|^\frac12)}.
		\end{equation}
Hence
		\begin{equation}
			\label{E:log-2}
				\begin{split}
			\sum_{k=1}^{\nu-1}\Bigg|\log\Big|1-\frac{1}{aq^{k-1}z}\Big|\Bigg|
			&<\sum_{k=1}^{\nu-1}\frac{|q|^{\nu-k}}{|q|^\frac12(1-|q|^\frac12)}
			<\frac{1}{|q|^\frac12(1-|q|^\frac12)}\sum_{j=1}^\infty{|q|^j}\\
			&=\frac{|q|^\frac12}{(1-|q|^\frac12)(1-|q|)}.
				\end{split}
		\end{equation}
We deduce that
		\begin{equation}
			\label{E:log-0}
			\big|S_1-(\nu-1)\log |az|+(\nu-1)(\nu/2-1) \log|q|\big|<\frac{|q|^\frac12}{(1-|q|^\frac12)(1-|q|)}.
		\end{equation}
Now let us compute $S_2$. Since $k\ge \nu+1$ and $\tau\in [0,\, 1)$, so 
	\[
		|aq^{k-1}z|=|q|^{\frac12+k-\nu-\tau}
< |q|^{k-\nu-\frac12} \le|q|^{1/2}<1.
	\]
Hence
	\begin{equation}
%			\label{E:log-3}
		\Big|\log\big|1-aq^{k-1}z\big|\Big|\le \log\frac{1}{1-|q|^{k-\nu-\frac12}}
		\le \frac{|q|^{k-\nu-\frac12}}{1-|q|^{k-\nu-\frac12}}
		<\frac{|q|^{k-\nu}}{|q|^\frac12(1-|q|^{\frac12})}.
	\end{equation}
We deduce 
	\begin{equation}
			\label{E:log-3}
		\begin{split}
		|S_2|&\le \frac{1}{|q|^\frac12(1-|q|^{1/2})}\sum_{k=\nu+1}^\infty |q|^{k-\nu}=
		 \frac{1}{|q|^\frac12(1-|q|^{1/2})}\Big(\sum_{j=1}^\infty |q|^{j}\Big)\\
		&= \frac{|q|^\frac12}{(1-|q|)(1-|q|^{1/2})}.
		\end{split}
	\end{equation}
It remains to estimate $S_3$. According to our assumption $|z|>|aq^{-\frac12}|$ that 		
		\[
			|aq^{k-1}/z|< |q|^{k-\frac12}\le |q|^\frac12<1
		\]
so that, we can invoke a similar argument used to estimate (\ref{E:log-1}--\ref{E:log-3}) to derive
	\begin{equation}
			\label{E:log-4}
		\begin{split}
		|S_3| &\le \sum_{k=1}^\infty\Bigg|\log\Big| 1-\frac{aq^{k-1}}{z}\Big|\Bigg|
		\le  \sum_{k=1}^\infty\log \frac{1}{1-\big|aq^{k-1}/z\big|}
		 < \sum_{k=1}^\infty\log \frac{1}{1-|q|^{k-\frac12}}\\
		 &< \sum_{k=1}^\infty \frac{|q|^{k-\frac12}}{1-|q|^{k-\frac12}}< \frac{1}{1-|q|^\frac12}
		\sum_{k=1}^\infty |q|^{k-\frac12}=\frac{|q|^\frac12}{(1-|q|)(1-|q|^\frac12)}.
		\end{split}
	\end{equation}
 Combining the estimates (\ref {E:split-1}), (\ref{E:log-0}), (\ref{E:log-3}) and (\ref{E:log-4}) yields the estimate (\ref{E:theta-asymptotic}) as required. 
\end{proof}

%\begin{remark} We note that the $\log\mathrm{meas}\,(E)\le \frac{\pi^2}{6}\log\frac1q$.
%\end{remark}
%\smallskip

\begin{lemma}
	\label{L:weight-estimate}
		Let $\omega (x)$ and $\tilde{\omega}(x)$ be as defined in $\mathrm{(\ref{E:weight})}$ and $\mathrm{(\ref{E:shifted-weight})}$ respectively. Then we have
			\begin{equation}
				\label{E:weight-estimate}
					m\Big(r,\, \frac{\tilde{\omega}}{\omega}\Big)=O(\log r),\quad |x|=r.
			\end{equation}
\end{lemma}

\medskip

\begin{proof} It is elementary that
	\begin{equation}
		\label{E:weight-ratio}
		\begin{split}
		m\Big(r,\, \frac{\tilde{\omega}}{\omega}\Big)
		&\le m\Big(r,\, \frac{(aq^{\frac12}e^{i\theta},\, aq^{\frac12}e^{-i\theta}; q)_\infty}{(ae^{i\theta},\, ae^{-i\theta}; q)_\infty}\Big)
		+ m\Big(r,\, \frac{(bq^{\frac12}e^{i\theta},\, bq^{\frac12}e^{-i\theta}; q)_\infty}{(be^{i\theta},\, be^{-i\theta}; q)_\infty}\Big)\\
		&\quad+
		 m\Big(r,\, \frac{(cq^{\frac12}e^{i\theta},\, cq^{\frac12}e^{-i\theta}; q)_\infty}{(ce^{i\theta},\, ce^{-i\theta}; q)_\infty}\Big)
		+ m\Big(r,\, \frac{(dq^{\frac12}e^{i\theta},\, dq^{\frac12}e^{-i\theta}; q)_\infty}{(de^{i\theta},\, de^{-i\theta}; q)_\infty}\Big).
		\end{split}
	\end{equation}
Without loss of generality, we consider $m\big(r,\, {(aq^{\frac12}z,\, aq^{\frac12}/z; q)_\infty}/{(az,\, a/z; q)_\infty}\big)$. 
Since $|z|=2|x|+o(1)$ as $|x|=r\to\infty$, we have by
Lemma \ref{L:asymptotic}

	\begin{equation}
		\log\Big| \frac{(aq^{\frac12}z,\, aq^{\frac12}/z; q)_\infty}{(az,\, a/z; q)_\infty}\Big|
		=-\log |1-a q^{\nu_1-1}z|+\log |1-a q^{\nu_2-\frac12}z| +O(\log r),
	\end{equation}
with $|az|=|q|^{\frac32-\tau_1-\nu_1},|aq^{\frac12}z|=|q|^{\frac32-\tau_2-\nu_2}$.
By Lemma 4.1 with $\alpha=\frac12$, we have 
\begin{equation}
		\log |1-a q^{\nu_1-1}z|=O\Big(|a q^{\nu_1-1}z|^{\frac12}+\Big|\frac{z}{z-a^{-1} q^{1-\nu_1}}\Big|^{\frac12}\Big)=O\Big(\frac{r^{\frac12}}{|z-a^{-1} q^{1-\nu_1}|^{\frac12}}\Big)+O(1).
	\end{equation}
On the other hand,
\begin{equation}
|x-(a^{-1}q^{1-\nu_1}+a q^{\nu_1-1})/2|=\frac12|z-a^{-1} q^{1-\nu_1}||1-aq^{\nu_1-1}z^{-1}|=O\big(|z-a^{-1} q^{1-\nu_1}|\big).
\end{equation}
Then 
\begin{equation}
		\log |1-a q^{\nu_1-1}z|=O\Big(\frac{r^{\frac12}}{|x-(a^{-1}q^{1-\nu_1}+a q^{\nu_1-1})/2|^{\frac12}}\Big)+O(1).
	\end{equation}
Similarly we have 
\begin{equation}
		\log |1-a q^{\nu_2-\frac12}z|=O\Big(\frac{r^{\frac12}}{|x-(a^{-1}q^{\frac12-\nu_2}+a q^{\nu_2-\frac12})/2|^{\frac12}}\Big)+O(1).
	\end{equation}
Substitute (12.22) and (12.23) into (12.19) yields
\begin{equation}
	\begin{split}
		&\log\Big| \frac{(aq^{\frac12}z,\, aq^{\frac12}/z; q)_\infty}{(az,\, a/z; q)_\infty}\Big|\\
		&=O\Big(\frac{1}{|x-(a^{-1}q^{1-\nu_1}+a q^{\nu_1-1})/2|^{\frac12}}+\frac{1}{|x-(a^{-1}q^{\frac12-\nu_2}+a q^{\nu_2-\frac12})/2|^{\frac12}}\Big)\, r^{\frac12} +O(\log r),
		\end{split}
	\end{equation}
	
Let $r$ be large enough and fixed, $x=re^{i\phi}$. We note that when $\phi$ varies in $[0,\, 2\pi]$, $|z|$ makes a corresponding small change with $|z|=2r+o(1)$. This may result in the change of integers $\nu_1$ and $\nu_2$ but each of them can assume at most two consecutive integer values. Thus
\begin{equation}
		m\Big(r,\, \frac{(aq^{\frac12}z,\, aq^{\frac12}/z; q)_\infty}{(az,\, a/z; q)_\infty}\Big)
		=O\Big(\sum_{j=1}^4 \int_0^{2\pi} \frac{r^{\frac12}}{|re^{i\phi}-w_j|^\frac12}\,d\phi\Big)+O(\log r),
	\end{equation}
where $w_{1,\,2}=(a^{-1}q^{1-\nu_{1}}+a q^{\nu_{1}-1})/2$, $w_{3,\, 4}=(a^{-1}q^{\frac12-\nu_{2}}+a q^{\nu_{2}-\frac12})/2$. 
Hence, we deduce from  \eqref{E:lemma-3} that
	\begin{equation}
		m\Big(r,\, \frac{(aq^{\frac12}z,\, aq^{\frac12}/z; q)_\infty}{(az,\, a/z; q)_\infty}\Big)
		=O(\log r).
	\end{equation}
Similarly we can repeat the above argument to the remaining three terms in (\ref{E:weight-ratio}). This completes the proof.
\end{proof}

\begin{theorem} \label{T:equation} Let $A(x)$ be an entire function of finite logarithmic order $\sigma_{\log}(A)> 1$. Suppose that $f$ is an  entire solution to the second-order difference equation
\smallskip

	\begin{equation}
			\mathcal{D}_q\big[\tilde{\omega}(x)\,\mathcal{D}_q\,  y(x)\big]
			+\omega(x) A(x)\,y(x)=0,
		\end{equation}
where the $\omega$ and $\tilde{\omega}$ are defined in \eqref{E:weight} and \eqref{E:shifted-weight} respectively. 
Then $\sigma_{\log}(f)\ge \sigma_{\log}(A)+1$.
\end{theorem}

\begin{proof}  Let
	\[
		F(x):= \tilde{\omega}(x) \mathcal{D}_qf (x).
	\]
We deduce from Theorem \ref{T:simple-order} that $\sigma_{\log} (F)\le \max\{\sigma_{\log}({\omega}),\, \sigma_{\log}(\mathcal{D}_qf)\}\leq \max\{2,\, \sigma_{\log}(f)\}$. Then by the Theorem \ref{T:log-lemma}, for each $\varepsilon>0$,
	\begin{equation}
		\begin{split}
			m (r,\, A) & = m\Big(r, \, \frac{\mathcal{D}_q\big[\tilde{\omega}(x)\,\mathcal{D}_q\,  f(x)\big]}{\omega\, f}\Big)\\
			&\le m\Big(r,\, \frac{\mathcal{D}_q F (x)}{F(x)}\Big)
				+m\Big(r,\, \frac{\tilde{\omega}}{\omega}\Big)
				+m\Big(r,\, \frac{\mathcal{D}_q f (x)}{f(x)}\Big)\\
			& = O\big((\log r)^{\sigma_{\log}(F)-1+\varepsilon}  \big)
				+O(\log r) +O\big((\log r)^{\sigma_{\log}(f)-1+\varepsilon}\big)\\
			&= O\big((\log r)^{\max\{1,\, \sigma_{\log}(f)-1\}+\varepsilon}  \big).
		\end{split}
	\end{equation}
Since $\sigma_{\log}(A)>1$ and $\varepsilon>0$ is arbitrary, so we deduce the desired result.
\end{proof}

We further define
	\smallskip
	\begin{equation}
			\label{E:shifted-weight-k}
			\tilde{\omega}_k(x):=\omega(x;\, a\,q^{\frac{k}{2}},\, b\,q^{\frac{k}{2}},\, c\,q^{\frac{k}{2}},\,  d\,q^{\frac{k}{2}}\,|q),
		\end{equation}
to be the $k-$shifted weight function of (\ref{E:weight}), where $k\ge 1$ and $\tilde{\omega}_0(x)=\omega(x;\, a,\, b,\, c,\, d\, |q)$.  Then Askey and Wilson \cite[(5.15)]{Askey:Wilson1985} derived a \textit{Rodrigues-type} formula:
	\begin{equation}
		\label{E:rodrigues}
			\begin{split}
			(\mathcal{D}_q)^k & \big[ \omega(x;\, a\,q^{\frac{k}{2}},\, b\,q^{\frac{k}{2}},\, c\,q^{\frac{k}{2}},\,  d\,q^{\frac{k}{2}}\,|q)\big]\\
			&=\Big(\frac{q-1}{2}\Big)^{-n} q^{-\frac12n(n-1)} \omega(x;\, a\,q^{\frac12},\, b\,q^{\frac12},\, c\,q^{\frac12},\,  d\,q^{\frac12}\,|q)
			p_n(x;\, a,\, b,\, c,\, d\,|q)
			\end{split}
	\end{equation}
which may be regarded as a higher order difference equation. We can apply a similar technique used in the last theorem to obtain the following theorem whose proof is omitted. See also \cite[Theorem 9.2]{Chiang:Feng2008}
 
\begin{theorem} Let $k\ge 1, A_j(x),\ j=0,\, 1,\, \cdots k-1$ be an entire functions such that $\sigma_{\log}(A_0)>\sigma_{\log}(A_j)\ge 1, \  j=1,\, \cdots, k-1$. Suppose that $f$ is an entire solution to the $k-$th order difference equation
\smallskip

	\begin{equation}
		\mathcal{D}^{(k)}_q  y(x)+ A_{k-1}\mathcal{D}^{(k-1)}_q\, y(x) + \cdots +
			+A_1\,\mathcal{D}_q\,  y(x) +A_0(x) \,y(x)=0.
	\end{equation}
Then $\sigma_{\log}(f)\ge \sigma_{\log}(A_0)+1$.
\end{theorem}

\section{Concluding remarks}

We have shown in this paper that the $\mathrm{AW-}$operator naturally induces a version of difference value distribution theory on meromorphic functions of finite logarithmic order of growth. Although the finite  logarithmic order growth appears to be restrictive, it turns out that this  class of functions contains a large family of meromorphic functions, including the Jacobi theta functions and theta-like functions and also many $q-$series type special functions.

In particular, a Picard-type theorem based on the $\mathrm{AW-}$operator is derived. For any complex $a$, instead of the classical Nevanlinna theory in which the Nevanlinna deficiency $\delta(a)$ plays an important role, we have shown that it is the  $\Theta_{\mathrm{AW}}(a)$  which corresponds to what we used to call the {ramification index} that plays the crucial role in our $\mathrm{AW-}$Nevanlinna theory. It appears to be a proper index to consider when dealing with function theoretic problems on finite differences in general and $\mathrm{AW-}$difference operator in particular. As a result, we have called the $\Theta_{\mathrm{AW}}(a)$, where $0\le \Theta_{\mathrm{AW}}(a)\le 1$, the $\mathrm{AW-}$deficiency and showed that $\sum_{a\in\mathbb{C}}\Theta_{\mathrm{AW}}(a)\le 2$ in this paper. Our new Picard theorem says that if a slow-growing (finite logarithmic order) meromorphic function $f$ has three such $\mathrm{AW-}$deficient values, then $f$ belongs to $\ker\D$. Special cases of an $a-$point being a $\mathrm{AW-}$deficient value of $f$ include when the
pre-image of an $a-$point lies on an infinite sequence of the type (\ref{E:AW-except}). Thus, although the equation $f(x)=a$ has infinitely many solutions, our theory suggests us to interpret these $a-$points as if they are not present in the sense of Askey-Wilson. We have also given an alternative derivation of functions that lie in the $\ker\D$. Unlike the kernel of conventional differential operator, the $\ker\D$ is non-trivial. As a consequence, we have derived a number of relationships exist amongst families of  $q-$infinite products.

 Although one can write down an infinite convergent series given by the $\mathrm{AW-}$Taylor-expansion (\ref{E:taylor}) in terms of the $\mathrm{AW-}$basis, little is is known about the value distribution of those functions. One such example is given by Koelink and Stokman in \cite{Koelink:Stokman2001a} where they constructed a transcendental function solution to (\ref{E:AW-eqn}) which is linearly independent to the Askey-Wilson polynomials (\ref{E:AW-poly}). This transcendental function
were further studied in \cite{Koelink:Stokman2001b}. But we still do not know its logarithmic order. Needless to say that much less is known about the value distribution properties of other transcendental meromorphic functions associated with the $\mathrm{AW-}$operator. Our $\mathrm{AW-}$Nevanlinna theory allows us to understand a little more. For example, the generating function $H(x)$ (\ref{E:generating}) for $q-$ultraspherical polynomials found by Rogers mentioned earlier has zero-sequence and pole sequences as described by  (\ref{E:Rogers-zeros}) and (\ref{E:Rogers-poles}) respectively. However, it has $\Theta_{\mathrm{AW}}(0)=1$ and $\Theta_{\mathrm{AW}}(\infty)=1$ under our interpretation. Thus the $H(x)$ can be regarded as $\mathrm{AW-}$zero-scarce and pole-scarce. On the other hand, the $H(x)$ is not in the form (\ref{E:q-constant}) described by the Theorem \ref{T:ker-mero}, so it
does not belong to the $\mathcal{D}_q$. Hence the $H(x)$ must assume all $a\not= 0,\, \infty$ infinitely often in the sense of Askey-Wilson.

We recall that a function is called a \textit{polynomial} if the function is annihilated after repeated application of conventional  differentiation a finite number of times. Thus, those functions that are annihilated after a differentiation are called \textit{constants}. If we replace the differential operator by the $\mathrm{AW-}$operator, then the Theorem \ref{T:ker-mero} shows that apart from the conventional constants, there are also \textit{constants} (given by (\ref{E:q-constant})) with respect to the  $\mathrm{AW-}$operator. So it is natural to ask what are the \textit{polynomials} and \textit{transcendental} with respect to the  $\mathrm{AW-}$operator. Since even the class of $\mathrm{AW-}$\textit{constants} consists of a rich collection of conventional transcendental meromorphic functions, thus it can be anticipated that these \textit{polynomials} should be rich and worth exploration.

The classical Picard theorem and Nevanlinna theory are about a particular way of counting zeros/poles and their multiplicities about a meromorphic function with respect to the basis $\{x^n\}\  (-\infty\le n  <+\infty)$ which is natural with respect to the derivative operator. However, the natural basis for a difference operator is not the usual $\{x^n\}$. It is known that some natural bases for difference operator $\Delta f(x)=f(x+c)-f(c)$ and the  $\mathrm{AW-}$operator $\D f$ are, respectively,
	\begin{enumerate}	
		\item \textrm{The Netwon basis}: $p_n(x)=x(x+1)\cdots (x+n-1)$
		\bigskip
		\item the $\mathrm{AW-}$basis: $\phi(x;\, a)_n=(ae^{i\theta},\, ae^{-i\theta};\, q)_n$,
	\end{enumerate}
when $n\ge 0$. However, when defined in an appropriate manner, they can be extended to the full-range $(-\infty< n  <+\infty)$. 
Thus it may be more appropriate to establish the various Nevanlinna theories for difference operators with respect to their natural interpolatory bases, and therefore this includes finding their appropriate \textit{residue calculus}.

\appendix
%\section{A summary of Nevanlinna Theory}
\section{Proof of theorem \ref{T:meromorphic} } \label{S:meromorphic}
Although the Askey-Wilson operator is defined on basic hypergeometric polynomials in their original memoir \cite{Askey:Wilson1985}, it follows from a terminating sum of (\ref{E:taylor}), that one can write $x^n$ explicitly in terms of $\{\phi_n(\cos\theta;\, a)\}$, together with (\ref{E:diff-basis}) show that the Askey-Wilson operator will reduce the a degree $n$ polynomial $f(x)$ to degree $n-1$.  Alternatively, one can verify this directly:
\begin{align*}
%\label{E:poly-basis}
	&[(q^{1/2}-q^{-1/2})i\sin\theta]\, \mathcal{D}_qx^k =
	(q^{1/2}z-q^{-1/2}z^{-1})^k-(q^{-1/2}z-q^{1/2}z^{-1})^k\\
%	&=\sum_{j=0}^k\binom{k}{j}q^{-(k-2j)/2}z^{-(k-2j)}-
%	\sum_{j=0}^k\binom{k}{j}q^{(k-2j)/2}z^{-(k-2j)}\notag\\
%	&\stackrel{j^\prime=k-j}{=} \sum_{j=0}^k\binom{k}{j}q^{-(k-2j)/2}z^{-(k-2j)}
%	+\sum_{j^\prime=0}^k\binom{k}{j^\prime}q^{-(k-2j^\prime)/2}z^{-(2j^\prime-k)}\notag\\	
%	&=\sum_{j=0}^k\binom{k}{j}q^{-(k-2j)/2}\big(z^{(k-2j)}-z^{-(k-2j)}\big)\notag\\
	&=\sum_{j=0}^k\binom{k}{j}q^{-(k-2j)/2}(2i)\sin(k-2j)\theta.\notag
\end{align*}
Hence
\begin{equation}
\label{E:poly-basis}
	\mathcal{D}_qx^k %= \sum_{j=0}^k\binom{k}{j}q^{-(k-2j)/2}\sin(k-2j)\theta(2i)/
%	\big[(q^{1/2}-q^{-1/2})(z-1/z)/2\big]\\
%	&= \sum_{j=0}^k\binom{k}{j}2q^{-(k-2j)/2}\sin(k-2j)\theta/
%	\big[(q^{1/2}-q^{-1/2})\sin\theta\big]\notag\\
	=\sum_{j=0}^k\binom{k}{j}\frac{2q^{-(k-2j)/2}}{q^{1/2}-q^{-1/2}}\,
	U_{k-2j}(x)\notag,
%	&=\sum_{j=0}^k\binom{k}{j}\frac{2q^{-(k-2j)/2}}{q^{1/2}-q^{-1/2}}
%	U_{k-2j}(x)\notag,
\end{equation}
where $x=\cos\theta$ and $U_k$ is the \textit{Chebyshev polynomial of the second kind}, shows that $\mathcal{D}_qx^k$ is indeed a polynomial in $x$.  
 
Let $f(x)=\sum_{k=0}^\infty a_kx^k$ be entire so  $f_N(x):=\sum_{k=N}^\infty a_kx^k\rightarrow 0$ uniformly on any compact subset of $\mathbb{C}$ as $N\to +\infty$. Thus it is clear that both $\breve{f_N}(q^{1/2}z)$ and $\breve{f_N}(q^{-1/2}z)$, and hence $\mathcal{D}_q(\sum_{k=N}^\infty a_kx^k\big)\rightarrow 0$ uniformly on any compact subset of $\mathbb{C}$ as $N\to +\infty$. Thus $\mathcal{D}_qf=\mathcal{D}_q\big(\sum_{k=0}^\infty a_kx^k\big)$ is analytic and hence entire.
\medskip

We would like to extend the definition of $\mathcal{D}_q$ to meromorphic functions. To do so we first establish that given $f(x)$ entire, then so is the 
\begin{equation*}
%\label{E:average-operator}
		(\mathcal{A}_qf)(x)=\frac12\big[\breve{f}(q^\frac12z)+
		\breve{f}(q^{-\frac12}z)\big],
\end{equation*}
which is called the \textit{averaging operator} \cite[p. 301]{Ism2005}. In view of the argument to justify the analyticity of $\mathcal{D}_qf$ above, it suffices to show that $\mathcal{A}_qx^k$ is again a polynomial in $x$. Since
\begin{align}
	\mathcal{A}_qx^k
%&=
%\frac12\Big\{	x^k\big[(q^{1/2}z+q^{-1/2}/z)/2\big]+x^k\big[(q^{-1/2}z+q^{1/2}/z)/2\big]\Big\}\notag\\
	&=\frac12\Big\{
	(q^{1/2}z+q^{-1/2}/z)^k+(q^{-1/2}z+q^{1/2}/z)^k\Big\}\notag\\
%	&=\frac12\Big\{\sum_{j=0}^k\binom{k}{j}q^{j/2}z^jq^{-(k-j)/2}z^{-(k-j)}+
%	\sum_{j=0}^k\binom{k}{j}q^{j/2}z^jq^{-(k-j)/2}z^{-(k-j)}\Big\}\notag\\
%	&=\frac12\Big\{\sum_{j=0}^k\binom{k}{j}q^{-(k-2j)/2}\big(z^{(k-2j)}+z^{-(k-2j)}\big)\Big\}\notag\\
&=
	\sum_{j=0}^k\binom{k}{j}q^{-(k-2j)/2}\cos(k-2j)\theta=\sum_{j=0}^k\binom{k}{j}q^{-(k-2j)/2}T_{k-2j}(x),\notag
\end{align}
where the $T_k$ is the \textit{Chebyshev polynomial of the first kind}.
Thus $\mathcal{A}_qf(x)$ is again entire.

Moreover, one can check easily that 
\begin{equation*}
%\label{E:one-over}
	(\mathcal{D}_q1/f)(x)=\frac{-(\mathcal{D}_qf)(x)}{f(q^{1/2}z+
	q^{-1/2}/z)f(q^{-1/2}z+q^{1/2}/z)}
\end{equation*}
(We note that it tends to $-f^\prime(x)/f^2(x)$ as $q\to 1$.) We consider
\begin{align}
\label{E:product-1}
		&f\big(q^{1/2}z+q^{-1/2}/z)/2\big)
		f\big(q^{-1/2}z+q^{1/2}/z)/2\big)\\
		&=
		\Big(\sum_{k=0}^\infty\frac{a_k}{2^k}\big(q^{1/2}z+q^{-1/2}/z\big)^k\Big)
		\Big(\sum_{k^\prime=0}^\infty\frac{a_{k^\prime}}{2^{k^\prime}}\big(q^{-1/2}z+q^{1/2}/z\big)^{k^\prime}\Big)\notag\\
		&=\sum_{k=0}^\infty\sum_{k^\prime=0}^\infty
		\frac{a_ka_{k^\prime}}{2^{k+k^\prime}}\big(q^{1/2}z+q^{-1/2}/z\big)^k\big(q^{-1/2}z+q^{1/2}/z\big)^{k^\prime}.\notag
\end{align}

We only need to consider the cases when $k\not=k^\prime$. Suppose $k>k^\prime$. We note the following term in the (\ref{E:product-1}).
\begin{align}
\label{E:product-2}
	&(q^{1/2}z+q^{-1/2}z^{-1})^k(q^{-1/2}z+q^{1/2}z^{-1})^{k^\prime}
	+(q^{1/2}z+q^{-1/2}z^{-1})^{k^\prime}(q^{-1/2}z+q^{1/2}z^{-1})^k\\
	&=(q^{1/2}z+q^{-1/2}z^{-1})^{k^\prime}(q^{-1/2}z+q^{1/2}z^{-1})^{k^\prime}\notag\\
	&\qquad\times
	\big[(q^{1/2}z+q^{-1/2}z^{-1})^{k-k^\prime}+(q^{-1/2}z+q^{1/2}z^{-1})^{k-k^\prime}\big]\notag\\
	&=(x^2+q+q^{-1}-2)^{k^\prime}\sum_{j=0}^{k-k^\prime}\binom{k-k^\prime}{j}2q^{(2j-k-k^\prime)/2}\cos(k+k^\prime-2j)\theta\notag.
\end{align}
But $\cos(k+k^\prime-2j)\theta=T_{k+k^\prime-2j}(x)$ is a polynomial in $x$, so that 
	\[
	f\big(q^{1/2}z+q^{-1/2}/z)/2\big)
		f\big(q^{-1/2}z+q^{1/2}/z)/2\big)
	\]
is again entire. Thus $\mathcal{D}_q\big(1/f\big)$ is meromorphic.
We note that $f$  can be represented as a quotient $f=g/h$ where both $g$ and $h$ are entire. Since 
\begin{equation*}
%\label{E:one-over}
	(\mathcal{A}_q1/h)(x)=\frac{(\mathcal{A}_qh)(x)}{h(q^{1/2}z+
	q^{-1/2}/z)h(q^{-1/2}z+q^{1/2}/z)},
\end{equation*}
 it is now easy to see that $\mathcal{A}_q1/h$ is meromorphic. Then, we deduce from the quotient rule \cite[p. 301]{Ism2005}
	\[
%\label{E:product}
	\mathcal{D}_q(g/h)=(\mathcal{A}_qg)(\mathcal{D}_q1/h)+(\mathcal{A}_q1/h)(\mathcal{D}_qg)
	\]
%where
%\[
%		(\mathcal{A}_qf)(x)=\frac12\big[\breve{f}(q^\frac12z)+
%		\breve{f}(q^{-\frac12}z)\big]
%\]
that $\mathcal{D}_qf$ is a meromorphic function.

%\vskip1cm 
\medskip
\noindent\textbf{Acknowledgement} The authors would like to express their sincere thanks to the referees for his/her valuable comments that help to improve the main results of this paper.
The first author would like to acknowledge the hospitality that he received during his visits to the Academy of Mathematics and Systems Sciences, of Chinese Academy of Sciences. Similarly he would also like to thank George Andrews for his hospitality and useful discussions during a visit to Pennsylvania State University. Finally, the first author would also like to thank his colleague T. K. Lam for useful discussions. %\medskip

\bibliographystyle{amsplain}

\end{document}